\documentclass[10pt,french,leqno]{smfart}

\newcount\Cor\Cor=0
\def\newcor{\global\advance\Cor by 1
\par\bigskip\noindent (\romannumeral\Cor) - }

\usepackage[T1]{fontenc}
\usepackage{enumitem}
\usepackage{bbm}
\usepackage{bm}
\usepackage{amsmath}
\usepackage{amssymb,url,xspace}
\usepackage{mathtools}
\usepackage{mathrsfs,euscript,color}
\usepackage[all]{xy}
\usepackage{xcolor}
\usepackage{stmaryrd}

\newtheorem{theorem}{\textsc{Th\'eor\`eme}}[subsection]
\newtheorem{proposition}[theorem]{\textsc{Proposition}}
\newtheorem{lemma}[theorem]{\textsc{Lemme}}
\newtheorem{corollary}[theorem]{\textsc{Corollaire}}
\newtheorem{remark}[theorem]{\textsc{Remarque}}

\newtheorem{definition}[theorem]{\textsc{D\'efinition}}

\newtheorem{hypoth}[theorem]{\textsc{Hypoth\`ese}}
\newtheorem{hypoths}[theorem]{\textsc{Hypoth\`eses}}

\newtheorem{notation}[theorem]{\textsc{Notation}}

\catcode`\@=11
\catcode`\Ž=13
\catcode`\=13
\catcode`\ˆ=13
\catcode`\=13
\catcode`\=13
\catcode`\‰=13
\catcode`\™=13
\catcode`\"=13
\catcode`\•=13
\catcode`\ž=13
\catcode`\=13
\catcode`\'=13
\catcode`\Ë=13
\def Ž{\'e}
\def {\`e}
\def ˆ{\`a}
\def {\`u}
\def {\^e}
\def ‰{\^a}
\def ™{\^o}
\def "{\^\i}
\def •{\"\i}
\def ž{\^u}
\def {\c c}
\def '{\"e}
\def Ë{\`A}
\catcode`\:=13
\def :{~\string:}
\catcode`\;=13
\def ;{~\string;}
\catcode`\!=13
\def !{~\string!}
\def\cad{c'est-\`a-dire\ }

\def\bydef{\buildrel \mathrm{d\acute{e}f}\over{=}}

\def\ES#1{\EuScript{#1}}

\def\wt#1{\widetilde{#1}}
\def\bs#1{\boldsymbol{#1}}

\def\L{L}

\def\mbb#1{\mathbb{#1}}

\def\AA{\mathfrak A}

\def\NN{\mathfrak N}

\def\UU{\mathfrak U}

\def\brT_#1{[T]_{#1}}
\def\brTo_#1{[T_0]_{#1}}
\def\brTX_#1{[T-X]_{#1}}

\def\ptf{\,.}
\def\vg{\,,}
\def\vgq{\,,\quad}

\def\mathpvg{\!\!;}

\def\dd{\,{\mathrm d}}

\def\scrX{\mathscr{X}}
\def\scrY{\mathscr{Y}}
\def\scrV{\mathscr{V}}
\def\scrG{\mathscr{G}}

\def\scrM{\mathscr{M}}
\def\scrA{\mathscr{A}}
\def\scrH{\mathscr{H}}

\def\bsfrX{\bs{\mathfrak{X}}}
\def\bsfrY{\bs{\mathfrak{Y}}}

\def\HHFu{\mathcal{H}^H_{F,u}}

\def\FU{{_F\mathfrak{U}}}
\def\UF{\mathfrak{U}_F}

\def\NF{\mathfrak{N}_F}

\def\guill#1{{\guillemotleft\,#1\,\guillemotright}}

\title[Une mesure de Radon invariante sur les $F$-strates unipotentes]
{Une mesure de Radon invariante\\ sur les $F$-strates unipotentes}

\author{Bertrand Lemaire}
\email{Bertrand.Lemaire@univ-amu.fr}
\address{Institut de Math\'ematique de Marseille (I2M), Aix-Marseille Universit\'e (AMU), CNRS (UMR 7373), France}

\thanks{Ce texte est le fruit d'une collaboration avec Jean-Pierre Labesse: c'est ensemble que nous avons dŽmontrŽ la proposition \ref{proposition noyau} qui 
est le rŽsultat technique principal de l'article. Je remercie aussi Philippe Gille et George McNinch pour avoir eu la patience de rŽpondre ˆ mes questions na•ves. }

\begin{document}

\setcounter{tocdepth}{3}

\begin{abstract} 
Soient $F$ un corps commutatif localement compact non archimŽdien et $G$ un groupe rŽductif connexe dŽfini sur $F$. 
Ë tout ŽlŽment unipotent $u$ de $G(F)$, on a associŽ en \cite{L} une $F$-strate $\bsfrY_{F,u}$ qui est 
rŽunion (Žventuellement infinie) de $G(F)$-orbites unipotentes. On dŽfinit une mesure de Radon positive $G(F)$-invariante non nulle \guill{canonique} 
sur $\bsfrY_{F,u}$. Sous certaines hypothses supplŽmentaires, on en dŽduit la convergence de l'intŽgrale orbitale associŽe ˆ la $G(F)$-orbite 
de $u$. La construction, valable en toute caractŽristique, gŽnŽralise celle de Deligne-Ranga Rao \cite{RR} et s'applique aussi bien aux 
$F$-strates nilpotentes dans $\textrm{Lie}(G)(F)$. 
\end{abstract}

\begin{altabstract}
Let $F$ be a non-Archimedean locally compact field and $G$ a connected reductive group defined over $F$. To any unipotent element 
$u$ in $G(F)$, we have associated in \cite{L} an $F$-stratum $\bsfrY_{F,u}$ which is a (possibly infinite) union of unipotent $G(F)$-orbits. We define a \guill{canonical} 
non-zero positive $G(F)$-invariant Radon measure on $\bsfrY_{F,u}$. Under additional assumptions, we deduce the convergence of the orbital integral associated to the $G(F)$-orbit of $u$. 
The construction, valid in any characteristic, generalizes the one of Deligne-Ranga Rao \cite{RR} and also 
applies to nilpotent strata in $\textrm{Lie}(G)(F)$. 
\end{altabstract}

\subjclass{22E50, 22E35}

\keywords{instabilitŽ, co-caractre optimal, strate unipotente, mesure de Radon, distribution invariante, thŽorme de Deligne-Ranga Rao}

\maketitle

\tableofcontents

\section{Introduction} \label{introduction}

\subsection{Notations}
Soit $F$ un corps commutatif d'exposant caractŽristique $p\geq 1$. On fixe une cl™ture 
algŽbrique $\overline{F}$ de $F$ et on note $F^{\rm s\acute{e}p}$, resp. $F^{\rm rad}$, la cl™ture sŽparable, resp. radicielle, de $F$ dans $\overline{F}$. 
On pose $\Gamma_F={\rm Gal}(F^{\rm s\acute{e}p}/F)$. 
Soit $G$ un groupe rŽductif connexe dŽfini sur $F$. On note $\mathfrak{g}= {\rm Lie}(G)$ son algbre de Lie. 
Comme il est d'usage, on identifie $G$ ˆ $G(\overline{F})$ et $\mathfrak{g}$ ˆ $\mathfrak{g}(\overline{F})$.

Si $F$ est localement compact non archimŽdien (non discret), ce que l'on suppose jusqu'ˆ la fin cette introduction, on note $|\,|_F$ 
la valeur absolue normalisŽe sur $F$ et on munit $G(F)$ et $\mathfrak{g}(F)$ de la topologie dŽfinie par $F$, appelŽe $\mathrm{Top}_F$.

\subsection{Le thŽorme de Deligne-Ranga~Rao ($p=1$)}\label{le thŽorme de D-RR} 
On suppose dans cette sous-section que $F$ est une extension finie de $\mbb{Q}_p$. La convergence des intŽgrales orbitales sur $G(F)$ est un rŽsultat connu, dž ˆ Deligne et Ranga Rao \cite{RR}. Gr‰ce ˆ la dŽcomposition de 
Jordan (rationnelle puisque $p=1$), le rŽsultat se dŽduit facilement de la convergence des intŽgrales orbitales unipotentes.

Soit $u$ un ŽlŽment unipotent de $G(F)$. Le centralisateur $G^u(F)$ de $u$ dans $G(F)$ est unimodulaire et l'espace quotient $G^{u}(F)\backslash G(F)$ est muni d'une 
mesure de Radon positive $G(F)$-invariante non nulle $\dd \bar{g}$ (o $\bar{g} = G^u(F)g$ avec $g\in G(F)$). Il est prouvŽ dans loc.~cit. que cette mesure 
$\dd \bar{g}$, vue comme une mesure sur la $G(F)$-orbite 
$$\ES{O}_{F,u} = \{g^{-1}ug\,\vert \, g\in G(F)\}\subset G(F)$$ via l'homŽomorphisme 
$G^u(F)\backslash G(F) \rightarrow \ES{O}_{F,u} ,\, \bar{g} \mapsto g^{-1}ug$, est une 
mesure de Radon positive sur $G(F)$. En d'autres termes, pour toute fonction $f\in C^\infty_{\rm c}(G(F))$, l'intŽgrale 
$$\int_{G^u(F)\backslash G(F)} f(g^{-1} u  g)\dd \bar{g}\leqno{(1)}$$ est absolument convergente et dŽfinit une distribution invariante sur $G(F)$. 
La preuve repose sur un calcul explicite de cette distribution, que nous rappelons ci-dessous. 

Puisque $p=1$, on dispose de l'application exponentielle $\exp : \mathfrak{g}(F)\rightarrow G(F)$. L'ŽlŽment 
$X\in \mathfrak{g}(F)$ tel que $u= \exp(X)$ s'insre dans un $\mathfrak{sl}_2$-triplet de Jacobson-Morosov $(X,H,Y)$: $H$ et $Y$ appartiennent ˆ $\mathfrak{g}(F)$, 
$[H,X]=2X$, $[H,Y]=-2Y$ et $[X,Y]=H$. L'ŽlŽment $H$ dŽfinit une fibration $$\mbb{Z}\in i \mapsto \mathfrak{g}(i)=\{Z\in \mathfrak{g}\,\vert \, [H,Z]= i Z\}$$ de $\mathfrak{g}$ dŽfinie sur $F$. 
Ainsi $\mathfrak{p}= \sum_{i\geq 0}\mathfrak{g}(i)$ est une sous-$F$-algbre parabolique de $\mathfrak{g}$ de radical nilpotent 
$\mathfrak{u}= \sum_{i\geq 1} \mathfrak{g}(i)$ et $\mathfrak{m}= \mathfrak{g}(0)$ est une sous-$F$-algbre de Levi de $\mathfrak{p}$. 
Soient $P$, $U$, $M$ les sous-groupes de $G$ (dŽfinis sur $F$) correspondant ˆ $\mathfrak{p}$, $\mathfrak{u}$ et $\mathfrak{m}$. 
Puisque $p=1$, le morphisme (surjectif) $$G\rightarrow \mathrm{Ad}_G(X),\, g \mapsto {\rm Ad}_{g}(X)\leqno{(2)}$$ est sŽparable; en d'autres termes le commutant 
$\mathfrak{g}^X$ de $X$ dans $\mathfrak{g}$ co\"{\i}ncide avec l'algbre de Lie du centralisateur $G^X = G^{\rm u}$ de $X$ dans $G$. Or $G^X \subset P$ par consŽquent $\mathfrak{g}^X\subset \mathfrak{p}$ et 
le $F$-morphisme $[X,\cdot]: \mathfrak{g}(-1) \rightarrow \mathfrak{g}(1)$ est un isomorphisme. 
Cela assure l'existence d'une fonction $\varphi: \mathfrak{g}(2\mathpvg F) \rightarrow \mbb{R}_+$ telle que 
$\varphi(X) \neq 0$ et, pour tout $X'\in \mathfrak{g}(2,F)$ et tout $m\in M(F)$, on ait 
$$\varphi({\rm Ad}_m(X'))= \vert \mathrm{det}_F\left({\rm Ad}_m\,\vert \, \mathfrak{g}(1;F)\right)\vert_F\varphi(X')\ptf\leqno{(3)}$$ 
Soit $K$ un sous-groupe ouvert compact maximal de $G(F)$ tel que $G(F)= K P(F)$. L'intŽgrale (1) se rŽcrit (pour un bon choix de mesures de Haar) 
$$\int_{G^X(F)\backslash P(F)}f^K \circ \mathrm{exp}({\rm Ad}_{p^{-1}}(X)) \dd_{\mathrm{r}} \bar{p}\quad \hbox{avec} \quad f^K(g)= \int_Kf(k^{-1}gk)\dd k\ptf\leqno{(4)}$$ 
La sŽparabilitŽ du morphisme (2) entra"ne aussi que:
\begin{itemize}
\item l'application $M(F)\rightarrow \mathfrak{g}(2\mathpvg F),\, m \mapsto {\rm Ad}_m(X)$ est submersive, par consŽquent son image est un ouvert de $\mathfrak{g}(2\mathpvg F)$ que l'on note $\ES{O}^M_{F,X}$; 
\item l'application $P(F) \rightarrow \ES{O}_{F,X}^M \oplus \mathfrak{g}_3(F), \, p \mapsto {\rm Ad}_{p}(X)$ est surjective, o l'on a posŽ $\mathfrak{g}_3(F)=\sum_{i\geq 3}\mathfrak{g}(i\mathpvg F)$.  
\end{itemize}
Soit $\Lambda$ la distribution sur $G(F)$ dŽfinie par
$$\Lambda(f) = \int_{\ES{O}_{F,X}^M\times \mathfrak{g}_3(F)} \varphi(X')f^K(\mathrm{exp}(X'+Z))\dd X' \dd Z\ptf$$ 
La propriŽtŽ (3) assure qu'elle est invariante \cite[theorem 1]{RR}: on a $\Lambda(f\circ {\rm Int}_g)= \Lambda(f)$ pour tout $g\in G(F)$. 
Cela prouve qu'ˆ une constante $>0$ prs (qui dŽpend du choix de $\varphi$ et des mesures de Haar), l'intŽgrale (1) co\"{\i}ncide avec $\Lambda(f)$\footnote{
C'est un calcul analogue que nous faisons ici, valable pour $p\geq 1$, en remplaant la $G(F)$-orbite $\ES{O}_{F,u}$ par la $F$-strate $\bsfrY_{F,u}$ (voir \ref{les co-caractres 
(F,u)-optimaux} et \ref{les F-lames et les F-strates}).}.

\subsection{L'Žtat des lieux en caractŽristique $p\geq 1$}\label{l'Žtat des lieux} Pour les \guill{bons}\footnote{Rappelons que si $G$ est (absolument) quasi-simple, $p\geq 1$ est dit \textit{bon} pour $G$ si $p=1$ ou si $p>1$ ne divise aucun c\oe fficient de la plus grande racine (exprimŽe comme combinaison linŽaire de racines simples) du systme de racines de $G$. Les \textit{mauvais} $p>1$ pour $G$ sont: aucun si $G$ est de type ${\bf A}_l$; 
$p=2$ si $G$ n'est pas de type ${\bf A}_l$; $p=3$ si $G$ est de type exceptionnel (${\bf G}_2$, ${\bf F}_4$ ou ${\bf E}_*$); $p=5$ si $G$ est de type ${\bf E}_8$. 
Si $G$ n'est pas quasi-simple, $p$ est dit \textit{bon} pour $G$ s'il est bon 
pour tout sous-groupe distinguŽ quasi-simple de $G$.} $p\geq 1$, les orbites gŽomŽtriques 
nilpotentes dans $\mathfrak{g}= \mathfrak{g}(\overline{F})$ sont dŽcrites par le thŽorme de Bala-Carter. Pour $p=1$ ou $p\gg 1$, on peut utiliser la thŽorie des 
$\mathfrak{sl}_2$-triplets. Pour $p$ seulement \guill{bon}, il faut remplacer le $H$ du $\mathfrak{sl}_2$-triplet associŽ ˆ $X$ par un co-caractre $\lambda\in \check{X}(G)$ qui destabilise $X$, \cad 
tel que $\lim_{t\rightarrow 0} {\rm Ad}_{t^\lambda}(X)=0$, et soit en certain sens meilleur que les autres; cela a ŽtŽ fait d'abord par Pommerening en utilisant les 
co-caractres dits \guill{associŽs} ˆ $X$, puis par Premet en utilisant les co-caractres dits \guill{optimaux} de la thŽorie de Kempf-Rousseau. Pour Žtudier les orbites 
nilpotentes rationnelles, il faut ensuite faire descendre ces rŽsultats de $\overline{F}$ ˆ $F$. La thŽorie de Kempf-Rousseau se comporte trs bien par descente sŽparable, 
ce qui permet en gŽnŽral (si $p>1$) de se ramener ˆ $F^{\rm rad}$; mais pas ˆ $F$. 

Des rŽsultats plus fins concernant le thŽorme de Deligne-Ranga~Rao pour $p>1$ sont džs ˆ George McNinch \cite{MN}. 
Sous des hypothses standard un peu moins fortes que \guill{$p$ trs bon}\footnote{Rappelons que $p>1$ est dit \textit{trs bon} pour $G$ s'il est bon pour $G$ et si pour tout sous-groupe distinguŽ 
quasi-simple de $G$ de type ${\bf A}_l$, on a 
$l\not\equiv -1\;({\rm mod}\,p)$.}
mais assurant nŽanmoins que toutes les orbites gŽomŽtriques nilpotentes sont sŽparables, McNinch prouve dans loc.~cit. que si $X$ est un ŽlŽment nilpotent de  
$\mathfrak{g}(F)$, il existe un co-caractre de $G$ qui soit ˆ la fois associŽ ˆ $X$ et dŽfini sur $F$. Il en dŽduit le  
thŽorme de Deligne-Ranga~Rao sur la convergence des intŽgrales orbitales nilpotentes, resp. unipotentes (sous des hypothses supplŽmentaires, 
via une application du type \guill{logarithme} $G(F)\rightarrow \mathfrak{g}(F)$). 

Si $p>1$, les rŽsultats prouvŽs ici redonnent ceux de McNinch \cite{MN} pour les orbites rationnelles nilpotentes, resp. unipotentes\footnote{Pour les orbites unipotentes, on 
n'a pas besoin ici d'hypothses supplŽmentaires car on n'utilise pas d'application logarithme pour passer de $G(F)$ ˆ $\mathfrak{g}(F)$.}. Ils donnent mme plus (cf. \ref{sur la CV des IOU}). 

\subsection{L'analogue du thŽorme de \cite{RR} pour les $F$-strates}\label{l'analogue de [RR]} Cette sous-section et la suivante sont valables 
pour $p\geq 1$, sans aucune hypothse supplŽmentaire. 

Soit $u\neq 1$ un (vrai) ŽlŽment unipotent de $G(F)$, \cad contenu dans le radical unipotent d'un $F$-sous-groupe parabolique de $G$. 
On renvoie ˆ \ref{les co-caractres (F,u)-optimaux} et \ref{les F-lames et les F-strates} pour les dŽfinitions des objets introduits ci-aprs.  
Soient $\lambda \in \Lambda^{\rm opt}_{F,u}$ un co-caractre $(F,u)$-optimal de $G$, $P={_FP_u} \;(=P_\lambda)$ le $F$-sous-groupe parabolique de $G$ associŽ ˆ $u$, 
$\scrY=\scrY_{F,u}$ et $\bsfrY= \bsfrY_{F,u}\,(=\mathrm{Int}_{G(F)}(\scrY_{F,u}))$ respectivement la $F$-lame et la $F$-strate associŽes ˆ $u$. 
Soient aussi $k=m_{F,u}\,(=m_u(\lambda))$ et $\scrX= \scrX_{F,u}\,(=G_{\lambda,k}(F))$. 
On prouve que pour $i=1,\ldots , k-1$, 
le $F$-endomorphisme $X \mapsto {\rm Ad}_u(X)- X$ de $\mathfrak{g}(F)$ induit par restriction et passage aux quotients un isomorphisme 
$$\mathfrak{g}_{\lambda,-i}(F)/\mathfrak{g}_{\lambda,-i+1}(F) \buildrel\simeq \over{\longrightarrow} \mathfrak{g}_{\lambda, -i+k}(F)/\mathfrak{g}_{\lambda, -i+k+1}(F)\ptf\leqno{(1)}$$ 
Cela assure l'existence d'une fonction $$\varphi: \overline{\scrX}= \scrX/G_{\lambda,k+1}(F) \rightarrow \mbb{R}_+$$ telle que pour tout $\overline{x} \in\overline{\scrX}$ et tout $p\in P(F)$, on ait 
$$\varphi({\rm Int}_p(\overline{x})) = \bs{\delta}_P(p)\bs{\delta}_{\lambda , k}(p)^{-1}\varphi(\overline{x})\leqno{(2)}$$ avec 
$$\bs{\delta}_{\lambda ,k} (p)= \vert \det({\rm Ad}_p\,\vert \, \mathfrak{g}_{\lambda,k}(F))\vert_F\quad \hbox{et} \quad \bs{\delta}_{P}= \bs{\delta}_{\lambda,1}\ptf$$

On relve $\varphi$ en une fonction $G_{\lambda,k+1}(F)$-invariante (par transaltions) sur $\scrX$. On a 
$$\varphi(y)\neq 0\quad \hbox{pour tout} \quad y\in \scrY\ptf\leqno{(3)}$$ 
Fixons un sous-groupe ouvert compact maximal $K$ de $G(F)$ tel que $G(F)= KP(F)$. Observons que la $F$-lame $\scrY$ est ($\mathrm{Top}_F$)-ouverte 
dans $\scrX$ et que ${\rm Int}_K(\scrY)=\bsfrY$. Soit $I_{\bsfrY}$ la distribution sur $G(F)$ dŽfinie par 
$$I_{\bsfrY}(f) = \int_{K\times \scrY} \varphi(x)f(kxk^{-1})\dd k \dd x\leqno{(4)}$$ o $\dd x$ est une mesure de Haar sur $\scrX$. Elle annule toute fonction 
$f\in C^\infty_{\rm c}(G(F))$ qui s'annule sur $\bsfrY$ et comme dans \cite{RR}, la propriŽtŽ (2) assure qu'elle est invariante pour la conjugaison dans $G(F)$. 

\subsection{Une formule du type \guill{intŽgrale orbitale} pour $I_{\bsfrY}$}\label{une formule du type IO}
Pour relier la distribution $I_{\bsfrY}$ ˆ une formule qui ressemble ˆ une intŽgrale orbitale du type \ref{le thŽorme de D-RR}\,(1), il faut travailler un peu plus. 
Soit $\mathcal{H}$ l'espace des fonctions localement constantes $\Phi$ sur $G(F) \times \scrX$ qui vŽrifient, pour 
tout $(g,x)\in G(F)\times \scrX$: 
\begin{itemize}
\item $\Phi(gp,p^{-1}xp) = \bs{\delta}_P(p)^{-1} \bs{\delta}_{\lambda,k}(p)\Phi(g,x)$ pour tout $p\in P(F)$; 
\item il existe un sous-ensemble compact $\Omega_\Phi \subset G(F) \times \scrX$ tel que si 
$\Phi(g,x)\neq 0$, alors il existe un $p\in P(F)$ tel que $(gp,p^{-1}xp)\in \Omega_\Phi$.
\end{itemize} 
On note $\bs{\mu}$ la fonctionnelle linŽaire positive sur $\mathcal{H}$ dŽfinie par 
$$\langle \bs{\mu}, \Phi \rangle = \int_{G(F)\times \scrX}f(g,x)\dd g \dd x$$ 
pour une (i.e. pour toute) fonction $f\in C^\infty_{\rm c}(G(F)\times \scrX)$ telle que 
$$\Phi(g,x) = \int_{P(F)} f(gp,p^{-1} x p) \bs{\delta}_{\lambda,k}(p)^{-1}\dd_{\rm r} p\ptf $$ 
L'application continue $$G(F)\times \scrY \rightarrow \bsfrY,\, (g,y) \mapsto (g,y) \mapsto gyg^{-1}$$ se quotiente en un homŽomorphisme
$$ G(F)\times^{P(F)} \scrY \buildrel \simeq \over{\longrightarrow} \bsfrY\ptf \leqno{(1)}$$ On prouve que  la mesure $\bs{\mu}_{\scrY}^\varphi$ sur 
$G(F)\times^{P(F)}\scrY$ dŽfinie par $$\dd \bs{\mu}_{\scrY}^\varphi (g,y) = \varphi(y) \dd \bs{\mu}(g,y)$$ donne via l'homŽomorphisme (1) une mesure de Radon positive et $G(F)$-invariante (pour la conjugaison) sur $G(F)$. 
PrŽcisŽment, les mesures de Haar Žtant choisies de manire cohŽrente, pour toute fonction $f\in C^\infty_{\rm c}(G(F))$, on a $$ I_{\bsfrY}(f)=\int_{G(F)\times^{P(F)}\scrY}f(gyg^{-1}) 
\varphi(y) \dd \bs{\mu}(g,y)\ptf \leqno{(2)}$$

\subsection{Sur la convergence des intŽgrales orbitales unipotentes}\label{sur la CV des IOU}
Supposons que la $G(F)$-orbite $\ES{O}_{F,u}$ soit ouverte dans $\bsfrY=\bsfrY_{F,u}$ (e.g si $p=1$\footnote{En effet dans ce cas, la $F$-strate $\bsfrY$   
est rŽunion {\it finie} de $G(F)$-orbites, chacune d'elles Žtant ouverte et fermŽe dans $\bsfrY$.}). Alors la 
$P(F)$-orbite $\ES{O}_{F,u}^P= \{pup^{-1}\,\vert\, p\in P(F)\}$ est ouverte dans $\scrY=\scrY_{F,u}$, 
donc aussi dans $\scrX=\scrX_{F,u}$, et en remplaant $\scrY$ par $\ES{O}_{F,u}^P$ dans la dŽfinition de $I_{\bsfrY}$ (cf. \ref{l'analogue de [RR]}\,(4)), 
on obtient une distribution positive $G(F)$-invariante non nulle $I_u=I_{\ES{O}_{F,u}}$ sur $G(F)$:  
$$I_u(f) = \int_{K \times \ES{O}_{F,u}^P} \varphi(x)f(k x k^{-1}) \dd k \dd x \ptf$$ 
L'existence de $I_u$ assure que le centralisateur $G^u(F)$ de $u$ dans $G(F)$ est unimodulaire. Si 
$\dd g^u$ est une mesure de Haar sur $G^u(F)$, il existe une constante $c>0$ telle que pour toute fonction $f\in C^\infty_{\rm c}(G(F))$, on ait 
$$\int_{G^u(F)\backslash G(F)}f(g^{-1} u g) \textstyle{\frac{dg}{dg^u}}= c\hskip0.3mm I_u(f);$$ 
l'intŽgrale Žtant absolument convergente. C'est le thŽorme de Deligne-Ranga Rao. 
Ce cas particulier contient tous les cas traitŽs dans \cite{MN}. 

On prouve ici que le thŽorme de Deligne-Ranga Rao reste vrai sous les hypothses plus faibles suivantes (\ref{hypothses H}):
\begin{enumerate}
\item[(H1)] Le centralisateur $G^u(F)$ de $y$ dans $G(F)$ est unimodulaire.
\item[(H2)] Il existe un feuillet $S$ pour $u$ (dans $\scrY$ relativement ˆ l'action de $P(F)$), cf. 
\ref{def slice}.
\end{enumerate}
Posons $$\wt{S}= \mathrm{Int}_{P(F)}(S)\quad \hbox{et}\quad \bs{\mathfrak{S}}=\mathrm{Int}_K(\wt{S})\;(= \mathrm{Int}_{G(F)}(S))\ptf$$
D'aprs l'hypothse (H2), $\wt{S}$ est ouvert dans $\scrY$ (par suite $\bs{\mathfrak{S}}$ est ouvert dans $\bsfrY$) et $S$ est fermŽ dans $\wt{S}$. De plus l'espace 
$\wt{S}/P(F)\;(=\bs{\mathfrak{S}}/G(F))$ des $P(F)$-orbites dans $\wt{S}$ est homŽomorphe ˆ $S$.  
On prouve que pour toute fonction $f\in C^\infty_{\rm c}(G(F))$ et tout $y\in \bs{\mathfrak{S}}$, l'intŽgrale  
$$I_y(f)= \int_{G^y(F)\backslash G(F)} f(g^{-1}yg)\textstyle{\frac{dg}{dg^y}}\leqno{(1)}$$ est absolument convergente. De plus, en remplaant $\scrY$ par $\wt{S}$ dans la dŽfinition de $I_{\bsfrY}$ 
(cf. \ref{l'analogue de [RR]}\,(4)), on obtient une distribution positive $G(F)$-invariante $I_{\bs{\mathfrak{S}}}$ sur $G(F)$, donnŽe par (cf. \ref{une formule du type IO}\,(2)): 
$$I_{\bs{\mathfrak{S}}}(f)= \int_{G(F)\times^{P(F)}\wt{S}}f(gyg^{-1}) \varphi(y) \dd \bs{\mu}(g,y)\ptf$$ 
On prouve qu'il existe une mesure de Radon positive non nulle  
$\bs{\eta}_{S}$ sur l'espace des orbites $\wt{S}/P(F)$ telle que pour toute fonction $f\in C^\infty_{\rm c}(G(F))$, on ait 
$$I_{\bs{\mathfrak{S}}}(f)= \int_{\wt{S}/P(F)} \varphi(y) I_y(f) \dd \bs{\eta}_{S}(y)\ptf\leqno{(2)}$$

On a les mmes rŽsultats pour les ŽlŽments nilpotents de $\mathfrak{g}(F)$. 

\vskip1mm
Une remarque conjecturale pour finir. Nous ne connaissons pas d'exemple o les hypothses (H1) et (H2) ne soient pas vŽrifiŽes. D'un autre c™tŽ le dŽveloppement fin 
de la contribution unipotente ˆ la formule des traces \cite{L} devrait fournir localement des distributions invariantes qui sont des intŽgrales sur des ouverts 
($G(F)$-invariants) de $F$-strates unipotentes et pas des intŽgrales orbitales sur les orbites rationnelles.  
C'est pourquoi pour les \textit{mauvais} $p$, 
il n'est peut-tre pas nŽcessaire de pousser plus avant l'Žtude des intŽgrale orbitales unipotentes.

\subsection{Organisation des rŽsultats}\label{organisation}
Dans la section \ref{un rŽsultat technique}, aprs des rappels sur la thŽorie des $F$-strates unipotentes et nilpotentes (\ref{les co-caractres (F,u)-optimaux}-\ref{de G ˆ Lie(G)}), on 
dŽmontre que les morphismes \ref{l'analogue de [RR]}\,(1) sont des isomorphismes --- le rŽsultat technique de l'article --- par la mŽthode des corps proches (\ref{le cas d'un corps local}); on donne aussi une preuve 
de ce rŽsultat valable pour $F$ quelconque si tous les facteurs quasi-simples de $G$ sont de type ${\bf A}_l$, ${\bf D}_l$ ou de type exceptionnel ${\bf E}_*$ (\ref{le cas o les N sont Žgaux ˆ 1 ou -1}). 

La distribution $I_{\bsfrY}$ fait l'objet de la section \ref{mesures de Radon}; en particulier l'ŽgalitŽ \ref{une formule du type IO}\,(2) est obtenue en 
\ref{le rŽsultat principal}. En \ref{le cas o l'orbite est ouverte}, on traite le cas o l'orbite rationnelle est ouverte dans la $F$-strate. 
Ce cas particulier nous permet en \ref{comparaison} 
de comparer nos rŽsultats avec ceux dŽjˆ connus \cite{RR,MN}. En \ref{dŽsintŽgration}, sous les 
hypothses (H1) et (H2), on prouve la convergence des intŽgrales orbitales unipotentes \ref{sur la CV des IOU}\,(1) ainsi que la formule \ref{sur la CV des IOU}\,(2). 
Enfin \ref{le cas global} contient une remarque sur le cas des corps globaux qui sera utile pour la suite de notre travail sur la formule des traces.

\section{Un rŽsultat technique}\label{un rŽsultat technique} 

Dans cette section, $F$ est un corps commutatif quelconque et la topologie est celle de Zariski; 
ˆ l'exception de \ref{le cas d'un corps local} o $F$ est un corps localement compact non archimŽdien et la topologie est celle dŽfinie par $F$ ($\mathrm{Top}_F$). 
Le rŽsultat technique en question est la proposition \ref{proposition noyau}. 
Il convient de lire cette section avec \cite{L} sous la main. 

\subsection{Les co-caractres $(F,u)$-optimaux}\label{les co-caractres (F,u)-optimaux} On a dŽveloppŽ dans \cite{L} une thŽorie des $F$-lames et des $F$-strates unipotentes de $G(F)$ basŽe sur les travaux de 
Kempf-Rousseau \cite{K,R} et Hesselink \cite{H1,H2}. Jusqu'en \ref{gŽnŽ et descente}, on rappelle les principaux rŽsultats de cette thŽorie. 

On fixe une $F$-norme $G$-invariante $\|\,\|$ sur $\check{X}(G)$ (cf. \cite[2.2]{L}). Pour chaque tore maximal $T$ de $G$ 
dŽfini sur $F$, on dispose donc d'une forme bilinŽaire symŽtrique dŽfinie positive $(\cdot , \cdot): \check{X}(T)\times \check{X}(T)\rightarrow \mbb{Z}$ qui 
est invariante par l'action du groupe de Weyl $W^G(T)= N^G(T)/T$ et par celle du groupe de Galois $\Gamma_F$. On Žtend $\| \,\|$ ˆ $\check{X}(G)_{\mbb{Q}}$ 
(cf. \cite[2.4]{L}) et $(\cdot,\cdot)$ ˆ $\check{X}(T)_{\mbb{Q}}= \check{X}\otimes_{\mbb{Z}}\mbb{Q}$ par conjugaison et linŽaritŽ. 

Soit $\UU=\UU^G$ la sous-variŽtŽ fermŽe de $G=G(\overline{F})$ formŽe des ŽlŽments unipotents. Soit $\FU=\FU^G$ le sous-ensemble de $\UU$ 
formŽ des ŽlŽments qui sont contenus dans le radical unipotent d'un $F$-sous-groupe parabolique de $G$, et soit $\UF=G(F) \cap \FU$ 
l'ensemble des (vrais) ŽlŽments unipotents de $G(F)$. L'inclusion $\UF \subset \UU(F)$ est en gŽnŽral stricte. 

Pour $\lambda\in \check{X}(G)$, on note:
\begin{itemize}
\item $P_\lambda$ l'ensemble des $g\in G$ tels que la limite $\lim_{t\rightarrow 0} t^\lambda g t^{-\lambda}$ existe\footnote{i.e. 
tels que le morphisme 
$\mbb{G}_{\mathrm{m}}\rightarrow G,\, t \mapsto t^\lambda g t^{-\lambda}$ se prolonge, de manire nŽcessairement unique, en un morphismde $\mbb{G}_{\mathrm{a}} \rightarrow G$.};  
\item $U_\lambda= U_{P_\lambda}$ l'ensemble des $g\in G$ tels que $\lim_{t\rightarrow 0} t^\lambda g t^{-\lambda}=1$; 
\item $M_\lambda$ le centralisateur de ${\rm Im}(\lambda)$ dans $G$. 
\end{itemize} Ainsi $P_\lambda$ est un sous-groupe parabolique de $G$ de radical unipotent $U_\lambda$ et $M_\lambda$ est une composante de Levi de $P_\lambda$: on a $P_\lambda =M_\lambda \ltimes U_\lambda$.  
Pour $u\in G\smallsetminus \{1\}$ et $\mu\in \check{X}(G)$ tels que $u\in U_\lambda$, la fibre schŽmatique au-dessus de $1$ du prolongement ˆ $\mbb{G}_{\mathrm{a}}$ du 
morphisme $\mbb{G}_{\mathrm{m}} \rightarrow G,\, t \mapsto t^\lambda u t^{-\lambda}$ est d'algbre affine $\overline{F}[T]/(T^m)$ pour un entier $m>0$; on pose 
$m_u(\lambda)=m$ et $\rho_u(\lambda) = \frac{m_u(\lambda)}{\|\lambda\|}$. 

Ë tout ŽlŽment $u\in \UU(F)\smallsetminus \{1\}$ sont associŽs comme suit 
\begin{itemize}
\item un sous-ensemble $\Lambda_{F,u}^{\rm opt}\subset \check{X}_F(G)$; 
\item un invariant $m_{F,u}\in \mbb{N}^*$; 
\item un $F$-sous-groupe parabolique $ {_FP_u}$ de $G$.  
\end{itemize}
Un co-caractre $\lambda\in \check{X}(G)$ est dit \textit{primitif} s'il n'existe aucun co-caractre 
$\lambda'\in \check{X}(G)$ tel que $\lambda = k \lambda'$ avec $k\in \mbb{Z}_{\geq 2}$. Un co-caractre primitif $\lambda\in \check{X}_F(G)$ est dit \textit{$(F,u)$-optimal} si $u\in U_\lambda$ et si pour tout 
$\mu\in \check{X}_F(G)$ tel que $u\in U_\mu$, on a $\rho_u(\mu)\leq \rho_u(\lambda)$.  
On note $\Lambda_{F,u}^{\rm opt}$ l'ensemble des co-caractres 
$(F,u)$-optimaux. L'invariant $m_u(\lambda)$ et le $F$-sous-groupe parabolique 
$P_\lambda$ de $G$ ne dŽpendent pas de $\lambda\in \Lambda_{F,u}^{\rm opt}$; on les note $m_{F,u}$ et ${_FP_u}$, et on pose ${_FU_u}=U_\lambda$. 
Observons que, bien que l'ensemble $\Lambda_{F,u}^{\mathrm{opt}}$ et le groupe ${_FP_u}$ soient 
dŽfinis via le choix de la $F$-norme $G$-invariante $\|\,\|$, ils n'en dŽpendent pas. 
De plus, l'ensemble $\Lambda_{F,u}^{\rm opt}$ forme une seule orbite sous l'action du sous-groupe parabolique $P_{F,u}= {_FP_u}(F)$ de 
$G(F)$; prŽcisŽment, $\Lambda_{F,u}^{\mathrm{opt}}$ est un espace principal homogne sous l'action de $U_{F,u}= {_FU_u}(F)$. Un tore $F$-dŽployŽ maximal $A$ 
de $G$ est dit \textit{$(F,u)$-optimal} s'il est contenu dans ${_FP_u}$; pour un tel $A$, on a $\check{X}(A)\cap \Lambda_{F,u}^{\rm opt}= \{\lambda'\}$ et $M_{\lambda'}$ est une 
composante de Levi de ${_FP_u}$ dŽfinie sur $F$. Enfin on note 
$$\bs{\Lambda}_{F,u} = \left\{ \textstyle{\frac{\lambda}{m_u(\lambda)}}\,\vert \, \lambda \in \Lambda_{F,u}^{\rm opt}\right\}\subset \check{X}_F(G)_{\mbb{Q}}$$ 
l'ensemble des co-caractres virtuels $(F,u)$-optimaux normalisŽs. 

On peut dŽfinir $\bs{\Lambda}_{F,u}$ directement de la manire suivante. 
Pour $\mu \in \check{X}(G)_{\mbb{Q}}$, on choisit un 
$r\in \mbb{N}^*$ tel que $\lambda = r\mu\in \check{X}(G)$. Posons $P_\mu= P_{\lambda}$ et pour $u'\in U_\mu = U_{\lambda}$, posons $m_{u'}(\mu) = \frac{1}{r}m_{u'}(\lambda)$. 
Ces dŽfinitions ne dŽpendent pas du choix de $r$ et l'on a $\frac{m_{u'}(\mu)}{\| \mu \|} = \frac{m_{u'}(\lambda)}{\|\lambda\|} = \rho_{u'}(\lambda)$. On voit 
donc que $\bs{\Lambda}_{F,u}$ est l'ensemble des $\mu\in \check{X}_F(G)_{\mbb{Q}}$ tels que $u\in U_\mu$, $m_u(\mu)\geq 1$ et 
$\|\mu\|$ soit minimal pour ces conditions; si $\mu\in \bs{\Lambda}_{F,u}$, on a forcŽment $m_u(\mu)=1$.

Il est commode d'Žtendre ces dŽfinitions ˆ $u=1$: 
on pose $\Lambda_{F,1}^{\rm opt} = \bs{\Lambda}_{F,1}= \{0\}$, $m_{F,1}=m_1(0)= + \infty$ et ${_FP_1}=G$.

\subsection{Les $F$-lames et les $F$-strates}\label{les F-lames et les F-strates}
Tout co-caractre $\lambda\in \check{X}(G)$ dŽfinit une filtration $(G_{\lambda ,i})_{i\in \mbb{N}}$ de $G$, cf. \cite[2.4]{L}. 
Si $\lambda$ est dŽfini sur $F$, ces filtrations le sont aussi. On a toujours $P_\lambda = G_{\lambda,0}$ et $U_\lambda = G_{\lambda,1}$. 
Pour $u\in U_\lambda$, on a $$m_u(\lambda)= \inf\{i\in \mbb{N}^*\,\vert\, u\in G_{\lambda,i}\}\ptf$$ 

Ë tout ŽlŽment $u\in \UF\smallsetminus\{1\}$ sont associŽes une $F$-lame $\scrY_{F,u}$ et une $F$-strate $\bsfrY_{F,u}$ (de $\UF$): 
$$\scrY_{F,u} = \{u'\in \UF\, \vert \, \bs{\Lambda}_{F,u'}= \bs{\Lambda}_{F,u}\}\quad \hbox{et} \quad \bsfrY_{F,u}=\mathrm{Int}_{G(F)}(\scrY_{F,u})\ptf$$ 
En d'autres termes, $\bsfrY_{F,u}$ est l'ensemble des $u'\in \UF$ tels que $\bs{\Lambda}_{F,u'}= {\rm Int}_g\circ \bs{\Lambda}_{F,u}$ pour un $g\in G(F)$. 
Les $F$-strates sont donc les classes de $G(F)$-conjugaison de $F$-lames. 
On dŽfinit aussi des sous-ensembles $\scrX_{F,u}$ et $\bsfrX_{F,u}$ de $\UF$ tels que 
$\scrY_{F,u}\subset \scrX_{F,u}$ et $\bsfrY_{F,u}\subset \bsfrX_{F,u}$. Pour cela on choisit 
un $\lambda\in \Lambda_{F,u}^{\rm opt}$ et on pose $k= m_u(\lambda)$; alors 
$$\scrX_{F,u}= G_{\lambda,k}(F)\quad\hbox{et}\quad \bsfrX_{F,u}={\rm Int}_{G(F)}(\scrX_{F,u})\ptf$$
D'aprs \cite[3.4.4]{L}, on a
$$\scrY_{F,u} G_{\lambda , k+1}(F)= G_{\lambda, k+1}(F)\scrY_{F,u}= \scrY_{F,u}\ptf$$ 
Pour $u=1$, on pose $\bsfrY_{F,1}=\scrY_{F,1}=\{1\}$ et $\bsfrX_{F,1}=\scrX_{F,1}=\{1\}$. 

Il n'y a qu'un nombre fini de $F$-strates $\bsfrY_{F,u}$ avec $u\in \UF$ et elles sont deux-ˆ-deux disjointes.

\subsection{La thŽorie sur l'algbre de Lie}\label{la thŽorie sur Lie(G)} 
On a la mme thŽorie pour $\mathfrak{g}={\rm Lie}(G)$ muni de l'action adjointe de $G$. Pour $\lambda \in \check{X}(G)$, on pose 
$$\mathfrak{p}_\lambda ={\rm Lie}(G_\lambda)\,,\quad \mathfrak{u}_\lambda = {\rm Lie}(U_\lambda)\quad \hbox{et} \quad \mathfrak{m}_\lambda= {\rm Lie}(M_\lambda)\ptf$$ 
Soit $\NN=\NN^G$ la sous-variŽtŽ fermŽe de $\mathfrak{g}= \mathfrak{g}(\overline{F})$ formŽe des ŽlŽments nilpotents. On note 
${_F\NN}= {_F\NN^G}$ le sous-ensemble de $ \NN$ formŽ des ŽlŽments qui sont contenus dans le radical nilpotent d'une sous-$F$-algbre parabolique de $\mathfrak{g}$, et on pose $\NN_F = {_F\NN}\cap \mathfrak{g}(F)$.  
Pour $X\in \NF$, on dŽfinit comme en \ref{les co-caractres (F,u)-optimaux} le sous-ensemble $\Lambda_{F,X}^{\rm opt}\subset \check{X}_F(G)$, l'invariant $m_{F,X}\in \mbb{N}^*\cup \{+\infty\}$ et 
le $F$-sous-groupe parabolique ${_FP_X}$; ainsi que le sous-ensemble $\bs{\Lambda}_{F,X}= \frac{1}{m_{F,X}}\Lambda_{F,X}^{\mathrm{opt}}\subset \check{X}_F(G)_{\mbb{Q}}$. 

Pour $\lambda\in \check{X}(G)$, on dŽfinit comme suit une filtration $P_\lambda$-invariante 
$(\mathfrak{g}_{\lambda,i})_{i\in \mbb{Z}}$ de $\mathfrak{g}$, dŽfinie sur $F$ si $\lambda$ l'est. 
Pour $i\in \mbb{Z}$, on pose $$\mathfrak{g}_{\lambda , i}= \bigoplus_{j\geq i} \mathfrak{g}_\lambda(j)\quad\hbox{avec }\quad 
\mathfrak{g}_\lambda(i)= \{X\in \mathfrak{g}\,\vert \, {\rm Ad}_{t^\lambda}(X)= t^i X,\, \forall t\in \mbb{G}_{\rm m} \};$$ 
l'espace quotient $\mathfrak{g}_{\lambda,i}/\mathfrak{g}_{\lambda,i+1}$ s'identifie naturellement ˆ $\mathfrak{g}_\lambda(i)$.  
On a toujours 
$\mathfrak{p}_\lambda = \mathfrak{g}_{\lambda,0}$, $\mathfrak{u}_\lambda = \mathfrak{g}_{\lambda,1}$ et $\mathfrak{m}_\lambda = \mathfrak{g}_\lambda(0)$. 
Pour $i\in \mbb{Z}$, $\mathfrak{g}_\lambda(i)$ est un sous-espace $M_\lambda$-invariant de $\mathfrak{g}$, dŽfini sur $F$ si $\lambda$ l'est; et pour 
$i\in \mbb{N}$, on a $\mathfrak{g}_{\lambda,i}= {\rm Lie}(G_{\lambda,i})$. 

Pour $X\in \NF$, on dŽfinit comme en \ref{les F-lames et les F-strates} les sous-ensembles $\scrY_{F,X}\subset \scrX_{F,X}$ et $\bsfrY_{F,X}\subset \bsfrX_{F,X}$ de $\NF$. 
Pour $X\in \NF\smallsetminus\{0\}$, $\lambda\in \Lambda_{F,X}^{\rm opt}$ et $k= m_{X}(\lambda)$, on a \cite[2.7.7]{L}
$$\scrY_{F,X}+ \mathfrak{g}_{\lambda,k+1}(F)= \scrY_{F,X}\ptf$$

Comme pour les $F$-strates de $\UF$, il n'y a qu'un nombre fini de $F$-strates de $\NF$ et elles sont deux-ˆ-deux disjointes.

\subsection{Du groupe ˆ l'algbre de Lie}\label{de G ˆ Lie(G)}\label{de G ˆ Lie(G)} 
Soit $P_0$ un $F$-sous-groupe parabolique minimal de $G$. Notons $U_0=U_{P_0}$ son radical unipotent. 
Observons que $$\UF= {\rm Int}_{G(F)}(U_0(F))\quad \hbox{et} \quad \NF= {\rm Ad}_{G(F)}(\mathfrak{u}_0(F))\ptf$$ Soit $M_0$ une composante de Levi de $P_0$ dŽfinie sur $F$ et soit $A_0$ le tore $F$-dŽployŽ maximal dans le centre de $M_0$. 
On sait qu'il existe un $F$-isomorphisme de variŽtŽs algŽbriques (cf. \cite[3.4]{L}) $$j_0: \mathfrak{u}_0={\rm Lie}(U_0) \rightarrow U_0$$ 
compatible ˆ l'action de $A_0$, \cad tel que $j_0({\rm Ad}_a(X))={\rm Int}_a(j(X))$ pour tout $a\in A_0$ et tout $X\in \mathfrak{u}_0$. 
On note $\ES{R}=\ES{R}_{A_0}\subset X(A_0)$ l'ensemble des racines de $A_0$ dans $G$, $\ES{R}^+\subset \ES{R}$ le sous-ensemble 
formŽ des racines dans $P_0$ et $\Delta\subset \ES{R}^+$ la base formŽe des racines simples.

Un ŽlŽment $u\in \UF$ est dit \textit{en position standard} si ${_FP_u}\supset P_0$. Observons 
que si $u$ est en position standard, alors $u$ appartient ˆ $U_0(F)$\footnote{Cette condition n'est 
pas suffisante: pour que $u\in U_0(F)$ soit en position 
standard, il faut et il suffit que le tore $A_0$ soit $(F,u)$-optimal et que le co-caractre $\lambda$ dŽfini par $\check{X}(A_0)\cap \Lambda_{F,u}^{\rm opt}
= \{\lambda\}$ vŽrifie l'inclusion $P_\lambda \supset P_0$.} et tout ŽlŽment 
de la $F$-lame $\scrY_{F,u}$ est en position standard; auquel cas on dit que $\scrY=\scrY_{F,u}\;(\subset U_0(F))$ est une $F$-lame \textit{standard}. Tout 
ŽlŽment de $\UF$ est $G(F)$-conjuguŽ ˆ un ŽlŽment en position standard et toute $F$-strate de $\UF$ contient 
une unique $F$-lame standard. Les mmes dŽfinitions s'appliquent aux ŽlŽments de $\NF$. 

D'aprs \cite[3.4.5]{L}, l'application $j_0$ induit une bijection entre les $F$-lames standard de $\UF$ et les $F$-lames standard 
de $\NF$. Cette bijection ne dŽpend pas de $j_0$: pour $X\in \NF$ en position standard, on a 
$\bs{\Lambda}_{F,j_0(X)}= \bs{\Lambda}_{F,X}$. On en dŽduit que pour $X\in \NF$ (en position standard ou non), 
l'ensemble $$\bsfrY_{F,X}^G \bydef \{u \in \UF\,\vert \, \hbox{$\exists g\in G(F)$ tel que $\bs{\Lambda}_{F,u} = {\rm Int}_g\circ \bs{\Lambda}_{F,X}$}\}$$ 
est une $F$-strate de $\UF$ qui ne dŽpend que de la $F$-strate $\bsfrY_{F,X}$ de $\NF$. Cela fournit 
une bijection naturelle entre les 
$F$-strates de $\NF$ et celles de $\UF$ \cite[3.4.8]{L}. C'est d'ailleurs gr‰ce ˆ cette bijection que certains rŽsultats pour les $F$-strates de $\UF$ ont ŽtŽ prouvŽs, 
ˆ partir des rŽsultats analogues pour les $F$-strates de $\NF$.

\subsection{GŽnŽralisation et descente}\label{gŽnŽ et descente}  Ces thŽories (sur $G$ et sur $\mathfrak{g}$) sont deux cas particuliers d'une thŽorie plus gŽnŽrale, valable pour toute $G$-variŽtŽ 
algŽbrique affine pointŽe $(V,e_V)$ dŽfinie sur $F$, avec $e_V\in V(F)$, cf. \cite[2.4, 2.5]{L}; le cas d'un $G$-module $(V,e_V=0)$ dŽfini sur $F$ gŽnŽralisant celui 
de $\mathfrak{g}$. 

Un ŽlŽment $v\in V(F)$ est dit \textit{$(F,G)$-instables} s'il existe un co-caractre $\lambda \in \check{X}_F(G)$ tel que $\lim_{t\rightarrow 0}t^\lambda \cdot v = e_V$; et 
il est dit \textit{$(F,G)$-semi-stable} sinon, \cad si pour tout $\lambda\in \check{X}_F(G)$ tel que la $\lim_{t\rightarrow 0} t^\lambda\cdot v$ existe, cette limite n'est pas $e_V$.
On note $\ES{N}_F^G(V,e_V)$ l'ensemble des ŽlŽments \textit{$(F,G)$-instables} de $V(F)$.
Pour $v\in \ES{N}_F^G(V,e_V)$, on dŽfinit comme plus haut ${_FP_v}$, $\bs{\Lambda}_{F,v}$ (etc.),  
et les sous-ensembles $\scrY_{F,v}$ et $\bsfrY_{F,v}=G(F)\cdot \scrY_{F,v}$ de $\ES{N}_F^G(V,e_V)$. 
L'une des propriŽtŽs fondamentales de la thŽorie est la \textit{descente sŽparable} \cite[5.5]{H1} 
(cf. \cite[2.3.11]{L} pour la version non algŽbrique), \cite[2.5.18, 2.5.19]{L}: 

\begin{proposition}\label{descente sŽparable}
Soit $v\in \ES{N}_F^G(V,e_V)$. Soit $E/F$ une extension sŽparable (algŽbrique ou non) telle que $(E^\mathrm{s\acute{e}p})^{\mathrm{Aut}_F(E^\mathrm{s\acute{e}p})}= F$\footnote{Cette  
ŽgalitŽ est toujours vŽrifiŽe si $E/F$ est algŽbrique ou de degrŽ de transcendance infini, e.g. si $E$ est le complŽtŽ $F_\nu$ d'un corps global $F$ en une place $\nu$ de $F$.}. 
\begin{enumerate}
\item[(i)] $\bs{\Lambda}_{F,v} = \check{X}_F(G)_{\mbb{Q}}\cap \bs{\Lambda}_{E,v}$; en particulier ${_EP_v}= {_FP_v}\times_FE$.  
\item[(ii)] $\scrY_{F,v}=V(F)\cap \scrY_{E,v}$ et $\bsfrY_{F,v}= V(F)\cap \bsfrY_{E,v}$.
\end{enumerate}
\end{proposition} 

Il convient aussi, quand c'est possible, de comparer la thŽorie sur $F$ avec la thŽorie gŽomŽtrique, \cad sur $\overline{F}$ (si $p=1$, 
cette comparaison est donnŽe par \ref{descente sŽparable}). Pour $F=\overline{F}$, on supprimera l'indice $F$ dans les dŽfinitions prŽcŽdentes: 
$P_u= {_{\smash{\overline{F}}}P_u}$, $\bs{\Lambda}_u= \bs{\Lambda}_{\smash{\overline{F},u}}$, etc. On dira $G\;(=G(\overline{F}))$-instable, resp. $G$-semi-stable, 
au lieu de $(\overline{F},G)$-stable, resp. $(\overline{F},G)$-semi-stable. Pour $v\in \ES{N}_F^G(V,e_V)$, l'intersection 
$\check{X}_F(G)_{\mbb{Q}} \cap \bs{\Lambda}_v$ peut tre vide; si elle est non vide, alors elle co\"{\i}ncide avec $\bs{\Lambda}_{F,v}$ et on a 
$\scrY_{F,v}= \scrY_v(F)$.

Cela nous amne ˆ considŽrer l'hypothse suivante: 

\begin{hypoth}\label{hyp bonnes F-strates}
Pour tout $v\in \ES{N}^G_F(V,e_V)$, on a $\check{X}_F(G)_{\mbb{Q}}\cap \bs{\Lambda}_v\neq \emptyset$. 
\end{hypoth}

Si l'hypothse \ref{hyp bonnes F-strates} est vŽrifiŽe, alors pour tout $v\in \ES{N}_F^G(V,e_V)$, on a \cite[2.6.13]{L} $$\bsfrY_{F,v}= \ES{N}_F^G(V,e_V)\cap \bsfrY_v(F);$$ 
cela fournit une bijection entre les strates (gŽomŽtriques) de $\ES{N}^G(V,e_V)$ qui intersectent $\ES{N}_F^G(V,e_V)$ et les $F$-strates de $\ES{N}_F^G(V,e_V)$. 

Nous savons que pour $V=(G,e_G=1)$ muni de l'action par conjugaison et pour $V=(\mathfrak{g}, e_\mathfrak{g}=0)$ muni de l'action adjointe, l'hypothse \ref{hyp bonnes F-strates} est vŽrifiŽe 
dans les cas suivants: 
\begin{itemize}
\item $p=1$ ou $p>1$ est \textit{trs bon} pour $G$; 
\item $F$ est un corps local (i.e. localement compact, non discret); 
\item $F$ est un corps global.
\end{itemize}

\subsection{\'EnoncŽ du rŽsultat technique et premires rŽductions}\label{ŽnoncŽ}On suppose jusqu'ˆ la fin de la section \ref{un rŽsultat technique} que 
l'hypothse \ref{hyp bonnes F-strates} est vŽrifiŽe pour $V=G$ muni de l'action par conjugaison: pour tout $u\in \UF$, on a $\check{X}_F(G)_{\mbb{Q}}\cap \bs{\Lambda}_u\neq \emptyset$. 
D'aprs \cite{L}, elle l'est alors aussi pour $V=\mathfrak{g}$ muni de l'action adjointe: pour tout $X\in \NF$, on a $\check{X}_F(G)_{\mbb{Q}}\cap \bs{\Lambda}_X\neq \emptyset$. 

Pour $u\in \UU$, on note $G^u$ le centralisateur schŽmatique $u$ dans $G$ et on pose 
$$\mathfrak{g}^u = \ker \{{\rm Ad}_u -{\rm Id}: \mathfrak{g} \rightarrow \mathfrak{g}\}\ptf$$ 
On a toujours l'inclusion ${\rm Lie}((G^u)^{\mathrm{r\acute{e}d}}) \subset \mathfrak{g}^u$ o $(G^u)^{\mathrm{r\acute{e}d}}$ est le centralisateur schŽmatique \textit{rŽduit} de $u$ dans $G$ 
(correspondant au quotient de l'algbre affine $F[G^u]$ par son nil-radical) que l'on identifie au centralisateur de $u$ dans $G$ au sens de Borel \cite{B}. 
Cette inclusion est une ŽgalitŽ si et seulement si le morphisme 
$$\pi_u:G \rightarrow {\rm Int}_G(u),\, g \mapsto gug^{-1}$$ est sŽparable, auquel cas on dit que $u$ est {\it sŽparable} (cf. \cite[2.1]{L}). Pour 
$u\in \UF$ et $\lambda \in \Lambda_{F,u}^{\rm opt}$, puisque $\lambda$ appartient ˆ $ \Lambda_{F^{\rm s\acute{e}p},u}^{\rm opt}$ (\ref{descente sŽparable}), on a toujours l'inclusion 
$$G^u(F^{\rm s\acute{e}p})\subset P_{F^{\rm s\acute{e}p},u} =P_\lambda(F^{\rm s\acute{e}p})\ptf$$ Si de plus $u$ est sŽparable, alors $G^u= (G^u)^{\mathrm{r\acute{e}d}}$ (cf. \cite[ch.~II, 6.7]{B}), 
i.e. $G^u$ \guill{est} un groupe algŽbrique linŽaire, et 
on a l'inclusion $G^u\subset P_\lambda$ (d'aprs \cite[ch.~AG, 13.3]{B}). Posons 
$$\mathfrak{p}_{F,u}= \mathfrak{p}_\lambda(F)\;(={\rm Lie}({_FP_u})(F))\ptf$$
Observons que si $u$ est sŽparable, alors $\mathfrak{g}^u ={\rm Lie}(G^u)\subset \mathfrak{p}_\lambda$ et donc 
$\mathfrak{g}^u(F) \subset \mathfrak{p}_{F,u}$. 

\begin{remark}
{\rm Nous ne savons pas si l'inclusion $\mathfrak{g}^u(F)\subset \mathfrak{p}_{F,u}$ est vraie en gŽnŽral, \cad pour les ŽlŽments $u$ qui ne sont pas sŽparables; de toutes faons nous n'aurons besoin 
que de la version graduŽe \ref{proposition noyau} de cette inclusion. En revanche nous savons que pour $Y \in \NF\smallsetminus \{0\}$, le centralisateur
$$\mathfrak{g}^Y(F) = \{X\in \mathfrak{g}(F)\,\vert \, [Y,X]=0\}$$ de $Y$ dans $\mathfrak{g}(F)$ peut ne pas tre contenu dans $\mathfrak{p}_{F,Y} ={\rm Lie}({_FP_Y})(F)$ 
(cf. \ref{contre-exemple}).
}
\end{remark}

D'aprs les relations de commutateurs de Chevalley, 
pour $\lambda \in \check{X}_F(G)$ et $x\in G_{\lambda ,k}$ avec $k\in \mbb{N}$, le morphisme linŽaire ${\rm Ad}_x -1: \mathfrak{g} \rightarrow \mathfrak{g}$ envoie 
$\mathfrak{g}_{\lambda,i}$ dans $\mathfrak{g}_{\lambda,k+i}$ pour tout $i\in \mbb{Z} $. Par restriction et passage au quotient, il induit donc, pour chaque $i\in \mbb{Z}$, un 
morphisme linŽaire $$\eta_{\lambda, x}(i): \mathfrak{g}_{\lambda}(i) \rightarrow \mathfrak{g}_{\lambda, k+i}/ \mathfrak{g}_{\lambda, k+i+1}= \mathfrak{g}_\lambda(k+i)$$ qui ne dŽpend que de 
l'image $\overline{x}$ de $x$ dans $$G_\lambda(k)\bydef G_{\lambda,k}/ G_{\lambda ,k+1}\ptf$$ Si $x\in G_{\lambda,k}(F)$, ou plus gŽnŽralement si $\overline{x}\in G_\lambda(k\mathpvg F)$, 
alors $\eta_{\lambda,x}(i)$ est dŽfini sur $F$ et induit un morphisme $F$-linŽaire 
$$\eta_{\lambda,x}(i\mathpvg F) : \mathfrak{g}_{\lambda}(i\mathpvg F) \rightarrow \mathfrak{g}_{\lambda,k+i}(F)/ \mathfrak{g}_{\lambda,k+i+1}(F)= \mathfrak{g}_\lambda(k+i\mathpvg F)\ptf$$

La proposition suivante est le rŽsultat technique principal de l'article.

\begin{proposition}\label{proposition noyau}
(On suppose que l'hypothse \ref{hyp bonnes F-strates} est vŽrifiŽe pour $G$.) Soient $u \in \UF\smallsetminus \{1\}$, $\lambda\in \Lambda_{F,u}^{\rm opt}$ et $k=m_u(\lambda)>0$. 
Pour $i=1,\ldots ,k-1$, 
on a l'inclusion
$$\{X\in \mathfrak{g}_{\lambda, -i }(F)\,\vert \, {\rm Ad}_u(X)- X \in \mathfrak{g}_{\lambda,k-i+1}(F)\} \subset \mathfrak{g}_{\lambda,-i+1}(F);$$ 
en d'autres termes, le morphisme $\eta_{\lambda,x}(-i;F)$ est injectif. 
\end{proposition}

On ne dŽmontrera \ref{proposition noyau} que dans les cas particuliers suivants, le cas gŽnŽral pouvant tre considŽrŽ comme conjectural: 
\begin{itemize}
\item $p=1$ ou $p>1$ est trs bon pour $G$; 
\item les sous-groupes (absolument) quasi-simples de $G$ sont tous de type ${\bf A}_l$, ${\bf D}_l$ ou de type exceptionnel $E_*$ (de manire Žquivalente, les 
entiers $N_{\alpha,\beta}$ des formules de commutateurs de chevalley sont tous Žgaux ˆ $\pm 1$, cf. \ref{remarque sur les constantes N}); 
\item $F$ est un corps local (i.e. localement compact, non discret) ou un corps global.
\end{itemize}
Ces cas sont non exclusifs l'un de l'autre. Observons que dans le premier et le troisime cas, l'hypothse \ref{hyp bonnes F-strates} est automatiquement vŽrifiŽe. 
La dŽmonstration de \ref{proposition noyau} 
occupera la suite de \ref{ŽnoncŽ} ainsi que les sous-sections \ref{le cas o les N sont Žgaux ˆ 1 ou -1} et \ref{le cas d'un corps local}. 

\begin{remark}
\textup{
Le cas qui nous servira ici est celui o $F$ est un corps local non archimŽdien de caractŽristique $p>1$. Il sera traitŽ en \ref{le cas d'un corps local} par la mŽthode des 
corps proches, qui n'est pas gŽnŽralisable ˆ d'autres corps. En revanche la mŽthode utilisŽe dans le deuxime cas (cf. \ref{le cas o les N sont Žgaux ˆ 1 ou -1}) est sans doute gŽnŽralisable, 
du moins si les constantes $N_{\alpha,\beta}$ des formules de commutateurs de Chevalley sont toutes inversibles dans $F$ (cf. \ref{remarque sur les constantes N}\,(iii)) ; 
et sinon, elle pourrait fournir des contres-exemples. 
}
\end{remark}

Observons que l'injectivitŽ de \textit{tous} les $F$-morphismes $\eta_{\lambda,u}(-i\mathpvg F)$ ($i=1,\ldots ,k-1$) Žquivaut ˆ leur surjectivitŽ: 

\begin{lemma}\label{injectivitŽ=surjectivitŽ}
Soient $u \in \UF\smallsetminus \{1\}$, $\lambda\in \Lambda_{F,u}^{\rm opt}$ et $k=m_u(\lambda)>0$. Les conditions suivantes sont Žquivalentes:
\begin{enumerate}
\item[(i)] pour $i=1,\ldots ,k-1$, le morphisme $\eta_{\lambda,u}(-i;F)$ est injectif;
\item[(ii)] pour $i=1,\ldots ,k-1$, le morphisme $\eta_{\lambda,u}(-i;F)$ est surjectif;
\item[(iii)] pour $i=1,\ldots ,k-1$, le morphisme $\eta_{\lambda,u}(-i;F)$ est un isomorphisme.
\end{enumerate}
Si les trois conditions Žquivalentes ci-dessus sont vŽrifiŽes, on a en particulier
$$\dim_F(\mathfrak{g}_\lambda(-i;F))= \dim_F(\mathfrak{g}_\lambda(k-i;F))\quad \hbox{pour} \quad i=1,\ldots,k-1\ptf$$
\end{lemma}

\begin{proof}
Pour $i\in \mbb{Z}$, posons $d_i= \dim_F(\mathfrak{g}_\lambda(i\mathpvg F))$. Pour $i\neq 0$, notons $\ES{R}_\lambda(i)$ l'ensemble des $\alpha\in \ES{R}$ 
tels que $\langle \alpha, \lambda \rangle =i$; on a donc $d_i= \sum_{\alpha \in \ES{R}_\lambda(i)} \dim_F(\mathfrak{u}_\alpha(F))$ o $\mathfrak{u}_\alpha\subset \mathfrak{g}$ 
est le sous-espace radiciel associŽ ˆ $\alpha$. Puisque $\ES{R}_\lambda(-i)= \ES{R}_{-\lambda}(i)$ et $\dim_F(\mathfrak{u}_\alpha(F))= \dim_F(\mathfrak{u}_{-\alpha}(F))$, on a $d_{-i}=d_i$. Par consŽquent 
$$\sum_{i=1}^{k-1}d_{-i}= \sum_{i=1}^{k-1} d_{k-i}\ptf$$ 
Cela entra"ne que si l'une des deux premires conditions de l'ŽnoncŽ est vŽrifiŽe, alors l'autre l'est aussi. 
\end{proof}

La proposition suivante est la version nilpotente de \ref{proposition noyau}:

\begin{proposition}\label{proposition L(noyau)}
(On suppose que l'hypothse \ref{hyp bonnes F-strates} est vŽrifiŽe pour $\mathfrak{g}$.) Soient $\lambda \in \check{X}(A_0)$, $k\in \mbb{N}^*$ et $Y\in \mathfrak{g}_\lambda(k\mathpvg F)$ tels que 
$\lambda \in \Lambda_{F,Y}^{\rm opt}$. Pour $i=1,\ldots ,k-1$, on a 
$$\mathfrak{g}^Y(F)\cap \mathfrak{g}_\lambda(-i;F)= \{0\};$$
en d'autres termes, le morphisme $[Y,\cdot ]:\mathfrak{g}_\lambda(-i;F) \rightarrow \mathfrak{g}_{\lambda}(k-i;F)$ est injectif. 
\end{proposition}

Lˆ encore, on ne dŽmontrera pas \ref{proposition L(noyau)} en gŽnŽral mais seulement dans les cas particuliers indiquŽs plus haut. On a aussi la version nilpotente de \ref{injectivitŽ=surjectivitŽ}:

\begin{lemma}\label{L(injectivitŽ)=L(surjectivitŽ)}
Soient $\lambda \in \check{X}(A_0)$, $k\in \mbb{N}^*$ et $Y\in \mathfrak{g}_\lambda(k;F)$ tels que 
$\lambda \in \Lambda_{F,Y}^{\rm opt}$. 
Les conditions suivantes sont Žquivalentes:
\begin{enumerate}
\item[(i)] pour $i=1,\ldots ,k-1$, on a $\mathfrak{g}^Y(F)\cap \mathfrak{g}_\lambda(-i;F)= \{0\}$;
\item[(ii)] pour $i=1,\ldots ,k-1$, on a $[Y,\mathfrak{g}_\lambda(-i;F)]= \mathfrak{g}_\lambda(k-i;F)$;
\item[(iii)] pour $i=1,\ldots ,k-1$, le morphisme  $[Y,\cdot ]:\mathfrak{g}_\lambda(-i;F) \rightarrow \mathfrak{g}_{\lambda}(k-i;F)$ est un isomorphisme.
\end{enumerate}
\end{lemma}

\begin{proof}
Elle est identique ˆ celle de \ref{L(injectivitŽ)=L(surjectivitŽ)}.
\end{proof}

Pour dŽmontrer \ref{proposition noyau}, nous verrons qu'il suffit de dŽmontrer \ref{proposition L(noyau)}, ou ce qui revient au mme 
\ref{L(injectivitŽ)=L(surjectivitŽ)}\,(ii), dans le cas o $G$ est $F$-dŽployŽ. Avant cela, prouvons \ref{proposition L(noyau)} dans le cas facile suivant: 

\begin{lemma}\label{le cas sŽparable}
Soient $\lambda \in \check{X}(A_0)$, $k\in \mbb{N}^*$ et $Y\in \mathfrak{g}_\lambda(k;F)$ tels que 
$\lambda \in \Lambda_{F,Y}^{\rm opt}$. Si $Y$ est sŽparable, \cad si le morphisme 
$G\rightarrow \mathrm{Ad}_G(Y),\, g \mapsto \mathrm{Ad}_g(Y)$ est sŽparable (e.g. si $p=1$), alors 
$$\mathfrak{g}^Y(F) \cap \mathfrak{g}_\lambda(-i;F)=\{0\}\quad\hbox{pour tout}\quad i\in \mbb{N}^*\ptf$$
\end{lemma}

\begin{proof}
Supposons $Y$ sŽparable. Dans ce cas on a l'ŽgalitŽ $\mathfrak{g}^Y = {\rm Lie}(G^Y)$. 
Comme d'autre part (d'aprs \ref{descente sŽparable}) on a l'inclusion $G^Y(F^{\rm s\acute{e}p})\subset P_{F^{\rm s\acute{e}p},Y}=P_\lambda(F^{\rm s\acute{e}p})$, on en dŽduit l'inclusion 
$\mathfrak{g}^Y \subset \mathfrak{p}_\lambda$ (cf. le dŽbut de \ref{ŽnoncŽ}). D'o le lemme.
\end{proof}

Soient $u \in \UF\smallsetminus \{1\}$, $\lambda\in \Lambda_{F,u}^{\rm opt}$ et $k=m_u(\lambda)>0$. Supposons que l'on a dŽmontrŽ \ref{proposition noyau} dans le cas o le groupe $G$ est $F$-dŽployŽ. 
Soit $E/F$ une sous-extension finie de $F^{\rm s\acute{e}p}/F$ dŽployant 
$G$. Puisque $u\in \UU_E$ et $\lambda \in \Lambda^{\rm opt}_{E,u}$ (\ref{descente sŽparable}), pour $i=1,\ldots ,k-1$, le morphisme 
$$\eta_{\lambda,u}(-i;E): \mathfrak{g}_\lambda(-i;E)\rightarrow \mathfrak{g}_{\lambda,k-i}\mathpvg E)/\mathfrak{g}_{\lambda, k-i+1}(E)= \mathfrak{g}_\lambda(k-i\mathpvg E)$$ est injectif; il induit 
par restriction le morphisme $\eta_{\lambda,u}(-i;F)$, qui est lui aussi injectif. 

On peut donc supposer que $G$ est $F$-dŽployŽ. 
On peut aussi supposer $u$ en position standard et $\lambda \in \check{X}(A_0)$, \cad $
\Lambda_{F,u}^{\rm opt}\cap \check{X}(A_0)=\{\lambda\}$.  

Rappelons que le $F$-isomorphisme $j_0: \mathfrak{u}_0 \rightarrow U_0$ compatible ˆ l'action de $A_0$ a ŽtŽ construit via le choix d'ŽlŽments $E_\alpha \in \mathfrak{u}_\alpha(F) \smallsetminus \{0\}$ 
pour $\alpha \in \ES{R}$ (cf. \cite[3.4]{L}\footnote{On a supprimŽ les \guill{tildes} de loc.~cit. car ici $G$ est supposŽ $F$-dŽployŽ. Observons aussi que seuls les $E_\alpha$ pour $\alpha\in \ES{R}^+$ sont utilisŽs dans la construction de $j_0$.}); o\footnote{Si $2\alpha \notin \ES{R}_{A_0}$, on a $\mathfrak{u}_\alpha = {\rm Lie}(U_\alpha)$; si $2\alpha \in \ES{R}_{A_0}$, on a 
$\mathfrak{u}_\alpha = {\rm Lie}(U_{(\alpha)}/U_{2\alpha})$ o $U_{(\alpha)}$ est le sous-groupe unipotent de $G$ correspondant ˆ $(\alpha)=\{\alpha,2\alpha\}$.} $$\mathfrak{u}_\alpha= \{X\in \mathfrak{g}\,\vert \, {\rm Ad}_a(X)= a^\alpha X,\,\forall a \in A_0\}\ptf$$ 
On peut imposer de plus que ces $E_\alpha$ 
soient obtenus par extension des scalaires de $\mbb{Z}$ ˆ $F$ ˆ partir d'une \guill{base de Chevalley} 
$\{H_\alpha\,\vert \, \alpha \in \Delta\} \cup \{E_\alpha\,\vert \, \alpha \in \ES{R}\}$ de 
${\rm Lie}(\scrG_{\rm der})(\mbb{Z})$: 
\begin{itemize}
\item pour tous $\alpha,\beta \in \Delta$, $[H_\alpha, H_{\beta}]=0$;
\item pour tout $\alpha\in \Delta$ et tout $\beta\in \ES{R}$, $[H_\alpha ,E_\beta ] = \beta(H_\alpha) E_\beta = 2 \frac{(\beta,\alpha)}{(\alpha,\alpha)}E_\beta$;
\item pour tout $\alpha\in \ES{R}$, $[E_\alpha,E_{-\alpha}]$ est une combinaison linŽaire ˆ coefficients dans $\mathbbm{Z}$ des $H_{\beta}$ ($\beta\in \Delta$); 
\item si $\alpha,\,\beta \in \ES{R}$ sont deux racines non proportionnelles, et si $\{ j\alpha + \beta \,\vert \, j \in [-q,r]\}$ est la 
$\alpha$-cha"ne de racines dŽfinie par $\beta$, alors $$[E_\alpha,E_\beta]=\left\{\begin{array}{ll}
0 & \hbox{si $r=0$}\\
 \pm (q+1)E_{\alpha+ \beta}&\hbox{si $\alpha + \beta \in \ES{R}$} \end{array}\right..$$
 \end{itemize}
 Le point crucial est que les relations sont ˆ coefficients dans $\mbb{Z}$. On peut considŽrer l'enveloppe linŽaire de cette base 
sur n'importe quel corps (e.g. $F$ ou $\mbb{F}_p$, mais aussi 
les corps algŽbriquement clos $\mbb{C}$, $\overline{\mbb{Q}}_p$ ou $\overline{\mbb{F}}_p$); les relations ci-dessus la munissent d'une structure d'algbre de Lie sur le corps en question. 

Les relations peuvent tre trs diffŽrentes d'un corps ˆ l'autre. En 
particulier pour $\alpha, \,\beta \in \ES{R}$ tel que $\alpha + \beta \in \ES{R}$, l'entier $N_{\alpha,\beta}$ dŽfini par 
$$[E_\alpha,E_\beta]\bydef N_{\alpha,\beta} E_{\alpha+\beta}$$ peut tre nul dans $F$ 
(si $p>1$ et $p$ divise $N_{\alpha,\beta}$).

\begin{remark}\label{remarque sur les constantes N}
\textup{
\begin{enumerate}
\item[(i)]  Les entiers $N_{\alpha,\beta}$ appartiennent ˆ $\{0, \pm 1,\pm 2 , \pm 3\}$. Si de plus $\ES{R}$ n'a aucun facteur irrŽductible 
de type ${\bf G}_2$, ils appartiennent ˆ $\{0,\pm 1,\pm 2\}$. En effet ${\bf G}_2$ est le seul systme de racines (rŽduit et irrŽductible) 
pour lequel il existe deux racines $\alpha$ et $\beta$ 
telles que l'ensemble des racines de la forme $i\alpha + j\beta$ avec $i,\, j \in \mbb{N}$ soit de cardinal $>2$.
\item[(ii)]Si  $\alpha,\, \beta ,\, \gamma $ sont trois racines d'une composante irrŽductible de $\ES{R}$ 
dont toutes les racines sont de mme longueur (\cad de type ${\bf A}_l$, ${\bf D}_l$  
ou de type exceptionnel ${\bf E}_*$) telles que $\alpha + \beta = \gamma$, alors $N_{\alpha,\beta}= \pm 1$. 
\item[(iii)]Supposons que $G$ soit (absolument) quasi-simple, i.e. que $\ES{R}$ soit irrŽductible. D'aprs (ii), si $G$ est de type ${\bf A}_l$, ${\bf D}_l$ ou de type exceptionnel ${\bf E}_*$, 
tous les $N_{\alpha,\beta}$ sont Žgaux ˆ $\pm 1$, donc en particulier inversibles dans $F$. Les seuls cas o les $N_{\alpha,\beta}$ peuvent ne pas tre inversibles dans $F$ sont les suivants: 
\begin{itemize}
\item $p=2$ et $G$ est de type ${\bf B}_l$, ${\bf C}_l$, ${\bf F}_4$ ou ${\bf G}_2$;
\item $p=3$ et $G$ est de type ${\bf G}_2$.
\end{itemize}
\end{enumerate}}
\end{remark}

Posons $$Y\bydef j_0^{-1}(u)\in \mathfrak{g}_{\lambda ,k}(F)\ptf$$ 
Pour $Z\in \mathfrak{g}$ et $r\in \mbb{Z} $, notons $Z(r)$ la composante de $Z$ sur $\mathfrak{g}_\lambda(r)$. Si $Z\in \mathfrak{g}_{\lambda,-i}(F)$ pour un 
$i\in \mbb{Z}$, on a  $${\rm Ad}_u(Z)- Z \equiv [Y(k),Z(-i)] \quad ({\rm mod}\; \mathfrak{g}_{\lambda, k-i+1}(F))\ptf$$ 
D'aprs \cite[3.4.5, 2.7.7]{L}, le co-caractre $\lambda$ est $(F,Y(k))$-optimal; et puisque par hypothse $\lambda$ est 
$(\overline{F},u)$-optimal (i.e.  $\lambda\in \check{X}_F(G)\cap \Lambda_{F,u}^{\mathrm{opt}}$), il est $(\overline{F},Y(k))$-optimal. En 
remplaant $Y$ par $Y(k)$ et $X\in \mathfrak{g}_{\lambda,-i}(F)$ par $X(-i)$, on obtient 
le:

\begin{lemma}\label{rŽduction au cas dŽployŽ L}
La proposition \ref{proposition L(noyau)} dans le cas o $G$ est $F$-dŽployŽ implique la proposition \ref{proposition noyau}.
\end{lemma}

\begin{remark}
\textup{On s'est ramenŽ par descente sŽparable au cas o $G$ est $F$-dŽployŽ puis on est passŽ ˆ l'algbre de Lie gr‰ce au $F$-isomorphisme 
$j_0: \mathfrak{u}_0\rightarrow U_0$. De la mme manire (par descente sŽparable), si les ŽnoncŽs \ref{proposition L(noyau)} et 
\ref{L(injectivitŽ)=L(surjectivitŽ)} sont vrais sur $E$ pour une sous-extension (finie ou pas) $E/F$ de $F^{\rm s\acute{e}p}/F$ 
dŽployant $G$, alors ils sont vrais sur $F$. 
}\end{remark}

D'aprs \ref{rŽduction au cas dŽployŽ L} et \ref{le cas sŽparable}, si $p=1$ ou $p>1$ est \textit{trs bon} pour $G$, la proposition \ref{proposition noyau} est dŽmontrŽe. En effet dans ce cas toutes $G$-orbites nilpotentes de $\mathfrak{g}$ 
sont sŽparables: c'est Žvident si $p=1$; et si $p>1$, c'est un rŽsultat de Richardson-Springer-Steinberg (cf. Jantzen \cite[2.5, 2.6]{J}). Il nous reste donc ˆ prouver 
\ref{proposition L(noyau)} dans le cas o $p>1$ n'est pas trs bon pour $G$. On a vu que l'on peut supposer $G$ dŽployŽ sur $F$. 
On peut aussi supposer que $G$ est absolument simple de type adjoint (cf. \cite[3.2.6]{L}). 

\begin{remark}\label{contre-exemple}
\textup{
Si $p>1$ n'est pas trs bon pour $G$, la condition $i<k$ est essentielle pour que la proposition \ref{proposition noyau} soit vraie. On suppose que $G$ est $F$-dŽployŽ. 
Soient $Y= \sum_{\alpha \in \Delta} E_\alpha$ et 
$u=j_0(Y)\in U_0(F)$. Alors $u$ est un ŽlŽment unipotent rŽgulier de $G$ au sens o il est contenu dans 
un unique sous-groupe de Borel de $G$, en l'occurence $P_0$. Soit $\mu =\frac{1}{2} \sum_{\alpha \in \ES{R}^+} \check{\alpha}\in 
\check{X}(A_0)_\mbb{Q} $. D'aprs Hesselink \cite[8.6\,(b)]{H1}, $\mu$ appartient ˆ $ \bs{\Lambda}_{F,u}$. Soit $k$ 
l'unique entier $>0$ tel que $\lambda = k \mu$ appartienne ˆ $\Lambda_{F,u}^{\rm opt}$. Soit $X=\sum_{\alpha \in \Delta} c_\alpha E_{-\alpha}$ 
avec $c_\alpha \in F$. Alors $Y\in \mathfrak{g}_\lambda(k;F)$, $X\in \mathfrak{g}_{\lambda, -k}(F)$ et 
$$ {\rm Ad}_u(X)-X\equiv  \sum_{\alpha\in \Delta} c_\alpha [E_\alpha, E_{-\alpha}] \quad ({\rm mod}\; \mathfrak{g}_{\lambda,1}(F))\ptf$$ 
Pour $\alpha\in \Delta$, l'ŽlŽment $ [E_\alpha , E_{-\alpha}]$ co\"{\i}ncide avec $H_\alpha ={\rm Lie}(\check{\alpha})(1) \in {\rm Lie}(A_0)(F)$. 
L'inclusion $\check{\ES{R}} \subset \check{X}(A_0)$ donne dualement un morphisme de $\mbb{Z} $-modules 
$$\eta: X(A_0) \rightarrow {\rm Hom}_\mbb{Z} (\mbb{Z}  (\check{\ES{R}}), \mbb{Z} )\, , \;  
\chi \mapsto  (\check{\alpha} \mapsto \langle \chi, \check{\alpha} \rangle )\ptf$$ Si $p>1$ divise l'ordre du conoyau ${\rm coker} (\eta)$, 
les $H_\alpha$ ($\alpha \in \Delta$) ne sont pas linŽairement indŽpendants sur $F$ (par exemple: si $G={\rm PGL}_n$ et $p$ divise $n$). En ce cas 
il existe un $X\in \mathfrak{g}_\lambda(-k;F)\smallsetminus \{0\}$ tel que ${\rm Ad}_u(X)-X \in \mathfrak{g}_{\lambda,1}(F)$. 
}
\end{remark}

\subsection{Constructions prŽliminaires et rappels}\label{constructions prŽliminaires}
Le groupe $G$ est supposŽ $F$-dŽployŽ. 
On peut supposer que $G= \scrG\times_{\mbb{Z}}F$ et $A_0 = \scrA_0\times_{\mbb{Z}}F$ o $\scrG$ est un $\mbb{Z}$-schŽma en groupes rŽductif (connexe)  
de Chevalley-Demazure et $\scrA_0$ est un tore maximal de $\scrG$. Chaque racine $\alpha \in \ES{R}$ provient par extension 
des scalaires d'un morphisme de $\mbb{Z}$-schŽmas en groupes $\alpha: \scrA_0 \rightarrow G_{{\rm m}/\mbb{Z}}$, i.e appartient ˆ $X(\scrA_0)= \mathrm{Hom}(\scrA_0,\mbb{G}_{\mathrm{m}/\mbb{Z}})$. 
De mme, chaque co-caractre $\lambda\in \check{X}(A_0)$ appartient ˆ $ \check{X}(\scrA_0)={\rm Hom}_{\mbb{Z}}(\mathbb{G}_{{\rm m}/\mbb{Z}},\scrA_0)$. L'identification 
$\check{X}(A_0)= \check{X}(\scrA_0)$ induit une identification $\check{X}(A_0)_{\mbb{Q}}= \check{X}(\scrA_0)_{\mbb{Q}}$.

Pour $\lambda\in \check{X}(A_0)$, $i\in \mbb{N}^*$ et une $\mbb{Z}$-algbre $A$, on note $\mathfrak{g}_\lambda(i\mathpvg A)$ le $A$-module 
libre de type fini dŽfini par $$\mathfrak{g}_\lambda(i\mathpvg A)= \bigoplus_{\alpha \in \ES{R}_\lambda(i)} AE_\alpha\quad \hbox{avec}\quad \ES{R}_\lambda(i)= \{\alpha\in \ES{R}\,\vert \, \langle \alpha,\lambda\rangle = i\}\ptf$$ 
Ces $A$-modules libres de type fini dŽfinissent un espace affine $\mathfrak{g}_\lambda(i)_{\mbb{Z}}\simeq_{\mbb{Z}}\mbb{A}_{\mbb{Z}}^{n}$ avec $n=\vert \ES{R}_\lambda(i)\vert$; 
et le $F$-espace affine $\mathfrak{g}_\lambda(i)$ co\"{\i}ncide avec $\mathfrak{g}_\lambda(i)_{\mbb{Z}} \times_{\mbb{Z}} F$. 

Pour $\lambda\in \check{X}(A_0)=\check{X}(\scrA_0)$, notons $\mathscr{M}_\lambda$ le centralisateur schŽmatique de $\lambda$ dans $\mathscr{G}$. On a  
$M_\lambda = \mathscr{M}_\lambda \times_{\mbb{Z}} F$ et pour $i\in \mbb{N}^*$, l'action $F$-linŽaire de $M_\lambda$ sur $\mathfrak{g}_\lambda(i)$ provient par le 
changement de base $\mbb{Z}\rightarrow F$ d'une action $\mbb{Z}$-linŽaire de $\mathscr{M}_\lambda$ sur $\mathfrak{g}_\lambda(i)_{\mbb{Z}}$. Soit $T$ un tore maximal de $M_\lambda$. Notons 
$T^\lambda$ le sous-tore de $T$ dŽfini par $$T^\lambda = \langle \mathrm{Im}(\mu) \,\vert \, \mu \in \check{X}(T),\, (\mu,\lambda)=0 \rangle$$ et  
$M_\lambda^\perp$ le sous-groupe fermŽ distinguŽ de $M_\lambda$ dŽfini par $$ M_\lambda^\perp = \langle T^\lambda , (M_\lambda)_\mathrm{der}\rangle\ptf$$ 
Le groupe $M_\lambda^\perp$ est rŽductif connexe et il ne dŽpend pas du choix de $T$. En particulier il est dŽfini sur $F$ (on peut prendre $T=A_0$). D'aprs \cite[4.3]{CP}, 
le $F$-sous-groupe fermŽ $M_\lambda^\perp$ de $M_\lambda$ provient par le changement de base $\mbb{Z}\rightarrow F$ d'un sous-schŽma en groupes fermŽ (rŽductif connexe) 
$\scrM_\lambda^\perp $ de $\mathscr{M}_\lambda$: on a $M_\lambda^\perp= \scrM_\lambda^\perp\times_{\mbb{Z}} F$. 

Soient $\lambda\in \check{X}(A_0)$ et $k\in \mbb{N}^*$. On suppose que $\lambda$ est primitif. Pour un ŽlŽment $Y\in \mathfrak{g}_\lambda(k\mathpvg F)$, le \textit{critre de Kirwan-Ness rationnel} (cf. \cite[2.8.1]{L}) 
dit que $\lambda$ est $(F,Y)$-optimal (relativement ˆ l'action adjointe de $G(F)$) si et seulement si $Y$ est \textit{$(F,M_\lambda^\perp)$-semi-stable}, cf. \ref{gŽnŽ et descente}. 
Pour $F=\overline{F}$, on obtient le critre de Kirwan-Ness habituel (i.e. gŽomŽtrique). 

Puisque l'action $F$-linŽaire de $M_\lambda^\perp$ sur $\mathfrak{g}_\lambda(k)$ provient par le changement 
de base $\mbb{Z}\rightarrow F$ d'une action $\mbb{Z}$-linŽaire du schŽma en groupes $\scrH= \scrM_\lambda^\perp$ sur l'espace affine $\scrV=\mathfrak{g}_\lambda(k)_{\mbb{Z}}\;(\simeq_{\mbb{Z}}\mbb{A}_{\mbb{Z}}^n)$, on 
dispose aussi du thŽorme de Seshadri \cite{S}: il existe un sous-schŽma ouvert $\scrV^{\mathrm{ss}}$ de $\scrV$ 
tel que pour tout corps commutatif $\bs{k}$ algŽbriquement clos, $\scrV^{\rm ss}(\bs{k})$ soit l'ensemble des ŽlŽments 
$\scrH(\bs{k})$-semi-stables de $\scrV(\bs{k})=\scrV(\mbb{Z})\otimes_{\mbb{Z}}\bs{k}$. 
D'aprs le critre de Kirwan-Ness gŽomŽtrique, un ŽlŽment $\bs{Y}$ de $\scrV(\bs{k})$ appartient ˆ $\scrV^{\rm ss}(\bs{k})$ si et seulement si 
le co-caractre $\lambda$ --- qui est dŽfini sur $\mbb{Z}$ et donc {\it a fortiori} sur $\bs{k}$ --- est $(\bs{k},\bs{Y})$-optimal 
(relativement ˆ l'action adjointe de $\scrH(\bs{k})$). 

\subsection{Le cas o les $N_{\alpha,\beta}$ sont tous Žgaux ˆ $\pm1$ ($p>1$)}\label{le cas o les N sont Žgaux ˆ 1 ou -1} Le groupe $G$ est supposŽ $F$-dŽployŽ, avec $p>1$. 
On suppose toujours que l'hypothse \ref{hyp bonnes F-strates} est vŽrifiŽe pour $V=G$ muni de l'action par conjugaison. On suppose de plus que pour toutes les racines $\alpha,\,\beta\in \ES{R}$ telles que $\alpha+\beta\in \ES{R}$, l'entier $N_{\alpha,\beta}$ est Žgal ˆ $\pm 1$; \cad 
que tout facteur (absolument) quasi-simple de $G$ est de type ${\bf A}_l$, ${\bf D}_l$ ou de type exceptionnel ${\bf E}_*$ (cf. \ref{remarque sur les constantes N}). 

Soient $\lambda$, $k$ et $Y$ comme dans l'ŽnoncŽ de \ref{proposition L(noyau)}. On peut supposer que $Y$ est en position standard et 
que $\lambda\in \check{X}(A_0)$; autrement dit $P_\lambda \supset P_0$ et $\check{X}(A_0)\cap \Lambda_{F,Y}^{\mathrm{opt}}= \{\lambda\}$. 
Fixons un entier $i\in \{1,\ldots ,k-1\}$ et un ŽlŽment $X\in \mathfrak{g}_\lambda(-i;F)$ tel que $[X,Y]=0$. On veut prouver que $X=0$. 
On Žcrit $$X= \sum_{\alpha \in \ES{R}_\lambda(-i)}a_\alpha E_\alpha\quad \hbox{et}\quad 
Y= \sum_{\beta \in \ES{R}_\lambda(k)}b_\beta E_\beta\,.$$ On note $\ES{R}(X)$ le sous-ensemble de $\ES{R}_\lambda(-i)$ 
formŽ des racines $\alpha$ telles que $a_\alpha \neq 0$. On dŽfinit de manire analogue le sous-ensemble $\ES{R}(Y)$ de $\ES{R}_\lambda(k)$.
Pour $\gamma\in \ES{R}$, on note $C(\gamma)= C(X,Y\mathpvg \gamma)$ l'ensemble des $(\alpha,\beta)\in \ES{R}(X)\times \ES{R}(Y) $ 
tels que $\alpha + \beta = \gamma$. Pour que $C(\gamma)$ soit non vide, il faut bien sžr que $\gamma$ appartienne ˆ $\ES{R}_\lambda(k-i)$. 
On a donc 
$$[X,Y]= \sum_{\gamma \in \ES{R}_\lambda(k-i)} 
\left(\sum_{(\alpha,\beta)\in C(\gamma)} c_{\alpha,\beta} N_{\alpha,\beta} \right) 
E_\gamma\quad \hbox{avec} \quad c_{\alpha,\beta}=a_\alpha b_\beta\in F^\times\,.$$ 
Puisque $[X,Y]=0$, on a $$\sum_{(\alpha,\beta)\in C(\gamma)} c_{\alpha,\beta} N_{\alpha,\beta}=0
\quad\hbox{pour tout}\quad \gamma\in \ES{R}_\lambda(k-i)\,.$$

\begin{lemma}\label{si C=0 alors X=0}
Si tous les ensembles $C(\gamma)$ sont vides, alors $X=0$.
\end{lemma}

\begin{proof}
Supposons que tous les ensembles $C(\gamma)$ soient vides. Alors 
$(\alpha, \beta) \geq 0$ pour tout $(\alpha,\beta)\in \ES{R}(X)\times \ES{R}(Y)$. Si $X\neq 0$, on choisit une racine $\alpha\in \ES{R}(X)$, 
que l'on identifie ˆ un ŽlŽment $\eta\in \check{X}(A_0)_{\mathbb{Q}}$. On a alors  
$$\left\{\begin{array}{ll}
\langle \beta , \eta \rangle  =(\beta,\alpha)\geq 0 & \hbox{pour tout $\beta \in \ES{R}(Y)$}\\
(\lambda , \eta) = -i  <0 & 
\end{array}\right..$$
Posons $\mu = \wt{\lambda} + \frac{1}{a}\eta \in \check{X}(A_0)_\mbb{Q} $ avec $\wt{\lambda}= \frac{1}{k}\lambda\in \bs{\Lambda}_{F,Y}$ et $a\in \mbb{N}^*$. 
Alors $$\langle \beta , \mu \rangle = \langle \beta , \wt{\lambda}\rangle  + \frac{1}{a} \langle \beta , \eta \rangle \geq \langle \beta, \wt{\lambda} \rangle = 1
\quad \hbox{pour toute racine} \quad \beta \in \ES{R}(Y)$$ et 
$$(\mu,\mu) = (\wt{\lambda},\wt{\lambda}) + \frac{2}{a} ( \wt{\lambda}, \eta ) + \frac{1}{a^2} (\eta,\eta) < (\wt{\lambda},\wt{\lambda}) \quad \hbox{si} \quad 
\hbox{$a$ est suffisamment grand}\ptf$$ 
Ainsi pour $a\gg 1$, on a $m_{Y}(\mu)\geq 1$ et $(\mu,\mu) < (\wt{\lambda},\wt{\lambda})$, 
ce qui contredit la minimalitŽ de $\| \wt{\lambda}\|$ et donc le fait que $\wt{\lambda}$ appartient $\bs{\Lambda}_{F,Y}$. Donc $X=0$.
\end{proof}

Il suffit donc de prouver que tous les ensembles $C(\gamma)$ sont vides. 
Supposons  par l'absurde qu'il existe un $\gamma\in \ES{R}_\lambda(k-i)$ 
tel que l'ensemble $C(\gamma)$ ne soit pas vide. 
Quitte ˆ remplacer $X$ par $c^{-1}X$ pour une constante $c\in F^\times$, 
on peut supposer que $c_{\alpha,\beta}=1$ pour au moins un couple $(\alpha,\beta)\in C(\gamma)$. En choisissant une base du sous-$\mbb{F}_p$-espace 
vectoriel de $F$ engendrŽ par les $c_{\alpha,\beta}$ pour $(\alpha,\beta)\in C(\gamma)$, on obtient une ŽgalitŽ non triviale dans $\mbb{F}_p$ de la forme
$$\sum_{(\alpha,\beta)\in C'(\gamma)}c_{\alpha,\beta}N_{\alpha,\beta}=0\quad \hbox{avec} \quad c_{\alpha,\beta}\in \mbb{F}_p\leqno{(1)}$$ 
pour un sous-ensemble non vide $C'(\gamma)\subset C(\gamma)$. Puisque par hypothse les $N_{\alpha,\beta}$ sont tous non nuls dans $\mbb{F}_p= \mbb{Z}/p\mbb{Z}$, l'ensemble $C'(\gamma)$ est de cardinal 
$\geq 2$. Pour chaque couple $(\alpha,\beta)\in C'(\gamma)$, choisissons un relvement arbitraire $c'_{\alpha,\beta}$ 
de $c_{\alpha,\beta}$ dans $\mbb{Z}$. On a $$\sum_{(\alpha,\beta)\in C'(\gamma)}c'_{\alpha,\beta}N_{\alpha,\beta}\equiv 0\quad(\mathrm{mod}\; p\mbb{Z})\ptf$$ Comme  
tous les $N_{\alpha,\beta}$ sont Žgaux ˆ $\pm 1$, il suffit de modifier l'un des relvements $c'_{\alpha,\beta}$ pour avoir une ŽgalitŽ non triviale dans $\mbb{Z}$ 
$$\sum_{(\alpha,\beta)\in C'(\gamma)}c'_{\alpha,\beta}N_{\alpha,\beta}=0\quad \hbox{avec} \quad c'_{\alpha,\beta}\in \mbb{Z}\ptf\leqno{(2)}$$

On va modifier l'ŽgalitŽ (2) et en dŽduire une identitŽ analogue sur $\mbb{C}$, suivant le principe qui consiste ˆ ramener la thŽorie des ŽlŽments unipotents, resp. nilpotents, 
su un corps algŽbriquement clos de caractŽristique $p>1$ ˆ celle sur $\mbb{C}$ (cf. \cite{Lu,CP}). 
Pour cela on va utiliser l'hypothse $\check{X}_F(G)_{\mbb{Q}}\cap \bs{\Lambda}_Y \neq \emptyset$ (qui permet de passer de $F$ ˆ $\overline{F}$) 
et le thŽorme de Seshadri (qui permet de passer de $\overline{F}$ ˆ $\mbb{C}$).  

Comme en \ref{constructions prŽliminaires}, posons $\scrH= \scrM_\lambda^\perp$ et $\scrV= \mathfrak{g}_\lambda(k)_{\mbb{Z}}$, et notons 
$\scrV^{\mathrm{ss}}$ la sous-variŽtŽ ouverte de $\scrV$ telle que pour tout corps commutatif $\bs{k}$ algŽbriquement clos, $\scrV^{\mathrm{ss}}(\bs{k})$ soit l'ensemble 
des ŽlŽments $\scrH(\bs{k})$-semi-stable de $\scrV(\bs{k})$ (relativement ˆ l'action adjointe de $\scrH(\bs{k})$). Puisque par hypothse $\check{X}_F(G)\cap \bs{\Lambda}_{Y}\neq \emptyset$, le co-caractre 
$\lambda$ est $(\overline{F},Y)$-optimal et d'aprs le critre de Kirwan-Ness gŽomŽtrique, l'ŽlŽment $Y$ appartient ˆ $\scrV^{\mathrm{ss}}(\overline{F})$. 
En particulier $\scrV_{\mathrm{ss}}(\overline{F})$ est non vide. Cela assure que $\scrV_{\mathrm{ss}}(\mbb{C})$ est non vide, donc (ouvert) dense dans $\scrV(\mbb{C})$ pour la topologie de Zariski. 

Pour une famille $z=(z_\beta)_{\beta \in \ES{R}_\lambda(k)}$ d'ŽlŽments de $\mbb{C}^\times$, 
notons $\bs{Y}_{z}\in \mathfrak{g}_\lambda(k\mathpvg \mbb{C}) = \scrV(\mbb{C})$ et 
$\bs{X}_{z}\in \mathfrak{g}_\lambda(-i\mathpvg \mbb{C})$ 
les ŽlŽments dŽfinis par $$\bs{Y}_{z}= \sum_{\beta \in \ES{R}_\lambda(k)} z_\beta E_\beta
\quad \hbox{et}\quad \bs{X}_{z}= \sum_{(\alpha,\beta)\in C'(\gamma)}z_\beta^{-1} c'_{\alpha,\beta}  E_\alpha\,.$$ 
Par construction, on a $$[\bs{X}_{z},\bs{Y}_{z}]=\left(\sum_{(\alpha,\beta)\in C'(\gamma)} c'_{\alpha,\beta}N_{\alpha ,\beta}\right) E_\gamma =0\ptf$$ 
L'ensemble des $\bs{Y}_{z}$ pour $z \in (\mbb{C}^\times)^{\vert \ES{R}_\lambda(k)\vert}$ est 
un ouvert de Zariski de $\scrV(\mbb{C})$. Puisque $\scrV^{\mathrm{ss}}(\mbb{C})$ est dense dans $\scrV(\mbb{C})$, il contient un tel $\bs{Y}_z$. 
D'aprs le critre de Kirwan-Ness gŽomŽtrique, le co-caractre $\lambda$ est $(\mbb{C},\bs{Y}_{\!z})$-optimal (relativement ˆ l'action adjointe de $\scrH(\mbb{C})$). Par consŽquent $\bs{X}_z=0$ (d'aprs \ref{le cas sŽparable}); contradiction. 
Donc tous les ensembles $C(\gamma)$ sont vides et $X=0$ (d'aprs \ref{si C=0 alors X=0}). 

Cela achve la dŽmonstration de \ref{proposition L(noyau)} dans le cas o $G$ est $F$-dŽployŽ et tous les entiers $N_{\alpha,\beta}$ sont Žgaux ˆ $\pm 1$ (sous l'hypothse \ref{hyp bonnes F-strates}). 
On conclut gr‰ce ˆ \ref{rŽduction au cas dŽployŽ L}.

\subsection{Le cas o $F$ est un corps local, resp. global ($p>1$)}\label{le cas d'un corps local} 
On suppose toujours que $G$ est $F$-dŽployŽ, avec $p>1$. On suppose de plus que $F= \mbb{F}_q((\varpi))$, avec $q=p^{r}$. D'aprs \cite{L}, l'hypothse \ref{hyp bonnes F-strates} est vŽrifiŽe pour 
$V=G$ muni de l'action par conjugaison et pour $V=\mathfrak{g}$ muni de l'action adjointe. Dans cette sous-section, on munit $G(F)$ et $\mathfrak{g}(F)$ de la topologie 
dŽfinie par $F$ ($\mathrm{Top}_F$).

Soient $\lambda$, $k$ et $Y$ comme dans l'ŽnoncŽ de \ref{proposition L(noyau)}. Comme en \ref{le cas o les N sont Žgaux ˆ 1 ou -1}, On peut supposer que $Y$ est en position standard et 
que $\lambda\in \check{X}(A_0)$. Comme en \ref{constructions prŽliminaires}, posons $\scrH= \scrM_\lambda^\perp$ et $\scrV= \mathfrak{g}_\lambda(k)_{\mbb{Z}}$. 
Quitte ˆ multiplier $Y$ par un ŽlŽment de $F^\times$, on peut aussi supposer que $Y\in \scrV(\mathfrak{o}_F)\smallsetminus \mathfrak{p}_F \scrV(\mathfrak{o}_F)$. 

Posons $H=M_\lambda^\perp\;(= \scrH\times_{\mbb{Z}}\mbb{Z})$ et $V=\mathfrak{g}_\lambda(k)\;(= \scrV\times_{\mbb{Z}}F)$. 
Puisque le co-caractre $\lambda$ est $(F,Y)$-optimal, d'aprs le critre de Kirwan-Ness rationnel (cf. \ref{constructions prŽliminaires}), 
l'ŽlŽment $Y$ est $(F,M_\lambda^\perp)$-semi-stable, i.e. il appartient ˆ $V(F)\smallsetminus \ES{N}_F^N(V,0)$. 

Posons $P_0^H=H\cap P_0$. C'est un $F$-sous-groupe de Borel de $H$, et en notant $K^H$ le sous-groupe compact maximal (pour 
$\mathrm{Top}_F$) $\scrH(\mathfrak{o}_F)$ de $H(F)$, on a la dŽcomposition d'Iwasawa $$H(F)= K^HP_0^H(F)=P_0^H(F)K^H\ptf$$ 
On en dŽduit que pour un ŽlŽment $Y_1\in V(F)$, les deux conditions suivantes sont Žquivalentes (cf. la preuve de \cite[2.8.7]{L}):
\begin{itemize}
\item $Y_1$ est $(F,H)$-semi-stable, i.e. $Y_1\in V(F)\smallsetminus \ES{N}^H_F(V,0)$;
\item pour tout co-caractre virtuel $\mu\in \check{X}(A_0)_{\mbb{Q}}$, la $K^H$-orbite $$K^H\cdot Y_1= \{\mathrm{Ad}_k(Y_1)\,\vert \, k\in K^H\}$$ ne rencontre pas $V_{\mu,1}(F)$. 
\end{itemize}

Puisque $Y$ appartient ˆ $V(F)\smallsetminus \ES{N}_F^H(V,0)$ et $\ES{N}^H_F(V,0)$ est $\mathrm{Top}_F$-fermŽ dans $V(F)$, 
il existe un plus petit entier $m\geq 1$ tel que l'ensemble $Y + \mathfrak{p}_F^m \scrV(\mathfrak{o}_F)$ soit contenu dans $V(F)\smallsetminus \ES{N}_F^H(V,0)$. Par consŽquent l'ensemble 
$$K^H\cdot (Y'+ \mathfrak{p}_F^m \scrV(\mathfrak{o}_F))= K^H\cdot Y + \mathfrak{p}_F^m \scrV(\mathfrak{o}_F)$$ est contenu dans $V(F)\smallsetminus \ES{N}_F^H(V,0)$. On a donc 
$$\left(K^H\cdot Y + \mathfrak{p}_F^m\scrV(\mathfrak{o}_F) \right)\cap V_{\mu,1}(F)= \emptyset \quad \hbox{pour tout} \quad \mu\in \check{X}(A_0)_{\mbb{Q}}\ptf$$ 
Pour $\mu\in \check{X}(A_0)_{\mbb{Q}}= \check{X}(\scrA_0)_{\mbb{Q}}$, le $F$-espace affine $V_{\mu,1}$ provient par le changement 
de base $\mbb{Z}\rightarrow F$ d'un sous-$\mbb{Z}$-schŽma fermŽ $\scrV_{\mu,1}\simeq_{\mbb{Z}}\mbb{A}_{\mbb{Z}}^{r}$ de $\scrV$. 
Observons que puisque $Y$ appartient ˆ $\scrV(\mathfrak{o}_F)$, l'ensemble $K_H\cdot Y + \mathfrak{p}_F^m\scrV(\mathfrak{o}_F)$ est contenu dans $\scrV(\mathfrak{o}_F)$; 
d'autre part pour tout $\mu\in \check{X}(A_0)$, on a l'ŽgalitŽ $\scrV(\mathfrak{o}_F)\cap V_{\mu,1}(F)= \scrV_{\mu,1}(\mathfrak{o}_F)$.  
Notons $\mathfrak{o}_{F,m}$ l'anneau tronquŽ $ \mathfrak{o}_F/ \mathfrak{p}_F^m$ et 
$\pi_{F,m}: \scrV(\mathfrak{o}_F) \rightarrow \scrV(\mathfrak{o}_{F,m})$ la projection naturelle. Elle est $K^H$-Žquivariante pour l'action de $K_H$ sur $\scrV(\mathfrak{o}_{F,m})$ donnŽe 
par la projection naturelle $K^H \rightarrow \scrH(\mathfrak{o}_{F,m})$. On obtient les ŽgalitŽs dans $\scrV(\mathfrak{o}_{F,m})$\footnote{Ces ŽgalitŽs permettent de dŽmontrer   
l'hypothse \ref{hyp bonnes F-strates} dans le cas de $\mathfrak{g}$ muni de l'action adjointe, cf. la preuve de \cite[2.8.7]{L}.}: 
$$\scrH(\mathfrak{o}_{F,m})\cdot \pi_{F,m}(Y) \cap \scrV_{\mu,1}(\mathfrak{o}_{F,m})=\emptyset \quad\hbox{pour tout}\quad \mu \in \check{X}(A_0)_{\mbb{Q}}\ptf$$ 

Soit $F'/\mbb{Q}_p$ une extension finie telle que les anneaux tronquŽs $\mathfrak{o}_{F'\!,m}$ et $\mathfrak{o}_{F,m}$ soient isomorphes, \cad telle que 
le corps rŽsiduel $\mathfrak{o}_{F'}/\mathfrak{p}_{F'}$ soit Žgal ˆ $\mbb{F}_q$ et l'indice de ramification de $F'/\mbb{Q}_p$ soit $\geq m$;  on dit alors que les corps $F$ et $F'$ sont \textit{$m$-proches}. 
Fixons un tel isomorphisme
$$\phi: \mathfrak{o}_{F,m}\buildrel\simeq\over{\longrightarrow} \mathfrak{o}_{F'\!,m}\ptf$$
En pratique, on considre la section de Teichm\"{u}ller (l'unique section multiplicative) 
$s:\mbb{F}_q\rightarrow \mathfrak{o}_{F'}$ de la projection naturelle $\mathfrak{o}_{F'}\rightarrow \mathfrak{o}_{F'}/\mathfrak{p}_{F'}=\mbb{F}_q$. 
Observons que $s(\mbb{F}_q)$ est contenu dans l'anneau des entiers $\mathfrak{o}_{F'_0}\simeq \mathrm{W}(\mbb{F}_q)$ de la sous-extension non ramifiŽe maximale $F'_0/\mbb{Q}_p$ de $F'/\mbb{Q}_p$. 
La section $s$ est additive modulo $p$: pour tous $x,\,y \in \mbb{F}_q$, on a $s(x+y)\equiv s(x) + s(y)\; (\mathrm{mod}\,p\mathfrak{o}_{F'})$. 
Choisissons une uniformisante $\varpi'$ de $F'$ et notons $\tilde{s}: \mbb{F}_q[[\varpi]] \rightarrow \mathfrak{o}_{F'}$ l'unique application multiplicative prolongeant $s$ et vŽrifiant 
$\tilde{s}(\varpi)= \varpi'$. Puisque $p\mathfrak{o}_{F'}\subset \mathfrak{p}_{F'}^m$, 
elle induit par passage aux quotients un isomorphisme d'anneaux $\phi=\phi_{s,m}(\varpi,\varpi'):\mathfrak{o}_{F,m} \rightarrow \mathfrak{o}_{F'\!,m}$.

Par transport de structure, $\phi$ induit pour tout $\mbb{Z}$-schŽma $\ES{Z}$ une application bijective $$\phi_{\ES{Z}}: \ES{Z}(\mathfrak{o}_{F,m})\rightarrow \ES{Z}(\mathfrak{o}_{F'\!,m})\ptf$$ 
Ainsi $\phi_{\scrV}: \scrV(\mathfrak{o}_{F,m})\rightarrow \scrV(\mathfrak{o}_{F'\!,m})$ est un isomorphisme de $\mathfrak{o}_{F,m}$-modules pour l'action de $\mathfrak{o}_{F,m}$ sur $\scrV(\mathfrak{o}_{F'\!,m})$ donnŽe 
par $\phi$, qui vŽrifie les propriŽtŽs: 
\begin{itemize}
\item $\phi_{\scrV}$ est $\scrH(\mathfrak{o}_{F,m})$-Žquivariante pour l'action de $\scrH(\mathfrak{o}_{F,m})$ sur $\scrV(\mathfrak{o}_{F'\!,m})$ donnŽe par $\phi_{\scrH}$; 
\item pour tout $\mu\in \check{X}(\scrA_0)_{\mbb{Q}}$, la restriction 
de $\phi_{\scrV}$ ˆ $\scrV_{\mu,1}$ co\"{\i}ncide avec $\phi_{\scrV_{\mu,1}}$. 
\end{itemize}
Soit $Y'\in \scrV(\mathfrak{o}_{F'})$ tel que 
$$\pi_{F'\!,m}(Y') = \phi_{\scrV}\circ \pi_{F,m}(Y)\ptf$$ D'aprs ce qui prŽcde, on a 
$$\scrH(\mathfrak{o}_{F'\!,m})\cdot \pi_{F'\!,m}(Y') \cap \scrV_{\mu,1}(\mathfrak{o}_{F'\!,m})=\emptyset \quad \hbox{pour tout} \quad \mu\in\check{X}(\scrA_0)_{\mbb{Q}}\ptf $$ 
Posons $V'= \scrV\times_{\mbb{Z}}F'$, $H'= \scrH\times_{\mbb{Z}}F'$, etc. On a prouvŽ que pour tout co-caractre virtuel $\mu\in \check{X}(A'_0)_{\mbb{Q}}= \check{X}(\scrA_0)_{\mbb{Q}}$, la $K^{H'}$-orbite 
$K^{H'}\cdot Y' $ ne rencontre pas $V'_{\mu,1}(F')$, ce qui assure que l'ŽlŽment $Y'$ 
est $(F'\!,H')$-semi-stable, i.e. $Y'\in V'(F')\smallsetminus \ES{N}_{F'}^{H'}(V'\!,0)$. 

D'aprs le critre de Kirwan-Ness rationnel, 
le co-caractre $\lambda\in \check{X}(A'_0)= \check{X}(\scrA_0)$ est $(F'\!,Y')$-optimal (relativement ˆ l'action adjointe de $\scrG(F')$). On peut donc appliquer \ref{le cas sŽparable}: 
pour $i\in \mbb{N}^*$, on a $\mathrm{Lie}(G')^{Y'}(F') \cap \mathfrak{g}_\lambda(-i;F')=\{0\}$. On en dŽduit (d'aprs \ref{L(injectivitŽ)=L(surjectivitŽ)}) que pour $i=1,\ldots ,k-1$, on a 
$[Y'\!, \mathfrak{g}_\lambda(-i\mathpvg F')]= \mathfrak{g}_\lambda(k-i\mathpvg F')$. 

Fixons un entier $i\in \{1,\ldots ,k-1\}$. Puisque $[Y'\!, \mathfrak{g}_\lambda(-i\mathpvg F')]= \mathfrak{g}_\lambda(k-i\mathpvg F')$, l'image de $\mathfrak{g}_\lambda(-i\mathpvg \mathfrak{o}_{F'})$ par 
l'application $[Y'\!,\cdot]$ est un $\mathfrak{o}_{F'}$-rŽseau de $\mathfrak{g}_\lambda(k-i\mathpvg F')$ contenu dans 
$\mathfrak{g}_\lambda(k-i\mathpvg \mathfrak{o}_{F'})$, disons $\Omega'$. Soient:
\begin{itemize}
\item $\Omega'_m$ l'image de 
$\Omega'$ dans $\mathfrak{g}_\lambda(k-1\mathpvg \mathfrak{o}_{F'\!,m})= \mathfrak{g}_\lambda(k-1\mathpvg \mathfrak{o}_{F'})/\mathfrak{p}_{F'}^m\mathfrak{g}_\lambda(k-1\mathpvg \mathfrak{o}_{F'})$ par 
la projection naturelle $$\mathfrak{g}_\lambda(k-i\mathpvg \mathfrak{o}_{F'})\rightarrow \mathfrak{g}_\lambda(k-i\mathpvg \mathfrak{o}_{F'\!,m});$$ 
\item $\Omega_m$ l'image rŽciproque de $\Omega'_m$ par 
l'isomorphisme de $\mathfrak{o}_{F,m}$-modules $$\phi_{\mathfrak{g}_{\lambda}(k-i)_{\mbb{Z}}}: 
\mathfrak{g}_\lambda(k-i\mathpvg \mathfrak{o}_{F,m}) \rightarrow \mathfrak{g}_\lambda(k-i\mathpvg \mathfrak{o}_{F'\!,m}) ;$$ 
\item $\Omega$ la prŽ-image de $\Omega_m$ dans $ \mathfrak{g}_\lambda(k-i\mathpvg \mathfrak{o}_F)$ par 
la projection naturelle $$\mathfrak{g}_\lambda(k-i\mathpvg \mathfrak{o}_{F})\rightarrow \mathfrak{g}_\lambda(k-i\mathpvg \mathfrak{o}_{F,m})\ptf$$
\end{itemize}Par construction, $\Omega$ est un $\mathfrak{o}_F$-rŽseau de $\mathfrak{g}_\lambda(k-i\mathpvg F)$ qui vŽrifie 
$$\mathfrak{p}_{F}^m\mathfrak{g}_{\lambda}(k-i\mathpvg \mathfrak{o}_F) \subset \Omega \subset \mathfrak{g}_{\lambda}(k-i\mathpvg \mathfrak{o}_F)$$ 
et $$[Y,\mathfrak{g}_\lambda(-i\mathpvg \mathfrak{o}_F)] + \mathfrak{p}_F^m\mathfrak{g}_{\lambda}(k-i\mathpvg \mathfrak{o}_F) = \Omega\ptf$$

Soit maintenant $e$ un entier $\geq 2$ et soit $F''/F'$ l'extension totalement ramifiŽe de degrŽ $e$ dŽfinie par le polyn™me d'Eisenstein $T^e -\varpi' \in F'[T]$. 
Les corps $F$ et $F''$ sont $me$-proches et l'on peut choisir un isomorphisme d'anneaux 
$\psi: \mathfrak{o}_{F,me}\rightarrow \mathfrak{o}_{F''\!,me}$  qui rend commutatif le diagramme 
$$\xymatrix{ \mathfrak{o}_{F,m} \ar[d] \ar[r]^\phi & \mathfrak{o}_{F'\!,m}\ar[d]\\
\mathfrak{o}_{F,me} \ar[r]_\psi & \mathfrak{o}_{F''\!,me}
}$$
o les flches verticales sont dŽfinies comme suit:
\begin{itemize}
\item celle de gauche est dŽduite par passage aux quotients de l'endomor\-phisme de l'anneau $\mathfrak{o}_F$ qui est l'identitŽ sur $\mbb{F}_q$ et envoie $\varpi$ sur $\varpi^e$;
\item celle de droite est dŽduite par passage aux quotients de l'inclusion $\mathfrak{o}_{F'}\subset \mathfrak{o}_{F''}$.
\end{itemize}
En effet, il suffit de construire $\psi$ comme on l'a fait pour $\phi= \phi_{s,m}(\varpi,\varpi')$ ˆ l'aide d'une uniformisante $\varpi''$ de $F''$ vŽrifiant $\varpi''^e = \varpi'$: 
on prend $\psi= \phi_{s,me}(\varpi,\varpi'')$.

Puisque l'extension $F''/F'$ est sŽparable, le co-caractre $\lambda$ est $(F''\!,Y')$-optimal; et comme sur $F'$, on a l'ŽgalitŽ 
$[Y'\!, \mathfrak{g}_\lambda(-i\mathpvg F'')]= \mathfrak{g}_\lambda(k-i\mathpvg F'']$. Notons $\Omega''$ le $\mathfrak{o}_{F''}$-rŽseau de $\mathfrak{g}_\lambda(k-i\mathpvg F'')$ image de $\mathfrak{g}_\lambda(-i\mathpvg \mathfrak{o}_{F''})$ par l'application $[Y'\!,\cdot]$. 
Puisque $\mathfrak{g}_\lambda(-i\mathpvg \mathfrak{o}_{F''})= \mathfrak{g}_\lambda(-i\mathpvg \mathfrak{o}_{F'})\otimes_{\mathfrak{o}_{F'}}\mathfrak{o}_{F''}$ et $\mathfrak{g}_\lambda(k-i\mathpvg \mathfrak{o}_{F''})= \mathfrak{g}_\lambda(k-i\mathpvg \mathfrak{o}_{F'})\otimes_{\mathfrak{o}_{F'}}\mathfrak{o}_{F''}$, on a l'ŽgalitŽ 
$$\Omega'' = \Omega' \otimes_{\mathfrak{o}_{F'}}\mathfrak{o}_{F''}\ptf$$
Soit $\Omega''_m$ l'image de 
$\Omega''$ dans $\mathfrak{g}_\lambda(k-1\mathpvg \mathfrak{o}_{F''\!,me})= \mathfrak{g}_\lambda(k-1\mathpvg \mathfrak{o}_{F'e})/\mathfrak{p}_{F''}^{me}\mathfrak{g}_\lambda(k-1\mathpvg \mathfrak{o}_{F''})$ par 
la projection naturelle $\mathfrak{g}_\lambda(k-i\mathpvg \mathfrak{o}_{F''})\rightarrow \mathfrak{g}_\lambda(k-i\mathpvg \mathfrak{o}_{F''\!,me})$.  
On a $$\Omega''_{me}= \Omega''\otimes_{\mathfrak{o}_{F''}}\mathfrak{o}_{F''\!,me}= (\Omega'\otimes_{\mathfrak{o}_{F'}}\mathfrak{o}_{F''})\otimes_{\mathfrak{o}_{F''}}\mathfrak{o}_{F''\!,me}
=\Omega \otimes_{\mathfrak{o}_{F'}}\mathfrak{o}_{F''\!,me}\ptf$$ 
Comme $\mathfrak{o}_{F''\!,me}= \mathfrak{o}_{F'\!,m}\otimes_{\mathfrak{o}_{F'}}\mathfrak{o}_{F''}$, on obtient l'ŽgalitŽ 
$$\Omega''_{me}= \Omega'_m \otimes_{\mathfrak{o}_{F'}}\mathfrak{o}_{F''}\ptf$$ 
On en dŽduit que l'image rŽciproque de $\Omega''_{me}$ par 
l'isomorphisme de $\mathfrak{o}_{F,me}$-modules $$\psi_{\mathfrak{g}_{\lambda}(k-i)_{\mbb{Z}}}: 
\mathfrak{g}_\lambda(k-i\mathpvg \mathfrak{o}_{F,me}) \rightarrow \mathfrak{g}_\lambda(k-i\mathpvg \mathfrak{o}_{F''\!,me})$$ est 
$$\Omega\otimes_{\mathfrak{o}_F}\mathfrak{o}_{F,me}\ptf$$ 
Par consŽquent $$[Y,\mathfrak{g}_\lambda(-i\mathpvg \mathfrak{o}_F)] + \mathfrak{p}_F^{me}\mathfrak{g}_{\lambda}(k-i\mathpvg \mathfrak{o}_F) = \Omega\ptf$$ 
Comme l'ŽgalitŽ ci-dessus est vraie pour tout entier $e\geq 1$, par passage ˆ la limite on obtient que
$$[Y,\mathfrak{g}_\lambda(-i\mathpvg \mathfrak{o}_F)] = \Omega \ptf$$ 
D'o l'ŽgalitŽ cherchŽe: $$[Y, \mathfrak{g}_\lambda(-i\mathpvg F)]= [Y,\mathfrak{g}_\lambda(-i\mathpvg \mathfrak{o}_F)] \otimes_{\mathfrak{o}_F} F= \Omega\otimes_{\mathfrak{o}_F} F= \mathfrak{g}_\lambda(k-i\mathpvg F)\ptf$$

Compte-tenu de \ref{L(injectivitŽ)=L(surjectivitŽ)}, cela achve la dŽmonstration de \ref{proposition L(noyau)} dans le cas o $G$ est $F$-dŽployŽ et o 
$F$ est un corps local non archimŽdien de caractŽristique $p>1$. 
On conclut gr‰ce ˆ \ref{rŽduction au cas dŽployŽ L}.

Supposons maintenant que $F$ soit un corps global de caractŽristique $p>1$. Soient $\lambda$, $k$ et $u$ comme dans l'ŽnoncŽ de \ref{proposition noyau}. 
Soit $v$ une place de $F$. Le complŽtŽ $F_v$ de $F$ en $v$ est une extension sŽparable de $F$ de degrŽ de transcendance infini. On peut donc appliquer \ref{descente sŽparable}: 
le co-caractre $\lambda$ (supposŽ $(F,u)$-optimal) est $(F_v,u)$-optimal. On vient de dŽmontrer que pour $i=1,\ldots ,k-1$, le morphisme  
$$\eta_{\lambda,u}(-i;F_v): \mathfrak{g}_\lambda(-i;F_v)\rightarrow \mathfrak{g}_{\lambda,k-i}(F_v)/\mathfrak{g}_{\lambda,k-i+1}(F_v)= \mathfrak{g}_\lambda(k-i\mathpvg F_v)$$ 
est injectif; cela implique que le morphisme $\eta_{\lambda,u}(-i;F)$ est lui aussi injectif.

\section{Mesure de Radon invariante sur les $F$-strates unipotentes}\label{mesures de Radon}
Dans toute cette section, $F$ est un corps commutatif localement compact non archimŽdien et la topologie est celle dŽfinie par $F$ ($\mathrm{Top}_F$); ˆ l'exception 
de \ref{le cas global} o $F$ est un corps global. 

\subsection{Les objets}\label{les objets}
Soit $u\in \UF\smallsetminus \{1\}$. 
Posons $\scrY=\scrY_{F,u}$, $\scrX=\scrX_{F,u}$, $\bsfrY=\bsfrY_{F,u}$, $\bsfrX= \bsfrX_{F,u}$ et $P={_FP_u}$. 
Les ensembles $\scrX$ et $\bsfrX$ sont fermŽs dans $G(F)$, la $F$-lame $\scrY$ est ouverte dans $\scrX$ et la $F$-strate $\bsfrY$ est ouverte dans $\bsfrX$ 
La $F$-lame $\scrY$ est ouverte dans son $F$-saturŽ $\scrX$ et $\scrX$ est fermŽ dans $G(F)$ (cf. \cite[2.8]{L}). Soit $G(F)\times^{P(F)}\scrX$ le quotient 
$(G(F)\times \scrX)/P(F)$ pour l'action de $P(F)$ donnŽe par $(g,x)\cdot p = (gp,p^{-1}xp)$. Pour un couple 
$(g,x)\in G(F)\times \scrX$, on note $[g,x]$ son image dans $G(F)\times^{P(F)} \scrX$. D'aprs \cite[2.5.13]{L}, l'application (surjective) 
$$\bs{\pi}:G(F)\times^{P(F)}\!\scrX \rightarrow \bsfrX\vgq [g,x] \mapsto gxg^{-1}$$ induit par restriction une application bijective 
$$G(F)\times^{P(F)}\scrY \rightarrow \bsfrY$$ qui est un homŽomorphisme. De plus on a (loc.~cit.)
$$\bs{\pi}^{-1}(\bsfrY)= G(F)\times^{P(F)} \scrY\ptf$$

Fixons un $\lambda \in \Lambda_{F,u}^{\rm opt}$. 
Rappelons que $\scrX= G_{\lambda,k}(F)$ avec $k=m_{F,u}=m_u(\lambda) \geq 1$. 
En particulier $\scrX$ est un groupe unipotent (contenu dans $U_P(F)$). 

Les groupes $G(F)$ et $\scrX$ sont unimodulaires. On les munit 
d'une mesure de Haar, respectivement ${\rm d}g$ et ${\rm d}x$. 
Le groupe $P(F)=P_\lambda(F)$ n'est pas unimodulaire (il est distinct de $G(F)$ car $u\neq 1$). 
On le munit d'une mesure de Haar 
invariante ˆ droite ${\rm d}_{\rm r}p$ et on note $\bs{\delta}_P$ son caractre module. 

Pour $p\in P(F)$ et $s,\, t \in \mbb{N}^*$ avec $s<t$, on pose 
$$\bs{\delta}_{\lambda, s}(p)= \vert \det ({\rm Ad}_p \,\vert \, \mathfrak{g}_{\lambda,s}(F))\vert_F$$ et 
$$\bs{\delta}_{\lambda,(s,t)}(p)=\bs{\delta}_{\lambda,s}(p) \bs{\delta}_{\lambda,t}(p)^{-1}=\vert \det ({\rm Ad}_p\,\vert \, \mathfrak{g}_{\lambda,s}(F)/ \mathfrak{g}_{\lambda,t}(F) )\vert_F\ptf$$ 
On a donc $$\bs{\delta}_{\lambda,1}(p)= \bs{\delta}_P(p) \quad \hbox{et} \quad \bs{\delta}_{\lambda,(1,k)}(p) = \bs{\delta}_P(p)\bs{\delta}_{\lambda,k}(p)^{-1}\ptf$$

\begin{remark}{\rm 
Si $\scrX\neq U_P$, il n'y a pas de mesure de Radon $G(F)$-invariante ˆ gauche non nulle sur $G(F)\times^{P(F)}\scrX$ 
mais il existe une fonctionnelle linŽaire $G(F)$-invariante ˆ gauche non nulle sur l'espace des sections du fibrŽ en droite sur 
$G(F)\times^{P(F)}\scrX$ dŽfini par $\bs{\delta}_{\lambda,(1,k)}^{-1}$. 
Nous aurons aussi besoin de considŽrer le quotient $$P(F)\times^{P(F)} \scrX= (P(F)\times \scrX)/P(F)\simeq \scrX\ptf$$ 
C'est pourquoi nous introduisons le groupe $H$ ci-dessous (en pratique on prendra $H=G$ ou $H=P$). 
}\end{remark}

Soit $H$ un sous-groupe fermŽ de $G$ dŽfini sur $F$ et contenant $P$. Le groupe $H(F)$ est localement compact; 
on peut donc le munir d'une mesure de Haar invariante ˆ droite  
$\dd_{\rm r}h$. Soit $\HHFu $ l'espace des fonctions localement constantes $\Phi$ sur $H(F) \times \scrX$ qui vŽrifient, pour 
tout $(h,x)\in H(F)\times \scrX$: 
\begin{itemize}
\item $\Phi((h,x)\cdot p) = \bs{\delta}_{\lambda,(1,k)}(p)^{-1}\Phi(h,x)$ pour tout $p\in P(F)$; 
\item il existe un sous-ensemble compact $\Omega_\Phi \subset H(F) \times \scrX$ tel que si 
$\Phi(h,x)\neq 0$, alors il existe un $p\in P(F)$ tel que $(h,x)\cdot p \in \Omega_\Phi$.
\end{itemize}
Observons que pour $\Phi\in \HHFu$ et $(h,x)\in H(F)\times \scrX$, on a 
$$\Phi(h,xx')= \Phi(h x'^{-1}, x'x) \quad \hbox{pour tout} \quad x'\in \scrX\ptf$$
Pour $\theta\in C^\infty_{\rm c}(H(F)\times\scrX)$ et $(h,x)\in H(F)\times \scrX$, posons 
$$\Phi_\theta(h,x) = \int_{P(F)} \theta(hp,p^{-1} x p) \bs{\delta}_{\lambda,k}(p)^{-1}\dd_{\rm r} p\ptf $$ 
Puisque $\dd_{\rm l}p= \bs{\delta}_{P}(p)^{-1} \dd_{\rm r}p$ est une mesure de Haar invariante ˆ gauche sur 
$P(F)$, la fonction $\Phi_\theta$ appartient ˆ $\HHFu$. 

\begin{lemma}\label{surjectivitŽ f donne Phi}
\begin{enumerate}
\item[(i)]L'application linŽaire $ C^\infty_{\rm c}(H(F)\times \scrX)\rightarrow \HHFu,\, \theta \mapsto \Phi_\theta$ est surjective. 
\item[(ii)] Si $\Phi_\theta=0$ alors $$\int_{H(F)\times \scrX} \theta(h,x)\dd_{\rm r}h \dd x=0\ptf$$
\end{enumerate}
\end{lemma}

\begin{proof}
Posons $\mathcal{C}= C^\infty_{\rm c}(H(F)\times \scrX)$ et $\mathcal{H}= \HHFu$. Soit $\Omega$ un sous-ensemble ouvert compact de $H(F)\times \scrX$ 
de la forme $\Omega_1\times \Omega_2$. On suppose que $\Omega_1$ est de la forme $K_1h_1$ pour un sous-groupe ouvert compact $K_1$ de 
$H(F)$ et un ŽlŽment $h_1\in H(F)$. On note $\mathcal{C}_\Omega$ le sous-espace de $\mathcal{C}$ formŽ des fonctions ˆ support dans $\Omega\cdot P(F)$ 
qui sont constantes sur $\Omega\cdot p= \Omega_1 p \times p^{-1}\Omega_2 p$ pour tout $p\in P(F)$. On dŽfinit de la mme manire 
le sous-espace $\mathcal{H}_\Omega$ de $\mathcal{H}$. 
Observons que si $\Omega\cdot p \cap \Omega\cdot p'\neq \emptyset$ pour des ŽlŽments $p,\, p'\in P(F)$, alors $p' \in h_1^{-1}K_1 h_1 p$ et 
$\bs{\delta}_{\lambda,(1,k)}(p')= \bs{\delta}_{\lambda,(1,k)}(p)$. Cela donne un sens ˆ la dŽfinition de $\mathcal{H}_\Omega$: on a 
$\mathcal{H}_\Omega \neq \{0\}$ et $\dim_{\mbb{C}}(\mathcal{H}_\Omega)=1$. L'application 
$\theta\mapsto \Phi_\theta$ envoie $\mathcal{C}_\Omega$ dans $\mathcal{H}_\Omega$ et il suffit de vŽrifier les assertions du lemme pour ces deux espaces. 

Prouvons (i). Soit $\Phi \in \mathcal{H}_\Omega$. Notons $\mathbf{1}_\Omega$ la fonction caractŽristique 
de $\Omega$ et pour $\omega= (h,x)\in \Omega$, posons 
$$\kappa(\omega) = \int_{P(F)}{\bf 1}_\Omega(\omega\cdot p) \bs{\delta}_{\lambda,(1,k)}(p)^{-1}\dd_{\rm r}p >0\ptf$$
La fonction $\theta = \kappa^{-1}\cdot {\bf 1}_\Omega\cdot \Phi$ appartient ˆ $\mathcal{C}_\Omega$ et on a $\Phi_\theta=\Phi$. 

Prouvons (ii). Pour $\theta\in \mathcal{C}_\Omega$ et $p\in P(F)$, notons $\theta^p\in G(F)\times \scrX$ la fonction dŽfinie par $\theta^p(h,x)= \theta(hp,p^{-1}xp)$. 
C'est encore un ŽlŽment de $\mathcal{C}_\Omega$ et tout ŽlŽment de $\mathcal{C}_\Omega$ est combinaison linŽaire (finie) de fonctions 
$({\bf 1}_\Omega)^p$ avec $p\in P(F)$. On en dŽduit que les fonctionnelles linŽaires $T$ sur $\mathcal{C}_\Omega$ 
telles que $T(\theta^p) = \bs{\delta}_{\lambda,k}(p)T(\theta)$ pour tout $\theta\in \mathcal{C}_\Omega$ et 
tout $p\in P(F)$, forment un espace vectoriel de dimension $1$. Comme les fonctionnelles linŽaires 
$$T(\theta)= \int_{H(F)\times \scrX}\theta(h,x)\dd_{\rm r}h \dd x\quad \hbox{et} \quad T'_\omega(\theta) = \Phi_\theta(\omega)\quad \hbox{avec}\quad \omega\in \Omega$$ 
sont dans cet espace, elles sont proportionnelles. Cela implique (ii). \end{proof}

On note $\bs{\mu}_{F,u}^H$ la fonctionnelle linŽaire positive sur l'espace $\HHFu$ dŽfinie par 
$$\langle \bs{\mu}_{F,u}^H, \Phi \rangle = \int_{H(F)\times \scrX}\theta(h,x)\dd_{\rm r} h \dd x$$ 
pour une (i.e. pour toute) fonction $\theta\in C^\infty_{\rm c}(H(F)\times \scrX)$ telle que $\Phi=\Phi_\theta$. 
On Žcrit aussi $$\langle \bs{\mu}^H_{F,u},\Phi \rangle = \int_{H(F)\times^{P(F)}\scrX} \Phi(h,x) \dd \bs{\mu}_{F,u}^H(h,x) \ptf$$ 
Le groupe $H(F)$ opre ˆ gauche sur l'espace $H(F)\times \scrX$ par translations sur la premire variable. L'application linŽaire 
$\theta \mapsto \Phi_\theta$ est $H(F)$-Žquivariante pour cette action: pour $h\in H(F)$, en notant ${^h\xi}$ la fonction $(h',x)\mapsto \xi(h^{-1}h',x)$ pour toute fonction 
$\xi$ sur $H(F)\times \scrX$, on a ${^h(\Phi_\theta)}= \Phi_{\smash{^h\!\theta}}$. Soit $\bs{\delta}_H$ le caractre-module de $H(F)$. 
Puisque $\dd_{\rm l} h= \bs{\delta}_H(h)^{-1}\dd_{\rm r}h$ est une mesure 
de Haar invariante ˆ gauche sur $H(F)$, on a le 

\begin{lemma}\label{action par translations ˆ gauche} 
Pour tout $\Phi\in \HHFu$ et tout $h\in H(F)$, on a
$$\langle \bs{\mu}_{F,u}^H,\!{^h\Phi}\rangle = \bs{\delta}_H(h)\langle \bs{\mu}^H_{F,u},\Phi \rangle\ptf$$ 
\end{lemma} 

\begin{remark}\label{remarques gŽnŽrales sur l'espace H rond}
\textup{
\begin{enumerate}
\item[(i)]On peut de la mme manire dŽfinir une fonctionnelle linŽaire $H(F)$-invariante ˆ gauche sur l'espace des sections du fibrŽ en droite sur 
$H(F)\times^{P(F)}\scrX$ dŽfini par $\bs{\delta}_H\bs{\delta}_{\lambda,(1,k)}^{-1}$: il suffit pour cela de remplacer 
$\bs{\delta}_{\lambda,k}(p)$ par $\bs{\delta}_H(p)\bs{\delta}_{\lambda,k}(p)$ dans la 
dŽfinition de l'application $\theta\mapsto \Phi_\theta$ et $\dd_{\rm r}h$ par $\dd_{\rm l}h $ dans celle de la fonctionnelle linŽaire $\bs{\mu}^H_{F,u}$. 
Mais pour $H=P$, on prŽfre 
travailler avec une fonctionnelle linŽaire qui se transforme ˆ gauche suivant $\bs{\delta}_{P}$.
\item[(ii)] Pour $H=P$ et $\Phi \in \mathcal{H}_{F,u}^P$, la fonction $\Phi(1,\cdot)$ sur $\scrX$ est lisse et ˆ support compact, et elle vŽrifie 
$$\Phi(1,pxp^{-1})= \bs{\delta}_{\lambda,(1,k)}(p)\Phi(p,x)\quad 
\hbox{pour tout} \quad (p,x)\in P(F)\times \scrX\ptf$$ 
L'application $\Phi \mapsto \Phi(1,\cdot)$ est un isomorphisme de $\mathcal{H}_{F,u}^P$ sur un sous-espace de $C^\infty_{\mathrm{c}}(\scrX)$ qui n'est pas vraiment commode ˆ 
utiliser et dont la dŽfinition passe de toutes faons par celle de $\mathcal{H}_{F,u}^P$. 
\end{enumerate}}
\end{remark}

Tous les objets que l'on vient de dŽfinir ne dŽpendent que de la $F$-lame $\scrY= \scrY_{F,u}$ ˆ laquelle appartient 
$u$, du groupe $H(F)$ et du choix des mesures de Haar. Rappelons que la $F$-lame $\scrY$ dŽtermine l'ensemble $\scrX$, 
la $F$-strate $\bsfrY$ et l'ensemble $\bsfrX$.

\begin{notation}
{\rm 
Pour allŽger l'Žcriture, on pose:
\begin{itemize}
\item $\mathcal{H}= \mathcal{H}^{G}_{F,u}$ et $\bs{\mu}=\bs{\mu}_{F,u}^{G}$;
\item $\mathcal{H}^P= \mathcal{H}^{P}_{F,u}$ et $\bs{\mu}^P= \bs{\mu}_{F,u}^{P}$.
\end{itemize}
}
\end{notation}

On a une action ˆ droite de $\scrX$ sur $P(F)\times \scrX$ donnŽe par $(p,x)\star x' = (p,x p^{-1}x' p)$ pour tout $(p,x)\in P(F)\times \scrX$ et 
tout $x'\in \scrX$. Cette action {\it ne commute pas} ˆ l'action ˆ gauche de $P(F)$ par translations sur la premire variable. 
PrŽcisŽment, pour tous $(p,x),\, (p'x')\in P(F)\times \scrX$, on a 
$$p \cdot ((p',x')\star x)=(p\cdot (p',x'))\star pxp^{-1}\ptf$$
En revanche elle commute ˆ l'action ˆ droite de $P(F)$ donnŽe par $$(p,x)\cdot p'=(pp',p'^{-1}xp')$$ 
et par consŽquent 
rend $\scrX$-Žquivariante l'application $C^\infty_{\rm c}(P(F)\times \scrX)\rightarrow \mathcal{H}^P,\, \theta\mapsto \Phi_\theta$: pour tout $x'\in \scrX$, 
on a $$(\Phi_\theta)^{\star x'}= \Phi_{\smash{\theta^{\star x'}}}$$ o, 
pour toute fonction $\xi$ sur $P(F)\times \scrX$, on a posŽ $\xi^{\star x'}(p,x)=\xi((p,x)\star x')$. 

\begin{remark}\textup{Cette action de $\scrX$ sur $\mathcal{H}^P$ n'est autre que celle dŽduite des translations 
ˆ droite via l'isomorphisme $\Phi \mapsto \Phi(1,\cdot)$ de \ref{remarques gŽnŽrales sur l'espace H rond}\,(ii): pour $\Phi\in \mathcal{H}^P$ et $x\in \scrX$, on a 
$\Phi^{\star x}(1,\cdot)= (\Phi(1,\cdot))^x$ avec $\psi^x(x')= \psi(x'x)$.
}\end{remark}

On en dŽduit que la fonctionnelle linŽaire $\bs{\mu}^P$ 
sur $\mathcal{H}^P$ est $\scrX$-invariante pour cette action ˆ droite: 

\begin{lemma}\label{action de X ˆ droite}
Pour tout $\Phi\in \mathcal{H}^P$ et tout $x\in \scrX$, on a
$$\langle \bs{\mu}^P\!, \Phi^{\star x}\rangle = \langle \bs{\mu}^P\! ,\Phi\rangle\ptf$$
\end{lemma}

Les propriŽtŽs d'invariance \ref{action par translations ˆ gauche} et \ref{action de X ˆ droite} impliquent le

\begin{lemma}\label{unicitŽ de tilde mu}
\`A multiplication prs par une constante non nulle, $\bs{\mu}^P$ est l'unique fonctionnelle linŽaire non nulle sur $\mathcal{H}^P$ 
qui se transforme ˆ gauche suivant $\bs{\delta}_{P}$ (\ref{action par translations ˆ gauche}) et soit $\scrX$-invariante ˆ droite (\ref{action de X ˆ droite}).
\end{lemma}

\begin{proof}
Fixons un systme fondamental $\{K^\nu\}_{\nu \in I}$ de voisinages de l'unitŽ dans $P(F)\times \scrX$ formŽ de sous-groupes 
ouverts compacts de la forme $K^\nu= K^\nu_1\times K^\nu_2$ avec $K^\nu_1\subset P(F)$, $K^\nu_2\subset \scrX$ et tel que $K^\nu_2$ soit 
distinguŽ dans $K^\nu_1$. Pour $\nu\in I$, notons 
$\bs{1}_{K^\nu}$ la fonction caractŽristique de $K^\nu$, posons $\xi^\nu=\Phi_{\bs{1}_{K^\nu}}\in \mathcal{H}^P$ et notons 
$\mathcal{H}^{P,\nu}$ le sous-espace de $\mathcal{H}^P$ engendrŽ par les fonctions ${^p((\xi^\nu)^{\star x})}$ 
pour $(p,x)\in P(F)\times \scrX$. 
Puisque ${^p(\Phi^{\star x})}= ({^p\Phi})^{\star pxp^{-1}}$, $\mathcal{H}^{P,\nu}$ est aussi le sous-espace de $\mathcal{H}^P$ 
engendrŽ par les fonctions $({^p(\xi^\nu)})^{\star x}$ pour $(p,x)\in P(F)\times \scrX$. 
D'autre part les isomorphismes $$\mathcal{H}^P\rightarrow \mathcal{H}^P,\, 
\Phi \mapsto {^p((\Phi)^{\star x})}\quad \hbox{pour}\quad (p,x)\in P(F)\times \scrX$$ 
stabilisent $\mathcal{H}^{P,\nu}$. 

Puisque $\mathcal{H}^P= \bigcup_{\nu \in I} \mathcal{H}^{P,\nu}$ et 
$\mathcal{H}^{P,\nu} \supset \mathcal{H}^{P,\nu'}$ si $K^\nu \subset K^{\nu'}$, pour dŽfinir 
une fonctionnelle linŽaire $\bs{\eta}$ sur $\mathcal{H}^P$ qui se transforme ˆ gauche suivant $P(F)$ et soit $\scrX$-invariante ˆ droite, 
il suffit de dŽfinir une famille compatible $\{\bs{\eta}^\nu\}_{\nu\in I}$ o $\bs{\eta}^\nu$ est une fonctionnelle linŽaire 
sur $\mathcal{H}^{P,\nu}$ telle que $\langle \bs{\eta}^\nu , {^p(\Phi^{\star x})} \rangle =\bs{\delta}_{P}(p) \langle \bs{\eta}^\nu, \Phi \rangle $ pour tout 
$\Phi\in \mathcal{H}^{P,\nu}$ et tout $(p,x)\in P(F)\times \scrX$. Une telle fonctionnelle est entirement dŽterminŽe par la valeur 
$\langle \bs{\eta}^\nu, \xi^\nu \rangle$. Cela prouve le lemme.
\end{proof}

\subsection{Le lemme-clŽ}\label{le lemme-clŽ} Continuons avec les notations de \ref{les objets}. 
Posons $$\overline{\scrX}=G_{\lambda}(k;F)\;(=G_{\lambda,k}(F)/G_{\lambda,k+1}(F))$$ 
et notons $\overline{\scrY}$ l'image de $\scrY$ dans $\overline{\scrX}$. Pour $x\in \scrX$ d'image 
$\bar{x}$ dans $\overline{X}$ et $i\in \mbb{Z}$, posons (cf. \ref{ŽnoncŽ})  
$$\alpha_i(\bar{x})= \eta_{\lambda,x}(i\mathpvg F) \quad \hbox{et}\quad d_i = \dim_F(\mathfrak{g}_\lambda(i\mathpvg F))\ptf$$  
D'aprs \ref{proposition noyau} (prouvŽ en \ref{le cas d'un corps local}) et \ref{injectivitŽ=surjectivitŽ}, pour $i=1,\ldots ,k-1$, l'application 
$\alpha_{-i}(\bar{u})$ est un isomorphisme; et on a $d_{k-i}=d_{-i}=d_i$. 
Fixons une $F$-base $(H_{-i,j})_{1\leq j \leq d_i}$ de $\mathfrak{g}_\lambda(-i;F)$ et une $F$-base 
$(H_{k-i,j})_{1\leq j \leq d_i}$ de $\mathfrak{g}_\lambda(k-i;F)$; on peut par exemple 
prendre $H_{k-i,j}= \alpha_{-i}(\bar{u})(H_{-i,j})$. 
Pour $\bar{x}\in \overline{\scrX}$, notons $A_{-i}(\bar{x})\in M(d_i,F)$ la matrice de $\alpha_{-i}(\bar{x})$ relativement 
ˆ ces bases et posons $$\varphi_{-i}(\bar{x})= \vert {\rm det}_F(A_{-i}(\bar{x}))\vert_F^{1/2}\ptf$$ Soit 
$$\varphi(\bar{x})= \varphi_{-1}(\bar{x})\cdots \varphi_{1-k}(\bar{x})\geq 0\ptf$$ 
D'aprs loc.~cit., on a $$\varphi(\bar{y})>0\quad\hbox{pour tout}\quad \bar{y}\in \overline{\scrY}\ptf$$
Observons qu'ˆ multiplication prs par une constante $>0$, la fonction $\varphi$ ne dŽpend pas du choix des bases $(H_{-i,j})$ et $(H_{k-i,j})$ 
pour $i=1,\ldots ,k-1$. 

\begin{lemma}\label{varphi et inverse}
Pour $\bar{x}\in \overline{\scrX}$, on a $\varphi(\bar{x}^{-1}) = \varphi(\bar{x})$.
\end{lemma}
\begin{proof}
On peut supposer $u$ en position standard et $\lambda \in \check{X}(A_0)$. Le $F$-isomorphisme $j_0: \mathfrak{u}_0\rightarrow U_0$ induit par 
restriction et passage aux quotients un $F$-isomorphisme $j_\lambda (k) : \mathfrak{g}_\lambda(k)\rightarrow G_\lambda(k)=G_{\lambda,k}/G_{\lambda,k+1}$. 
Pour $\bar{x}\in \overline{\scrX}$ et $i=1,\ldots , k-1 $, 
en posant $\overline{X}= j_\lambda(k)^{-1}(\bar{x})\in \mathfrak{g}_\lambda(k\mathpvg F)$,  on a 
$$\alpha_{-i}(\bar{x})= [\overline{X}, \cdot] : \mathfrak{g}_\lambda(-i\mathpvg F)\rightarrow \mathfrak{g}_\lambda(k-i\mathpvg F)\ptf$$ 
D'o le lemme, puisque $\bar{x}^{-1}= j_\lambda(k)(- \overline{X})$. 
\end{proof}

Le lemme suivant est l'analogue du lemma 2 de \cite{RR}:

\begin{lemma}\label{lemme RRao}
Pour tout $\bar{x}\in \overline{\scrX}$ et tout $m\in M_\lambda(F)$, on a 
$$\varphi ({\rm Int}_m(\bar{x}))=  \bs{\delta}_{\lambda,(1,k)}(m) \varphi(\bar{x})\ptf$$ 
\end{lemma}

\begin{proof}
Pour $i\in \{\pm 1,\ldots ,\pm (k-1)\}$, notons $B_i(m)\in M(d_i,F)$ la matrice de l'automorphisme 
${\rm Ad}_m$ de $\mathfrak{g}_\lambda(i;F)$ relativement ˆ la base $(H_{ i,j})$. On a 
$$A_{-i}({\rm Int}_m(\bar{x}))= B_{k-i}(m)A_{-i}(\bar{x}) B_{-i}(m^{-1})$$ 
En observant que $\vert \det_F (B_{-i}(m^{-1}))\vert_F^{1/2}= \vert \det_F(B_i(m))\vert_F^{1/2}$, on obtient 
$$\varphi({\rm Int}_m(\bar{x}))= \vert {\rm \det}_F(B_1(m))\vert_F \cdots \vert {\rm det}_F(B_{k-1}(m))\vert_F \varphi(\bar{x})\ptf $$ 
Or on a 
$$\vert {\rm \det}_F(B_1(m))\vert_F \cdots \vert {\rm det}_F(B_{k-1}(m))\vert_F = 
\vert {\rm det}_F({\rm Ad}_m\,\vert \, \mathfrak{g}_{\lambda,1}(F)/ \mathfrak{g}_{\lambda,k}(F))\vert_F\vg$$ 
ce qui achve la dŽmonstration du lemme. 
\end{proof}

\subsection{Le rŽsultat principal}\label{le rŽsultat principal}
La $F$-paire parabolique minimale $(P_0,A_0)$ de $G$ Žtant fixŽe, on choisit un sous-groupe ouvert compact maximal $K$ de $G(F)$ tel que 
$$G(F)= KP_0(F)\ptf$$ Quitte ˆ remplacer $u$ par un conjuguŽ dans $G(F)$, on peut supposer que $u$ est en position standard (cf. \ref{de G ˆ Lie(G)}) et $\lambda\in \check{X}(A_0)$. 
On a donc $$\bsfrY=\mathrm{Int}_K(\scrY)\;(=\{ky k^{-1}\,\vert\, y \in \scrY,\, k\in K\}) \quad \hbox{et} \quad \bsfrX= \mathrm{Int}_K(\scrX)\ptf$$ 

Rappelons qu'il existe constante $c>0$  (qui ne dŽpend que du choix des mesures de Haar) telle que 
pour toute fonction $f\in C^\infty_{\rm c}(G(F))$, on ait  
$$\int_{K\times P(F)} f(kp) \dd k \dd_{\rm r}p = c\int_{G(F)}f(g)\dd g$$ 
o $\dd k$ est la mesure de Haar sur $K$ dŽduite de $\dd g$ par restriction. On normalise $\dd g$ et $\dd_{\mathrm{r}}p$ de telle manire que $c=1$. Pour une fonction 
$\phi$ sur $G(F)$ ou sur $P(F)$, on note ${^p\phi}=\phi^{\,p^{-1}}$ la fonction $\phi\circ {\rm Int}_{p^{-1}}$. Pour $f\in C^\infty_{\rm c}(G(F))$ et $g\in G(F)$, on pose 
$$f^K(g)= \int_{K} f(kgk^{-1})\dd k\ptf$$

\begin{lemma}\label{invariance par G}
Soit $\Lambda$ une distribution sur $P(F)$. Si pour tout $p\in P(F)$, on a 
$$\Lambda ({^p\phi})= \bs{\delta}_{P}(p) \Lambda(\phi)\quad\hbox{pour toute fonction}\quad  \phi \in C^\infty_{\rm c}(P(F))\vg$$
alors la distribution $D_\Lambda$ sur $G(F)$ dŽfinie par $$D_\Lambda(f)= \Lambda(f^K)\;(=\Lambda(f^K\vert_{P(F)}))$$ 
est $G(F)$-invariante (pour la conjugaison).
\end{lemma}

\begin{proof}Soit $\mathcal{F}_{P}$ l'espace des fonctions localement constantes $\phi$ sur $G(F)$ telles que $\phi(gp)= \bs{\delta}_{P}(p)^{-1}\phi(g)$ 
pour tout $(g,p)\in G(F)\times P(F)$. Pour $f\in C^\infty_{\rm c}(G(F))$ et $g\in G(F)$, posons $$\phi_f(g)= \int_{P(F)}f(gp)\dd_{\rm r} p \ptf$$ 
La fonction $\phi_f$ appartient ˆ $\mathcal{F}_{P}$. L'application linŽaire $C^\infty_{\rm c}(G(F))\rightarrow \mathcal{F}_{P},\, f \mapsto \phi_f$ est 
surjective et $G(F)$-Žquivariante pour l'action de $G(F)$ par translations ˆ gauche sur les deux espaces. 
ConsidŽrons la fonctionnelle linŽaire $\alpha$ sur $\mathcal{F}_{P}$ donnŽe par $$\langle \alpha,\phi\rangle= \int_K \phi(k)\dd k\ptf$$ 
Puisque $$\langle \alpha,\phi _f\rangle=  \int_{G(F)}f(g)\dd g\vg $$  elle est invariante pour l'action de $G(F)$ par translations ˆ gauche sur $\mathcal{F}_{P}$. 

Pour  $f\in C^\infty_{\rm c}(G(F))$ et $g\in G(F)$, posons 
$\Theta_f(g)= \Lambda(f\circ {\rm Int}_g\vert_{P(F)})$. La fonction $\Theta_f$ appartient ˆ $\mathcal{F}_{P}$ et l'application linŽaire $f \mapsto \Theta_f$ est $G(F)$-Žquivariante pour 
l'action de $G(F)$ par conjugaison sur $C^\infty_{\rm c}(G(F))$ et par translations ˆ gauche sur $\mathcal{F}_{P}$: 
$$\Theta_f(xg) = \Theta_{f\circ {\rm Int}_x} (g) \quad \hbox{pour tout}\quad x\in G(F)\ptf$$ 
Comme $$\langle \alpha, \Theta_f \rangle = D_\Lambda(f)\vg$$ cela dŽmontre le lemme.
\end{proof}

Soient $\dd \bar{x}$ et $\dd v$ des mesures de Haar sur $\overline{\scrX}$ et $\scrV= G_{\lambda,k+1}(F)$, choisies de telle manire que pour toute fonction 
$\phi\in C^\infty_{\rm c}(\scrX)$, on ait l'ŽgalitŽ 
$$\int_{\scrX}\phi(x)\dd x= \int_{\overline{\scrX}\times \scrV} \phi(\bar{x}v) \dd \bar{x} \dd v\ptf$$
Pour $x\in \scrX$, on pose $\varphi(x)=\varphi(\bar{x})$. 

\begin{proposition}\label{ACV et distribution invariante}
Pour toute fonction $f\in C^\infty_{\rm c}(G(F))$, l'intŽgrale 
$$\int_{\scrY}\varphi(y) f^K(y)\dd y= \int_{\overline{\scrY} \times \scrV} \varphi(\overline{y})f^K(\bar{y} v) \dd \bar{y} \dd v$$ 
est absolument convergente et dŽfinit une distribution $G(F)$-invariante (pour la conjugaison) sur $G(F)$, que l'on note 
$I_{\bsfrY}(f)= I_{\scrY}^P(f^K)$.  
Observons que $I_{\bsfrY}(f)=0$ pour toute fonction 
$f\in C^\infty_{\rm c}(G(F))$ qui s'annule sur $\bsfrY$.
\end{proposition}

\begin{proof}
Il est clair que l'intŽgrale $I_{\bsfrY}(f)$ est absolument convergente et dŽfinit une distribution sur $G(F)$. 
Il faut prouver que cette distribution est $G(F)$-invariante. Pour $\phi\in C^\infty_{\rm c}(P(F))$, posons 
$$\Lambda(\phi)= \int_{\overline{\scrY} \times \scrV} \varphi(\overline{y})\phi(\bar{y} v) \dd \bar{y} \dd v\ptf$$ 
Pour $p\in P(F)\;(=P_\lambda(F))$, calculons $\Lambda({^p\phi})$ avec ${^p\phi}=\phi \circ {\rm Int}_{p^{-1}}$. \'Ecrivons $p=mu$ avec $m\in M_\lambda(F)$ et 
$u\in U_\lambda(F)$. On a 
\begin{eqnarray*}
\Lambda({^p\phi})&=& \int_{\overline{\scrY} \times \scrV} \varphi(\overline{y})\phi({\rm Int}_{m^{-1}}(\bar{y}) {\rm Int}_{p^{-1}}(v)) \dd \bar{y} \dd v\\
&=&  \bs{\delta}_{\lambda,(k,k+1)}(m) \bs{\delta}_{\lambda,k+1}(p)\int_{\overline{\scrY} \times \scrV} \varphi({\rm Int}_m(\overline{y}))\phi(\bar{y} v) \dd \bar{y} \dd v\ptf
\end{eqnarray*}
En utilisant \ref{lemme RRao} et l'ŽgalitŽ $\bs{\delta}_{\lambda,(1,k)}(m)\bs{\delta}_{\lambda,(k,k+1)}(m) \bs{\delta}_{\lambda,k+1}(p) = \bs{\delta}_{P}(p)$, on obtient 
$$\Lambda({^p\phi}) = \bs{\delta}_{P}(p) \Lambda(\phi)\ptf$$
On conclut gr‰ce ˆ \ref{invariance par G}.
\end{proof}

\begin{remark}
\textup{La $F$-paire parabolique minimale $P_0=M_0\ltimes U_0$ Žtant fixŽe, la $F$-lame standard $\scrY$ est dŽterminŽe 
par la $F$-strate $\bsfrY$. Observons que la distribution $I_{\bsfrY}$ sur $G(F)$ dŽpend bien sžr du choix de la mesure de Haar $\dd x$  sur $\scrX$ mais aussi de celui 
du sous-groupe compact maximal $K$ de $G(F)$ tel que $G(F)=KP_0(F)$. En effet, remplaons par exemple $K$ par $K'=pKp^{-1}$ pour un 
$p\in P_0(F)$ et notons $I'_{\bsfrY}$ la distribution sur $G(F)$ obtenue en remplaant $K$ par $K'$ dans la dŽfinition de 
$I_{\bsfrY}$. Pour $f\in C^\infty_{\mathrm{c}}(G(F))$ et $y\in \scrY$, on a $f^{K'}(y)=({^{p^{-1}\!\!}f})^K(p^{-1}yp)$. On en dŽduit que 
$$I'_{\bsfrY}(f)= I_{\scrY}^P(f^{K'})= \bs{\delta}_P(p)I_{\scrY}^P(({^{p^{-1}}\!\!f})^K)= \bs{\delta}_P(p)I_{\bsfrY}(f)\ptf$$}
\end{remark}

Pour $\Phi\in \mathcal{H}$ et $(g,x)\in G(F)\times \scrX$, posons 
$$\Phi_K (g,x)= \int_K \Phi(kg ,x)\dd k \ptf$$ 
La restriction de $\Phi_K$ ˆ $P(F)\times \scrX$ appartient ˆ $\mathcal{H}^P$.

\begin{lemma}\label{lien mu et tmu}
Pour $\Phi \in \mathcal{H}$, on a
$$ \langle \bs{\mu},\Phi \rangle = \langle  \bs{\mu}^P\!, \Phi_K\rangle  \ptf$$
\end{lemma}

\begin{proof}
Pour $f\in C^\infty_{\rm c}(G(F)\times \scrX)$, on a 
$$\int_{G\times \scrX}f(g,x) \dd g \dd x = \int_{P(F)\times \scrX}f_K(p,x)\dd_{\rm r}p \dd x$$ 
o $f_K\in C^\infty_{\rm c}(P(F)\times \scrX)$ est dŽfinie comme $\Phi_K$. D'autre part pour $\Phi = \Phi_f\;(\in \mathcal{H})$, on a $\Phi_K= \Phi_{\smash{f_K}}$. D'o le lemme. 
\end{proof} 

\begin{theorem}\label{rŽponse aux questions}
Pour toute fonction $f\in C^\infty_{\rm c}(G(F))$, on a l'ŽgalitŽ 
$$I_{\bsfrY}(f)= \int_{P(F)\times^{P(F)} \scrY} f^K(pyp^{-1}) \varphi(y)\dd \bs{\mu}^P\!(p,y);$$ 
l'intŽgrale Žtant absolument convergente.
\end{theorem}

\begin{proof} 
Notons $\wt{\mathcal{H}}^{P}$ le sous-espace de $\mathcal{H}^P$ formŽ des fonctions ˆ support dans $P(F)\times \scrY$. 
Soient $$ \wt{\mathcal{H}}^P\rightarrow C^\infty_{\rm c}(\scrY),\, \Phi \mapsto \psi_\Phi \quad \hbox{et}\quad 
C^\infty_{\rm c}(\scrY)\rightarrow \wt{\mathcal{H}}^P,\, \psi \mapsto \Phi^\psi$$ les applications linŽaires dŽfinies par 
$$\psi_\Phi(y)= \varphi(y)^{-1}\Phi(1,y) \quad \hbox{et} \quad \Phi^\psi(p,y)=  \varphi(y)\psi(pyp^{-1})\ptf$$ 
Puisque $$\psi_\phi(pyp^{-1})= \varphi(y)^{-1}\Phi(p,y)\quad\hbox{pour tout}\quad p\in P(F)\vg$$ 
elles sont inverses l'une de l'autre; ce sont donc des isomorphismes\footnote{Notons $\scrX^*$ l'ouvert de $\scrX$ dŽfini par $\scrX^*=\{x\in \scrX\,\vert\, \varphi(x)\neq 0\}$ et 
$\mathcal{H}^{P,*}\;(\supset \wt{\mathcal{H}}^P)$ le sous-espace de $\mathcal{H}^P$ formŽ des fonctions ˆ support dans $P(F)\times \scrX^*$. On obtient de la mme manire un isomorphisme 
$\mathcal{H}^{P,*}\simeq C^\infty_{\rm c}(\scrX^*)$.}. 

Soit $\Lambda'$ la distribution sur $\scrY$ dŽfinie par 
$$\Lambda'(\psi)= \langle \bs{\mu}^P\!,\Phi^\psi\rangle\quad \hbox{pour toute fonction} \quad \psi\in C^\infty_{\rm c}(\scrY)\ptf$$ 
Observons que pour $p\in P(F)$, puisque $\Phi^{^p\psi}= {^p(\Phi^\psi)}$ avec ${^p\psi}= \psi\circ {\rm Int}_{p^{-1}}$, on a bien 
$$\Lambda'({^p\psi})= \bs{\delta}_P(p) \Lambda'(\psi)\ptf$$ 

Soit $\psi\in C^\infty_{\rm c}(\scrY)$. Posons $\Phi= \Phi^\psi$ et Žcrivons $\Phi= \Phi_\theta$ avec 
$\theta \in C^\infty_{\rm c}(P(F)\times \scrY)$ comme en \ref{les objets}: 
$$\Phi(p,y)=\int_{P(F)} \theta(pp',p'^{-1}yp) \bs{\delta}_{\lambda,k}(p') \dd_{\rm r}p'\ptf$$ On a donc 
$$\Lambda'(\psi)=\int_{P(F)\times^{P(F)}\scrY} \Phi(p,y)\dd \bs{\mu}^P(p,y)=\int_{P(F)\times \scrY} \theta(p,y)\dd_{\rm r}p \dd y\ptf$$
Or on a 
\begin{eqnarray*}
\int_{P(F)\times \scrY} \theta(p,y)\dd_{\rm r}p\dd y&=& \int_{P(F)} \left(\int_{\scrY} \theta(p,p^{-1}y p) \dd (p^{-1}yp)\right) \dd_{\rm r}p\\
&= & \int_{P(F)} \left(\int_{\scrY}\theta(p,p^{-1}yp) \bs{\delta}_{1,k}(p)^{-1}\dd y\right) \dd_{\rm r}p\\
&=& \int_{\scrY} \Phi(1,y)\dd y\ptf
\end{eqnarray*}
Puisque $\Phi(1,y) = \varphi(y)\psi(y)$, on a donc
$$\Lambda'(\psi)= \int_{\scrY} \varphi(y) \psi (y)\dd y \ptf$$

On a prouvŽ que pour toute fonction $f\in C^\infty_{\rm c}(G(F))$ et tout sous-ensemble ouvert compact $\omega$ de $\scrY$, on a l'ŽgalitŽ
$$\Lambda'({\bf 1}_{\omega}\cdot f^K\vert_{\scrY}) = \int_{\omega}\varphi(y) f^K(y) \dd y \ptf$$ 
Puisque l'intŽgrale $\int_{\scrY} \varphi(y) f^K(y) \dd y=I_{\bsfrY}(f)$ est absolument convergente, 
le thŽo\-rme de convergence dominŽe assure que l'intŽgrale $$\Lambda'(f^K\vert_{\scrY})= \int_{P(F)\times^{P(F)}\scrY} f^K(py p^{-1}) \varphi(y) \dd \bs{\mu}^P(p,y)$$ l'est aussi et qu'on a l'ŽgalitŽ 
$\Lambda'(f^K\vert_{\scrY})=I_{\bsfrY}(f)$. 
\end{proof}

\begin{corollary}\label{corollaire au thŽorme}
La mesure $\bs{\mu}_{\scrY}^\varphi$ sur $G(F)\times^{P(F)}\scrY$ dŽfinie par $$\dd \bs{\mu}_{\scrY}^\varphi(g,y)= \varphi(y)\dd \bs{\mu}(g,y)$$ 
donne via l'homŽomorphisme $\bs{\pi}: G(F)\times^{P(F)}\scrY \buildrel \simeq\over{\longrightarrow} \bsfrY$ une mesure 
sur $\bsfrY$ qui dŽfinit une mesure de Radon positive $G(F)$-invariante non nulle sur $G(F)$. PrŽcisŽment, pour toute fonction $f\in C^\infty_{\rm c}(G(F))$, on a l'ŽgalitŽ 
$$I_{\bsfrY}(f)=\int_{G(F)\times^{P(F)}\scrY} f(gyg^{-1})\varphi(y)\dd \bs{\mu}(g,y);$$ l'intŽgrale Žtant absolument convergente. 
\end{corollary}

\begin{remark}\label{le thm pour Lie(G)}
\textup{Le mme rŽsultat vaut bien sžr pour les $F$-strates de $\NF$. D'ailleurs pour construire la fonction 
$\varphi$ nous avons dž travailler avec la $F$-strate de $\NF$ associŽe ˆ $\bsfrY$ (cf. \ref{de G ˆ Lie(G)}).}
\end{remark}

\subsection{Le cas o la $G(F)$-orbite est ouverte dans la $F$-strate}\label{le cas o l'orbite est ouverte} 
Continuons avec les hypothses de \ref{le rŽsultat principal}. 
Si $\Omega$ est un sous-ensemble ouvert $P(F)$-invariant de $\scrY$, en remplaant $\scrY$ par $\Omega$ dans la dŽfinition de 
$I_{\bsfrY}$, on obtient une distribution positive $G(F)$-invariante non nulle sur $G(F)$: pour $f\in C^\infty_{\rm c}(G(F))$, on pose 
$$I_{\mathrm{Int}_K(\Omega)}(f)=\int_{\Omega} \varphi(y) f^K(y) \dd y \ptf$$ 
L'intŽgrale est absolument convergente et la distribution $I_{\mathrm{Int}_K(\Omega)}$ annule toute fonction $f\in C^\infty_{\rm c}(G(F))$ qui s'annule sur $\mathrm{Int}_K(\Omega)$. 
Observons que $\mathrm{Int}_K(\Omega) \cap \scrY= \Omega$.  

On suppose dans cette sous-section 
que la $G(F)$-orbite $\ES{O}_{F,u}$ est ouverte dans la $F$-strate $\bsfrY= \bsfrY_{F,u}$. Puisque $$\ES{O}_{F,u}\cap \scrY= \{p^{-1}yp\,\vert \, p\in P(F)\}\bydef \ES{O}_{F,u}^P\quad 
\hbox{et} \quad \ES{O}_{F,u}= {\rm Int}_K(\ES{O}_{F,u}^P)\vg\leqno{(1)}$$ 
cela revient ˆ supposer que la $P(F)$-orbite $\ES{O}_{F,u}^P$ est ouverte dans la $F$-lame $\scrY = \scrY_{F,u}$. Pour $f\in C^\infty_{\rm c}(G(F))$, on pose 
$$I_u(f)=I_{\ES{O}_{F,u}}(f)\ptf$$

\begin{proposition}\label{mesure sur l'orbite ouverte}
On suppose que la $P(F)$-orbite $\ES{O}_{F,u}^P$ est ouverte dans la $F$-lame $\scrY=\scrY_{F,u}$. 
\begin{enumerate}
\item[(i)] Le centralisateur $G^u(F)$ de $u$ dans $G(F)$ est unimodulaire.
\item[(ii)] Si $\dd g^u$ est une mesure de Haar sur $G^u(F)$, il existe une constante $c=c(dg^u)>0$ telle que pour toute fonction $f\in C^\infty_{\rm c}(G(F))$, on ait 
$$  \int_{G^u(F)\backslash G(F)} f(g^{-1}u g) \textstyle{\frac{dg}{dg^u}} = c \hskip0.3mm I_{u}(f);$$ l'intŽgrale Žtant absolument convergente.
\end{enumerate}
\end{proposition}

\begin{proof}
L'application $f\mapsto I_{\bsfrY}(f)$ dŽfinit \textit{a fortiori} une mesure de Radon positive $G(F)$-invariante (pour la conjugaison) non nulle $\bs{\mu}_u=\bs{\mu}_{\ES{O}_{F,u}}$ sur $\ES{O}_{F,u}$: pour toute fonction 
$\psi\in C^\infty_{\rm c}(\ES{O}_{F,u})$, on a $$\langle \bs{\mu}_u,\psi\rangle = \int_{K \times \ES{O}_{F,u}^P}\varphi(y)\psi(k y k^{-1}) \dd k \dd y\ptf$$ Via l'homŽomorphisme 
$$ G^u(F)\backslash G(F) \buildrel\simeq\over{\longrightarrow} \ES{O}_{F,u}\vgq g \mapsto g^{-1}ug\vg$$ on a donc une mesure de Radon positive $G(F)$-invariante (ˆ droite) non nulle 
sur $G^u(F)\backslash G(F)$. L'existence d'une telle mesure entra"ne (i). Comme une telle mesure est unique ˆ une constante ($>0$) prs, si $dg^u$ est une mesure de Haar sur 
$G^u(F)$, il existe une constante $c>0$ telle que pour toute fonction $\psi\in C^\infty_{\rm c}(\ES{O}_{F,u})$, on ait  
$$\int_{G^u(F)\backslash G(F)}\psi(g^{-1}u g ) \textstyle{\frac{\dd g}{\dd g^u}}=  c\hskip0.3mm \langle \bs{\mu}_u, \psi \rangle \ptf$$ 
On en dŽduit que pour toute fonction $f\in C^\infty_{\rm c}(G(F))$ et tout 
sous-ensemble ouvert compact $\omega$ de $\ES{O}_{F,u}$, on a l'ŽgalitŽ 
$$\int_{G^u(F)\backslash G(F)} {\bf 1}_\omega(g^{-1}u g) f(g^{-1}u g) \textstyle{\frac{dg}{dg^u}}=  c\hskip0.3mm\langle \bs{\mu}_u, {\bf 1}_\omega \cdot f\vert_{\ES{O}_{F,u}}\rangle \ptf$$ 
Puisque l'intŽgrale $$\int_{ \ES{O}_{F,u}^P}\varphi(y) f^K(y) \dd y = I_{u}(f)$$ est absolument convergente, 
d'aprs le thŽorme de convergence dominŽe, cela prouve (ii). 
\end{proof}

\begin{corollary}
Supposons que toutes les $P(F)$-orbites $\ES{O}_{F,y}^P$ avec $y\in \scrY$ soient ouvertes (et donc aussi fermŽes) dans $\scrY$. Choisissons un ensemble reprŽsentants $\{u_i\,\vert\, i\in I\}\subset \scrY$ des $P(F)$-orbites dans $\scrY$. 
Pour chaque $i\in I$, normalisons la mesure de Haar $\dd g^{u_i}$ sur $G^{u_i}(F)$ par $c(\dd g^{u_i})=1$. Alors pour toute fonction 
$f\in C^\infty_{\rm c}(G(F))$, on a $$I_{\bsfrY}(f) = \sum_{i\in I} I_{u_i}(f)\ptf$$
\end{corollary}

\begin{remark}
\textup{On a la mme proposition pour $Y\in \NF$, sous l'hypothse o la $G(F)$-orbite $\ES{O}_{F,Y}$ est ouverte dans la $F$-strate $\bsfrY_{F,Y}$
}\end{remark}

\subsection{Comparaison avec les rŽsultats connus}\label{comparaison}
Pour $x\in G$, on note $\ES{O}_x$ l'orbite gŽomŽtrique ${\rm Int}_G(x)$; de mme pour 
$X\in \mathfrak{g}$, on pose $\ES{O}_X=\textrm{Ad}_G(X)$.

Supposons tout d'abord $p=1$. Dans ce cas on sait que la $F$-strate $\bsfrY=\bsfrY_{F,u}$ d'un ŽlŽment $u\in \UF$ est l'ensemble $\ES{O}_u(F)$ des points $F$-rationnels de 
l'orbite gŽomŽtrique $ \ES{O}_u$. Cet ensemble est rŽunion \textit{finie} de $G(F)$-orbites, l'une d'elle Žtant la $G(F)$-orbite 
$\ES{O}_{F,u}$, et chacune de ces $G(F)$-orbites est ouverte et fermŽe dans $\bsfrY$. La mesure sur $\bsfrY$ dŽfinie par \ref{corollaire au thŽorme} 
donne une mesure de Radon positive $G(F)$-invariante non nulle sur $\ES{O}_{F,u}$ qui (d'aprs \ref{mesure sur l'orbite ouverte}) dŽfinit une mesure de Radon sur $G(F)$; c'est le thŽorme 
de Deligne-Ranga Rao \cite{RR}.

\vskip1mm
Supposons maintenant $p>1$. Pour les ŽlŽments sŽparables de $\UF$, resp. $\NF$, on sait d'aprs \cite[3.5.6]{L} que les points $F$-rationnels de l'orbite gŽomŽtrique sont contenus dans la $F$-strate: 
pour $u\in \UF$ sŽparable, on a l'inclusion $\ES{O}_u(F)\subset \bsfrY_{F,u}$; et pour $Y\in \NF$ sŽparable, on a l'inclusion $\ES{O}_Y(F)\subset \bsfrY_{F,Y}$. Si de plus $p$ est trs bon pour 
$G$, alors tous les ŽlŽments de $\UF$, resp. $\NF$, sont sŽparables (cf. \cite[3.5.4]{L}) et les inclusions prŽcŽdentes sont des ŽgalitŽs \cite[3.5.7]{L}. 

Avant de poursuivre, introduisons une notion intermŽdiaire entre \guill{$p$ trs bon} et \guill{$p$ bon} pour $G$, celle de groupe \guill{$F$-standard} (McNinch \cite[2.2, def.~3]{MN}). Le groupe $G$ 
est dit \textit{$F$-standard} s'il existe une $F$-isogŽnie \textit{sŽparable} (donc centrale) entre $G$ et un groupe rŽductif connexe $H$ dŽfini sur $F$ telle que:
\begin{enumerate}
\item[(S1)] le groupe dŽrivŽ $H_{\rm der}$ de $H$ soit \text{simplement connexe};
\item[(S2)] $p$ soit \text{bon} pour $H$, i.e. pour $H_{\mathrm{der}}$;
\item[(S3)] il existe une forme bilinŽaire $H$-invariante non dŽgŽnŽrŽe sur $\mathfrak{h}= {\rm Lie}(H)$.
\end{enumerate}
Observons que la condition (S3) porte sur $\mathfrak{h}$ et non sur $\mathfrak{h}_{\rm der}= {\rm Lie}(H_{\rm der})$ et qu'elle est prŽservŽe par les isogŽnies sŽparables. 
Elle assure que $\mathfrak{h}$ est une somme directe d'algbres de Lie de la forme suivante (cf. \cite[2.9]{J}): une algbre de Lie simple 
qui n'est pas de type ${\bf A}_l$; une algbre de Lie isomorphe ˆ $\mathfrak{sl}_n$ avec $(p, n)=1$; une algbre de Lie isomorphe ˆ $\mathfrak{gl}_n$; une 
algbre de Lie commutative (formŽe d'ŽlŽments centraux dans $\mathfrak{h}$). En particulier $p$ n'est pas forcŽment \textit{trs bon} pour $H$. 
NŽanmoins, la description ci-dessus assure que $\mathfrak{h}$ est muni d'une forme bilinŽaire $H$-invariante \textit{symŽtrique} non dŽgŽnŽrŽe; 
observons que tout comme (S3), cette propriŽtŽ est invariante par isogŽnie sŽparable. 

Si $G$ est $F$-standard, on sait d'aprs \cite[prop.~6]{MN} que toutes les orbites gŽomŽ\-triques nilpotentes dans 
$\mathfrak{h}$, resp. $\mathfrak{g}$, sont sŽparables. 
 
La proposition suivante, jointe ˆ \ref{mesure sur l'orbite ouverte}, donne le thŽorme de Deligne-Ranga Rao pour les orbites rationnelles. 

\begin{proposition}\label{le cas trs bon}
Supposons que $G$ soit $F$-standard. 
\begin{enumerate}
\item[(i)] Pour $u\in \UF$, la $G(F)$-orbite $\ES{O}_{F,u}$ est ouverte dans la $F$-strate $\bsfrY_{F,u}$. 
\item[(ii)] Pour $Y\in \NF$, la $G(F)$-orbite $\ES{O}_{F,Y}$ est ouverte dans la $F$-strate $\bsfrY_{F,Y}$. 
\end{enumerate}
\end{proposition}

\begin{proof} 
Commenons par le point (ii). Soient $Y\in \NF\smallsetminus \{0\}$ (pour $Y=0$ il n'y a rien ˆ dŽmontrer), $\lambda\in \Lambda_{F,Y}^{\rm opt}$ et $k = m_Y(\lambda)$. Posons $Y'=Y(k)\in \mathfrak{g}_\lambda(k;F)$. 
Puisque $\lambda$ est $(F,Y')$-optimal (d'aprs \cite[2.8.2]{L}) et $Y'$ est sŽparable, on a  
$$\mathfrak{g}^{Y'} = {\rm Lie}(G^{Y'}) \subset \mathfrak{p}_{\lambda} = \mathfrak{g}_{\lambda, 0}\ptf$$ 
Puisque $G$ est $F$-standard, on sait que:
\begin{itemize}
\item toutes les orbites gŽomŽtriques nilpotentes de $\mathfrak{g}$ sont sŽparables;
\item il existe une forme bilinŽaire $G$-invariante \textit{symŽtrique} non dŽgŽnŽrŽe sur $\mathfrak{g}$.
\end{itemize} 
On en dŽduit comme dans la preuve de \cite[lemma~5.7]{J} que $$[Y'\!, \mathfrak{g}_\lambda(i-k)]= \mathfrak{g}_\lambda(i)\quad \hbox{pour tout} \quad i\in \mbb{N}^*\ptf\leqno{(1)}$$ 
Comme $\mathfrak{p}_\lambda = \bigoplus_{i\geq k} \mathfrak{g}_\lambda(i-k)$, on obtient que $[Y'\!,\mathfrak{p}_\lambda]= \mathfrak{g}_{\lambda,k}$ puis (par approximations successives) 
que $[Y,\mathfrak{p}_\lambda]= \mathfrak{g}_{\lambda,k}$. Par consŽquent  
l'application $P_\lambda\rightarrow \mathfrak{g}_{\lambda ,k},\, p \mapsto {\rm Ad}_p(Y)$ est submersive. Elle induit une application ouverte $P_\lambda(F)\rightarrow \mathfrak{g}_{\lambda,k}(F)$ 
qui assure que la $P_\lambda(F)$-orbite de $Y$ est ouverte dans la $F$-lame $\scrY_{F,Y}$. 

Prouvons (i). Soit $u \in \UF\smallsetminus \{1\}$. On suppose que $u$ est en position standard. Soient $\lambda$ l'unique ŽlŽment 
de $\check{X}(A_0)$ qui soit $(F,u)$-optimal et $k= m_u(\lambda)$. Posons $Y=j_0^{-1}(u)\in \mathfrak{g}_{\lambda,k}(F)$. Le $F$-isomorphisme 
$j_0: \mathfrak{u}_0\rightarrow U_0(F)$ induit par restriction et passage aux quotients un $F$-isomorphisme de variŽtŽs 
$$j_\lambda(k): \mathfrak{g}_\lambda(k) \rightarrow G_\lambda(k)= G_{\lambda,k}/G_{\lambda ,k+1}\ptf$$ 
D'aprs \cite[3.4.2]{L}, l'isomorphisme $j_\lambda(k)$ est $M_\lambda$-Žquivariant et c'est un isomorphisme de groupes\footnote{Observons que si de plus $G$ est $F$-dŽployŽ, alors 
$j_\lambda(k)$ munit $G_\lambda(k)$ d'une structure de $M_\lambda$-module dŽfini sur $F$ (loc.~cit.).}. Notons $\ES{O}_{F,u}^{M_\lambda}$ la $M_\lambda(F)$-orbite de $u$. 
Compte-tenu de ce que l'on a prouvŽ pour $Y$, on obtient via $j_\lambda(k)$ que l'ensemble $\ES{O}_{F,u}^{M_\lambda}G_{\lambda,k+1}(F)$ est ouvert dans $G_{\lambda,k}(F)$. 
D'autre part gr‰ce ˆ (1) on obtient (par approximations successives) que la $U_\lambda(F)$-orbite de $u$ est Žgale ˆ $uG_{\lambda,k+1}(F)$. Cela prouve 
que la $P_\lambda(F)$-orbite de $u$ est Žgale ˆ $\ES{O}_{F,u}^{M_\lambda}G_{\lambda,k+1}(F)$, donc ouverte dans la $F$-lame $\scrY_{F,u}$. On en dŽduit que 
la $G(F)$-orbite $\ES{O}_{F,u}$ est ouverte dans la $F$-strate $\bsfrY_{F,u}$. 
\end{proof}

\begin{remark}\label{hypothse optimale}
\textup{
\begin{enumerate}
\item[(i)]
Si $p$ est \textit{bon} pour $G$, d'aprs la thŽorie de Bala-Carter (cf. \cite[4]{J}), il n'y a qu'un nombre fini d'orbites gŽomŽtriques nilpotentes dans $\mathfrak{g}$. 
Si de plus on sait que toutes les orbites gŽomŽtriques des ŽlŽments nilpotents de $\mathfrak{g}(F)$ sont sŽparables, e.g. si $p$ est \textit{trs bon}, alors d'aprs McNinch \cite[theorem~40]{MN}, il n'y a qu'un nombre fini 
de $G(F)$-orbites nilpotentes dans $\mathfrak{g}(F)$. 
\item[(ii)]L'hypothse \guill{$G$ est $F$-standard} est meilleure que \guill{$p$ est trs bon} mais elle n'est pas optimale. Par exemple $G=\textrm{PGL}_{mp}$ n'est pas $F$-standard; or dans 
ce cas les $F$-strates de $\UF$, resp. $\NF$, co\"{\i}ncident avec les $G(F)$-orbites et les points 
(i) et (ii) de \ref{le cas trs bon} sont trivialement vrais. 
\end{enumerate}
}
\end{remark}

\subsection{DŽsintŽgration de la mesure sur la $F$-lame}\label{dŽsintŽgration}On prouve dans cette sous-section que sous certaines hypothses 
moins fortes que celle de \ref{le cas o l'orbite est ouverte} (voir \ref{hypothses H}), on peut dŽsintŽgrer la mesure  $\varphi(y) \dd y$ au voisinage de $u$ dans la $F$-lame $\scrY$ et en dŽduire 
la convergence de l'intŽgrale orbitale unipotente associŽe ˆ la $G(F)$-orbite $\ES{O}_{F,u}$. Cela amŽliore les rŽsultats connus (cf. \ref{comparaison}). 

Soient $u\in \UF\smallsetminus \{1\}$, $\lambda\in \Lambda_{F,u}^{\rm opt}$ et $k=m_u(\lambda)$. On suppose comme en \ref{le rŽsultat principal} que 
$u$ est en position standard et que $\lambda\in \check{X}(A_0)$. On pose $P=P_\lambda$, $\scrY= \scrY_{F,u}$ et $\bsfrY=\bsfrY_{F,u}$. 

Pour $y\in \scrY$, si le centralisateur $G^y(F)$ de $y$ dans $G(F)$ est unimodulaire, le choix d'une mesure de Haar $dg^y$ sur $G^y(F)$ dŽfinit 
une \guill{mesure de Haar} $\dd \bar{g}_y = \frac{\dd g}{\dd g^y}$ sur l'espace quotient $G^y(F)\backslash G(F)$, \cad une fonctionnelle 
linŽaire positive non nulle et $G(F)$-invariante ˆ droite sur $C^\infty_{\rm c}(G^y(F)\backslash G(F))$. La question est: via l'homŽomorphisme 
$$G^y(G)\backslash G(F) \rightarrow \ES{O}_{F,y},\, g \mapsto g^{-1} y g\vg$$ cette fonctionnelle linŽaire dŽfinit-elle une mesure de Radon sur $G(F)$? Autrement dit, 
pour toute fonction $f\in C^\infty_{\rm c}(G(F))$, l'intŽgrale orbitale $$I_y(f)= \int_{G^y(F)\backslash G(F)} f(g^{-1} y g) \dd \bar{g}_y\leqno{(1)}$$ est-elle absolument 
convergente? 

Pour $y\in \scrY$, commenons par rappeler la notion de \guill{mesure de Haar} sur l'espace quotient $G^y(F)\backslash P(F)$. Mme si ce n'est pas nŽcessaire, on suppose toujours 
que $G^y(F)$ est unimodulaire. Notons $\ES{F}_y$ l'espace des fonctions $h\in C^\infty(P(F))$ 
telles que:
\begin{itemize}
\item $h(gp) =\bs{\delta}_{P}(g)^{-1} h(p)$ pour tout $g\in G^y(F)$ et tout $p\in P(F)$;
\item il existe un sous-ensemble compact $\Omega_h$ de $P(F)$ tel que le support de $f$ soit contenu dans $G^y(F)\Omega_h$.
\end{itemize}

\begin{remark}
\textup{Observons que pour $y\in \scrY$, puisque $\varphi(y)\neq 0$, pour tout $g\in G^y(F)$, on a $\bs{\delta}_{\lambda,(1,k)}(g)=1$ \cad $\bs{\delta}_P(g)= \bs{\delta}_{\lambda,k}(g)$.
}\end{remark}

Ë une constante $>0$ prs, il existe une unique fonctionnnelle linŽaire positive non nulle $\bs{\nu}_{y}$ sur $\ES{F}_{y}$ qui soit 
invariante par translations ˆ droite par $P(F)$. Concrtement, on choisit une mesure de Haar $\dd g^y$ sur $G^y(F)$. Toute fonction 
$h\in \ES{F}_{y}$ s'Žcrit $$h(p)= \int_{G^y(F)} \Psi(gp)\bs{\delta}_{P}(g) \dd g^y $$ pour une fonction $\Psi\in C^\infty_{\rm c}(P(F))$. Cette fonction n'est pas 
unique mais l'intŽgrale $\int_{P(F)} \Psi(p)\dd_{\rm r}p$ est bien dŽfinie (i.e. elle ne dŽpend pas du choix de $\Psi$); on la note 
$$\langle \bs{\nu}_y,h \rangle = \int_{G^y(F)\backslash P(F)} h(p)\dd \bs{\nu}_y(p)\ptf$$ 
La mesure de Haar ˆ droite $\dd_{\rm r}p$ sur $P(F)$ Žtant fixŽe, la mesure de Haar $\dd g^y$ sur $G^y(F)$ dŽtermine $\bs{\nu}_y$ et rŽciproquement. 

La question posŽe plus haut, ˆ savoir la convergence absolue de l'intŽgrale (1), est Žquivalente ˆ: via l'homŽomorphisme 
$$G^y(F)\backslash P(F) \rightarrow \ES{O}_{F,y}^P,\, p \mapsto p^{-1}yp\vg$$ la fonctionnelle linŽaire $\bs{\nu}_y$ dŽfinit-elle une mesure de Radon sur 
$P(F)$? Autrement dit, pour toute fonction $\phi\in C^\infty_{\rm c}(P(F))$, l'intŽgrale orbitale
$$I^P_y(\phi) = \int_{G^y(F)\backslash P(F)} \phi(p^{-1} y p) \bs{\delta}_P(p)^{-1} \dd \bs{\nu}_y(p)\leqno{(2)}$$ 
est-elle absolument convergente?

\begin{remark}\label{ŽgalitŽ formelle}
\textup{Rappelons que $G(F)= KP(F)=P(F)K$ et que les mesures de Haar $\dd g$, $\dd k$ et $\dd_{\rm r}p$ sont normalisŽes de telle 
manire que pour toute fonction $f\in C^\infty_{\rm c}(G(F))$, on ait l'ŽgalitŽ $\int_{G(F)}f(g)\dd g = \int_{K\times P(F)} f(kp) \dd k\dd_{\rm r}p$. Alors pour  
$f\in C^\infty_{\rm c}(G(F))$, on a l'ŽgalitŽ (au moins formellement) $$I_y(f)= I_y^P(f^K)\ptf$$}
\end{remark}

Soit $$\scrY/P(F)\;(= \bsfrY/G(F))$$ l'espace des $P(F)$-orbites dans $\scrY$, que l'on munit de la topologie quotient. 
Si $\omega$ est un ouvert de $\scrY$, l'ensemble $\bigcup_{p\in P(F)}p^{-1}\omega p$ est encore ouvert dans $\scrY$. 
Par consŽquent le morphisme quotient $$\bs{q}:\scrY \rightarrow \scrY/P(F)$$ est une application ouverte. Pour $y\in \scrY$, on note $$\ES{O}_{F,y}^P=\{p^{-1}yp\,\vert\, p\in P(F)\}$$ la $P(F)$-orbite 
de $y$, que l'on identifie ˆ l'ŽlŽment $\bs{q}(y)$ de $\scrY/P(F)$.

On veut dŽcrire le comportement des $P(F)$-orbites de $\scrY$ au voisinage de $u$, en imposant 
certaines conditions de \guill{rŽgularitŽ} aux centralisateurs $G^y(F)$ des ŽlŽments $y\in \scrY$ au voisinage de $u$. Pour cela 
il est commode d'introduire la dŽfinition suivante.

\begin{definition}\label{def slice}
\textup{Pour un sous-groupe fermŽ $\ES{Z}$ de $P(F)$, un sous-ensemble $S$ de $\scrY$ est appelŽ \textit{$\ES{Z}$-feuillet}\footnote{\textit{$\ES{Z}$-slice} en anglais.} (dans $\scrY$ relativement ˆ l'action 
de $P(F)$)
s'il vŽrifie les conditions suivantes: 
\begin{itemize}
\item $S$ est $\ES{Z}$-invariant, i.e. ${\rm Int}_{\ES{Z}}(S)= S$;
\item $\wt{S}= {\rm Int}_{P(F)}(S)$ est ouvert dans $\scrY$ et $S$ est fermŽ dans $\wt{S}$;
\item pour tout $p\in P(F)$, on a $p^{-1}Sp \cap S \neq \emptyset\Rightarrow p\in \ES{Z}$. 
\end{itemize}
La troisime condition assure que pour tout $s\in S$, le centralisateur $G^s(F)$ de $s$ dans $G(F)$, qui est un sous-groupe de $P(F)$, 
est contenu dans $\ES{Z}$. Si $y\in \scrY$, un \textit{feuillet pour $y$} (dans $\scrY$ relativement ˆ l'action de $P(F)$) est 
un $G^y(F)$-feuillet $S$ tel que: 
\begin{itemize}
\item $y\in S$;
\item $G^s(F)=G^y(F)$ pour tout $s\in S$. 
\end{itemize}
}
\end{definition}

Si $S$ est un $\ES{Z}$-feuillet pour un sous-groupe fermŽ $\ES{Z}$ de $P(F)$, alors $S$ est localement fermŽ dans $\scrY$. On le munit de la topologie 
induite par celle de $\scrY$, ce qui en fait un \textit{td-espace} au sens de Bernstein-Zelevinski \cite{BZ}, i.e. Hausdorff, localement compact et totalement discontinu. Soit $P(F)\times^{\ES{Z}}S$ le quotient $(P(F)\times S)/\ES{Z}$ pour l'action ˆ droite 
de $\ES{Z}$ donnŽe par $(p,s)\cdot z= (pz, z^{-1}sz)$. On munit $P(F)\times S$ de la topologie produit et $P(F)\times^{\ES{Z}}S$ de la topologie quotient. Alors l'application 
$P(F)\times S \rightarrow \wt{S},\, (p,s)\mapsto psp^{-1}$ se quotiente en un homŽomorphisme 
$$\alpha:P(F)\times^{\ES{Z}} S \buildrel \simeq\over{\longrightarrow} \wt{S}\ptf$$ 
Observons que pour $y\in \wt{S}$, il existe un $p\in P(F)$ tel que $pyp^{-1}\in S$ et le centralisateur $G^y(F)$ de $y$ dans $G(F)$ vŽrifie 
$\mathrm{Int}_p(G^y(F))= G^{pyp^{-1}}(F)\subset \ES{Z}$. 

Supposons de plus que pour tout $s\in S$, l'inclusion $G^s(F)\subset  \ES{Z}$ soit une ŽgalitŽ, \cad que $S$ soit un feuillet pour un (i.e. pour tout) $s\in S$. Alors 
$$P(F)\times^{\ES{Z}}S=(P(F)/\ES{Z}) \times S\ptf$$ La projection sur le second facteur $P(F)\times S \rightarrow S$ se quotiente en 
une application continue (surjective) $\beta: P(F)\times^{\ES{Z}} S \rightarrow S$ et l'application composŽe $\beta\circ \alpha^{-1}: \wt{S} \rightarrow S$ 
se quotiente en une bijection continue $$\bs{s}:\wt{S}/P(F) \rightarrow S\ptf$$ Pour tout $y\in \wt{S}$, on a $S\cap \ES{O}_{F,y}^P= \{\beta \circ \alpha^{-1}(y)\}$. 
Puisque $\bs{s}$ est une section continue du morphisme quotient 
$\bs{q}\vert_{\wt{S}}: \wt{S}\rightarrow \wt{S}/P(F)$, l'application $\bs{q}\vert_S: S  \rightarrow  \wt{S}/P(F)$ est 
un homŽomorphisme et $\bs{q}(S)= \wt{S}/P(F)$ est un td-espace. En particulier pour tout $y\in \wt{S}$, la $P(F)$-orbite $\ES{O}_{F,y}^P$ 
est fermŽe dans $\wt{S}$. Observons aussi que pour tout sous-ensemble $\omega$ de $\wt{S}$, on a $\bs{q}(\omega)= \bs{q}(S \cap \omega)$. On en dŽduit que le morphisme 
quotient $$\bs{q}\vert_{\wt{S}}: \wt{S} \rightarrow \bs{q}(\wt{S})= \wt{S}/P(F)$$ est une application (ouverte et) fermŽe. 

\vskip1mm
ConsidŽrons maintenant les hypothses suivantes. 

\begin{hypoths}\label{hypothses H}
\textup{
\begin{enumerate}
\item[(H1)] Le centralisateur $G^u(F)$ de $u$ dans $G(F)$ est unimodulaire.
\item[(H2)] Il existe un feuillet $S$ pour $u$ (dans $\scrY$ relativement ˆ l'action de $P(F)$).
\end{enumerate}
}
\end{hypoths}

\begin{remark}\label{sur les hypothses H}
\textup{
\begin{enumerate}
\item[(i)] L'hypothse (H2) assure que pour tout $y\in \wt{S}= {\rm Int}_{P(F)}(S)$, la $P(F)$-orbite $\ES{O}_{F,y}^P$ est fermŽe dans l'ouvert $\wt{S}$ de $\scrY$; ou, ce qui revient au mme d'aprs \ref{le cas o l'orbite est ouverte}\,(1), 
la $G(F)$-orbite $\ES{O}_{F,y}$ est fermŽe 
dans l'ouvert ${\rm Int}_{K}(\wt{S})$ de $\bsfrY$. En particulier la $P(F)$-orbite $\ES{O}_{F,y}^P$ est localement fermŽe dans $\scrY$\footnote{Ce que l'on savait dŽjˆ d'aprs 
Bernstein-Zelevinski \cite[Appendix]{BZ}: $\scrY_{F,u}= \scrY_u(F)$ et l'action de $P(F)$ sur $\scrY_{F,u}$ provient d'une action 
algŽbrique (dŽfinie sur $F$) de $P=P_u$ sur la variŽtŽ $\scrY_u$.} et la $G(F)$-orbite $\ES{O}_{F,y}$ 
est localement fermŽe dans $\bsfrY$. 
\item[(ii)] Si la $P(F)$-orbite $\ES{O}_{F,u}^P$ de $u$ est ouverte dans $P(F)$, alors on sait d'aprs \ref{mesure sur l'orbite ouverte}\,(i) que le centralisateur 
$G^u(F)$ est unimodulaire; d'autre part $S=\{u\}$ est un feuillet pour $u$. 
\end{enumerate}}
\end{remark}

On suppose jusqu'ˆ la fin de \ref{dŽsintŽgration} que les hypothses \ref{hypothses H} sont vŽrifiŽes. D'aprs \ref{sur les hypothses H}\,(ii), il s'agit bien  
d'une gŽnŽralisation de \ref{le cas o l'orbite est ouverte}. 

On dŽfinit comme suit, pour chaque $y\in \wt{S}= {\rm Int}_{P(F)}(S)$, une mesure de Haar $dg^y$ sur $G^y(F)$ et donc une fonctionnelle linŽaire $\bs{\nu}_y$ sur $\ES{F}_y$. 
On fixe une mesure de Haar $dg^u$ sur $G^u(F)$ et pour $s\in S$, on prend $dg^s= dg^u$. Pour $y=p^{-1}sp$ avec $s\in S$ et $p\in P(F)$, 
on prend pour $\dd g^{y}$ la mesure dŽduite de $\dd g^s= \dd g^u$ via l'homŽomorphisme ${\rm Int}_{p^{-1}}: G^s(F) \rightarrow G^{y}(F)$. 
Alors $\bs{\nu}_{y}$ se dŽduit de 
$\bs{\delta}_P(p)^{-1}\bs{\nu}_s$ via l'isomorphisme $$\ES{F}_{s} \rightarrow \ES{F}_{y},\, h\mapsto h \circ {\rm Int}_{p}  \ptf$$ 
En d'autres termes, on a 
$$\langle \bs{\nu}_{y}, h\circ {\rm Int}_p \rangle = \bs{\delta}_P(p)^{-1} \langle \bs{\nu}_s,  h \rangle \quad \hbox{pour tout} \quad h\in \ES{F}_{y}\ptf$$ 
La fonctionnelle linŽaire $\bs{\nu}_y$ sur $\ES{F}_y$ est bien dŽfinie car si $y=p'^{-1}s' p'$ avec $s'\in S$ et $p'\in P(F)$, alors 
$s'=s$ et $pp'^{-1} \in G^s(F)$; or $\langle \bs{\nu}_s, h\circ \mathrm{Int}_{pp'^{-1}}\rangle = \bs{\delta}_P(p'p^{-1})\langle \bs{\nu}_s, h\rangle$ . 
Pour $\phi \in C^\infty_{\mathrm{c}}(G(F))$, on a donc les ŽgalitŽs (au moins formellement) 
$$I_{p^{-1}sp}^P(\phi) = \bs{\delta}_P(p)^{-1} I_s^P({^p\phi})= I_s^P(\phi)\ptf$$
On en dŽduit que pour tout $(y,p)\in \wt{S}\times P(F)$ et toute fonction $\phi\in C^\infty_{\mathrm{c}}(G(F))$, on a 
l'ŽgalitŽ (formelle) $$I_{p^{-1}yp}^P(\phi) = I_y^P(\phi)\ptf\leqno{(3)}$$

\vskip1mm
Pour $y\in \wt{S}$, puisque la $P(F)$-orbite $\ES{O}_{F,u}^P$ est fermŽe dans $\wt{S}$, pour toute fonction $\psi\in C^\infty_{\rm c}(\wt{S})$, l'intŽgrale (2) 
$$I_y^P(\psi)= \int_{G^y(F)\backslash P(F)} \psi(p^{-1}yp)\bs{\delta}_P(p)^{-1} \dd\bs{\nu}_y(p)$$ est absolument convergente. 
L'hypothse (H2) assure que $\bs{q}(S)= \wt{S}/P(F)$ est un td-espace. On peut donc introduire l'espace  
$C^\infty_{\mathrm{c}}(\wt{S}/P(F))$ des fonctions localement constantes et ˆ support compact sur $\wt{S}/P(F)$. 

\begin{lemma}
Pour $\psi\in C^\infty_{\rm c}(\wt{S})$, la fonction $y\mapsto I_y^P(\psi)$ est dans $C^\infty_{\mathrm{c}}(\wt{S}/P(F))$. 
\end{lemma}

\begin{proof}
Soit $\psi\in C^\infty_{\rm c}(\wt{S})$. D'aprs l'ŽgalitŽ (3), la fonction $y\mapsto I_y^P(\psi)$ sur $\wt{S}$ se factorise par $\wt{S}/P(F)$; et 
elle est ˆ support le compact ouvert $\bs{q}(\mathrm{Supp}(\psi))$ de $\wt{S}/P(F)$ (rappellons que l'application continue $\bs{q}\vert_{\wt{S}}: \wt{S} \rightarrow \wt{S}/P(F)$ est ouverte et fermŽe). 
Il reste ˆ vŽrifier que la fonction 
$y\mapsto I_y^P(\psi)$ est lisse sur $\wt{S}/P(F)$. Pour cela reprenons l'isomorphisme $\wt{\mathcal{H}}^P \rightarrow C^\infty_{\rm c}(\scrY),\, \Phi \mapsto \psi_\Phi$ 
du dŽbut de la preuve de \ref{rŽponse aux questions}. Il induit par restriction un isomorphisme du sous-espace $\wt{\mathcal{H}}^{P,*}\subset \wt{\mathcal{H}}^P$ formŽ des fonctions 
ˆ support dans $P(F)\times \wt{S}$ sur le sous-espace $C^\infty_{\rm c}(\wt{S})\subset C^\infty_{\rm c}(\scrY)$. \'Ecrivons $\psi= \psi_\Phi$ avec $\Phi\in \wt{\mathcal{H}}^{P,*}$. Pour 
$y\in \wt{S}$, on a 
\begin{eqnarray*}
I_y^P(\psi)&=& \int_{G^y(F)\backslash P(F)} \varphi(p^{-1}yp)^{-1}\Phi(1,p^{-1}yp)\bs{\delta}_P(p)^{-1}\dd \bs{\nu}_y(p)\\
&=& \varphi(y)^{-1} \int_{G^y(F)\backslash P(F)} \Phi(p^{-1}\!,y) \bs{\delta}_P(p)^{-1} \dd \bs{\nu}_y(p)\ptf
\end{eqnarray*}
La fonction $y \mapsto \int_{G^y(F)\backslash P(F)} \Phi(p^{-1}\!,y) \bs{\delta}_P(p)^{-1} \dd \bs{\nu}_y(p)$ est localement constante sur $S$. 
Puisque la fonction $y\mapsto \varphi(y)^{-1}$ est localement constante sur $\scrY$, elle l'est \textit{a fortiori} sur $S$. Par consŽquent la fonction 
$y \mapsto I_y^P(\psi)$ est localement constante sur $S$, donc sur $\wt{S}$ puisqu'elle est $P(F)$-invariante. 
\end{proof}

\begin{lemma}
\begin{enumerate}
\item[(i)] L'application linŽaire $$C^\infty_{\rm c}(\wt{S})\rightarrow C^\infty_{\mathrm{c}}(\wt{S}/P(F)),\, \psi \mapsto (y\mapsto I_y^P(\psi))$$ est surjective.
\item[(ii)] Si $I_y^P(\psi)= 0$ pour tout $y\in \wt{S}$, alors $$\int_{\wt{S}}\varphi(y)\psi(y) \dd y=0\ptf$$
\end{enumerate}
\end{lemma}

\begin{proof}
Il suffit d'adapter celle du lemme \ref{surjectivitŽ f donne Phi}. Posons $\mathcal{C}= C^\infty_{\rm c}(\wt{S})$ et $\ES{F}=C^\infty_{\mathrm{c}}(\wt{S}/P(F))$. 
Soit $\Omega$ un sous-ensemble ouvert compact de $\wt{S}$. 
Posons $$\wt{\Omega}= \mathrm{Int}_{P(F)}(\Omega)= \bigcup_{p\in P(F)} p^{-1} \Omega p\ptf$$ Soit $\mathcal{C}_\Omega$ le sous-espace de $\mathcal{C}$ formŽ des fonctions ˆ support dans $\wt{\Omega}$ 
et qui sont constantes sur les sous-ensembles $p^{-1}\Omega p$ pour tout $p\in P(F)$. Soit $\ES{F}_{\bs{q}(\Omega)}$ le sous-espace de $\ES{F}$ formŽ des fonctions ˆ support dans $\bs{q}(\Omega)= \wt{\Omega}/P(F)$ et 
qui sont constantes sur $\bs{q}(\Omega)$. On a $\dim_{\mbb{C}}(\ES{F}_{\bs{q}(\Omega)})=1$. 
L'application $\psi \mapsto (y\mapsto I_y^P(\psi))$ envoie $\mathcal{C}_\Omega$ dans $\ES{F}_{\bs{q}(\Omega)}$ et il 
suffit de vŽrifier les assertions du lemme pour ces deux espaces. 

Prouvons (i). Soit $\xi\in \ES{F}_{\bs{q}(\Omega)}$. Notons $\kappa$ l'ŽlŽment de $\ES{F}_{\bs{q}(\Omega)}$ dŽfini par 
$$\kappa(y) = I_y^P({\bf 1}_\Omega) >0\quad \hbox{pour tout}\quad y\in \wt{\Omega}\ptf$$ 
Soit $\psi\in \mathcal{C}_\Omega$ la fonction dŽfinie par $\psi = \kappa^{-1}\cdot {\bf 1}_\Omega \cdot \xi$. 
Elle vŽrifie $$\psi(p^{-1}y p) = \kappa(y)^{-1} {\bf 1}_\Omega(p^{-1} y p) \xi(y)\quad \hbox{pour tout}\quad (y,p)\in \wt{\Omega}\times P(F)\ptf$$
On en dŽduit que $I_y^P(\psi) =\xi(y)$ pour tout $y\in \wt{\Omega}$. 

Prouvons (ii). Tout ŽlŽment de $\mathcal{C}_\Omega$ est combinaison linŽaire (finie) de fonctions 
${^p({\bf 1}_\Omega)}= {\bf 1}_\Omega \circ {\rm Int}_{p^{-1}}$ avec $p\in P(F)$. On en dŽduit que les fonctionnelles linŽaires $T$ sur $\mathcal{C}_\Omega$ 
telles que $T({^p\psi}) = \bs{\delta}_P(p)T(\psi)$ pour tout $\psi\in \mathcal{C}_\Omega$ et 
tout $p\in P(F)$, forment un espace vectoriel de dimension $1$. Pour $y'\in \wt{\Omega}$, les fonctionnelles linŽaires 
$$T(\psi)= \int_{\scrY}\varphi(y)\psi(y) \dd y\quad \hbox{et} \quad T'(\psi)= I_{y'}^P(\psi)$$ 
sont dans cet espace; elles sont donc proportionnelles. Cela implique (ii).
\end{proof}

Soit $\bs{\eta}_{S}$ la fonctionnelle linŽaire sur $C^\infty_{\mathrm{c}}(\wt{S}/P(F))$ dŽfinie par 
$$\langle \bs{\eta}_{S}, \xi\rangle = \int_{\wt{S}}\varphi(y) \psi(y)  \dd y $$ 
pour toute fonction $\psi\in C^\infty_{\rm c}(\scrY)$ telle que $\xi(y)= I_y^P(\psi)$. On Žcrit aussi 
$$\langle \bs{\eta}_{S}, \xi\rangle  = \int_{\wt{S}/P(F)} \xi(y) \dd \bs{\eta}_{S}(y)\ptf$$ 

\begin{proposition}\label{le cas slice}
On suppose que les hypothses \ref{hypothses H} sont vŽrifiŽes. Soit $S$ un feuillet pour $u$ (dans $\scrY$ relativement ˆ l'action de $P(F)$) et 
soit $\wt{S}= \mathrm{Int}_{P(F)}(S)$. 
\begin{enumerate}
\item[(i)] Pour toute fonction $f\in C^\infty_{\rm c}(G(F))$ et tout $y\in \wt{S}$, l'intŽgrale orbitale $$I_y(f)\;(= I_y^P(f^K))$$ 
est absolument convergente.
\item[(ii)] Pour toute fonction $f\in C^\infty_{\rm c}(G(F))$, on a l'ŽgalitŽ $$\int_{\wt{S}/P(F)}\varphi(y) I_y^P(f^K)  \dd \bs{\eta}_{S}(y) = \int_{\wt{S}}\varphi(y)f^K(y) \dd y\vg$$ les deux intŽgrales 
Žtant absolument convergente. 
\end{enumerate}
\end{proposition}

\begin{proof}Pour toute fonction $f\in C^\infty_{\rm c}(G(F))$ et tout sous-ensemble ouvert compact $\omega$ de $\wt{S}$, 
l'intŽgrale orbitale $I_y^P({\bf 1}_\omega \cdot f^K\vert_{\wt{S}})$ est bien dŽfinie et elle est absolument convergente; 
et d'aprs la dŽfinition de la fonctionnelle linŽaire $\bs{\eta}_{S}$, on a 
$$ \int_{\wt{S}/P(F)}\varphi(y) I_y^P({\bf 1}_\omega\cdot f^K\vert_{\wt{S}}) \dd \bs{\eta}_{S}(y) = \int_{\omega}\varphi(y) f^K(y)\dd y\ptf$$ 
Posons $\bs{\mathfrak{S}} = \mathrm{Int}_K(\wt{S})= \mathrm{Int}_{G(F)}(S)$; c'est un ouvert $G(F)$-invariant de la $F$-strate $\bsfrY$. 
Puisque l'intŽgrale $$I_{\bs{\mathfrak{S}}}(f)= \int_{\wt{S}}\varphi(y) f^K(y)\dd y $$ est absolument convergente (cf. \ref{le cas o l'orbite est ouverte}), d'aprs le thŽorme de convergence dominŽe, l'intŽgrale 
$$ \int_{\wt{S}/P(F)}\varphi(y) I_y^P(f^K) \dd \bs{\eta}_{S}(y)$$ l'est aussi et elle co\"{\i}ncide avec $I_{\bs{\mathfrak{S}}}(f)$; 
cela prouve (ii). On en dŽduit que pour toute fonction $f\in C^\infty_{\rm c}(G(F))$ avec $f\geq 0$, on a 
$I_y^P(f^K) <+\infty$, ce qui prouve (i).  
\end{proof}

\begin{remark}
\textup{
On a une proposition analogue pour les $G(F)$-orbites dans une $F$-strate de $\NF$, sous les mmes hypothses que \ref{hypothses H}. 
}\end{remark}

\begin{remark}
\textup{
\begin{enumerate}
\item[(i)]Le cas o l'orbite $\ES{O}_{F,u}$ est ouverte dans $\bsfrY$ (\ref{mesure sur l'orbite ouverte}) est un cas particulier de \ref{le cas slice}: il suffit en effet de prendre $S= \{u\}$ 
(cf. \ref{sur les hypothses H}\,(ii)). 
\item[(ii)] Pour dŽsintŽgrer la mesure globale $\varphi(y)\dd y$ sur $\scrY$, il suffit qu'il existe une famille $(\Omega_i)_{i\in I}$ (forcŽment dŽnombrable) d'ouverts $P(F)$-invariants 
$\Omega_i$ de $\scrY$ deux-ˆ-deux disjoints telle que: 
\begin{itemize}
\item L'ouvert $\scrY^*= \bigcup_{i\in I}\Omega_i$ de $\scrY$ est dense dans $\scrY$ (pour $\mathrm{Top}_F$);
\item Pour chaque $i\in I$, les centralisateurs $G^y(F)$ des ŽlŽments $y\in \Omega_i$ sont unimodulaires et conjuguŽs dans $P(F)$;
\item Pour chaque $i\in I$ et pour un (i.e. pour tout) $u_i\in \Omega_i$, il existe un feuillet $S_i$ pour $u_i$ tel que $\wt{S}_i = \Omega_i$.
\end{itemize}
Alors pour toute fonction $f\in C^\infty_{\rm c}(F(F))$, on a 
$$I_{\bsfrY}(f)= \sum_{i\in I} I_{\mathrm{Int}_K(\Omega_i)}(f)= \int_{\Omega_i /P(F)}\varphi(y) I_y^P(f^K)  \dd \bs{\eta}_{S_i}(y)\ptf$$ 
\end{enumerate}
}
\end{remark}

\subsection{Une remarque sur les corps globaux}\label{le cas global}
Dans cette sous-section, on suppose que $F$ est un corps global, \cad un corps de nombres  ($p=1$) ou un corps de fonctions ($p>1$). 
Soit $\mbb{A}=\mbb{A}_F$ l'anneau des adles de $F$. 
On note $\vert\,\vert_{\mbb{A}}$ la valeur absolue sur $\mbb{A}$ dŽfinie par 
$$\vert x \vert_{\mbb{A}}= \prod_v \vert x_v\vert_{F_v} \quad \hbox{pour}\quad x = \prod_v x_v \in \mbb{A}$$ o $v$ parcourt les places 
de $F$ et $x_v\in F_v$ (le complŽtŽ de $F$ en $v$). 

Soit $u\in \UF\smallsetminus \{1\}$. On pose $\scrY= \scrY_{F,u}$ et $\scrX= \scrX_{F,u}$. 
Soient aussi $\lambda \in \Lambda_{F,u}^{\rm opt}$ et $k= m_u(\lambda)$. Alors $\scrX= G_{\lambda,k}(F)$ et on pose 
$\overline{\scrX}= G_{\lambda,k}(F)/G_{\lambda ,k+1}(F)$. 
On pose aussi $$\scrX_{\mbb{A}}= G_{\lambda,k}(\mbb{A}) \quad \hbox{et}\quad \overline{\scrX}_{\!\!\mbb{A}}= G_{\lambda,k}(\mbb{A})/G_{\lambda,k+1}(\mbb{A})\ptf$$ 

Reprenons les constructions de \ref{ŽnoncŽ} en remplaant le corps de base $F$ par l'anneau $\mbb{A}$. 
Pour $x\in \scrX_{\mbb{A}}$ et $i\in \mbb{Z} $, le $\mbb{A}$-endomorphisme ${\rm Ad}_x - {\rm Id}$ de $\mathfrak{g}(\mbb{A})$ 
induit par restriction et passage aux quotients un morphisme $\mbb{A}$-linŽaire 
$$\eta_{\lambda,x}(i\mathpvg \mbb{A}):\mathfrak{g}_\lambda(i\mathpvg\mbb{A}) \rightarrow \mathfrak{g}_\lambda(k+i\mathpvg \mbb{A})$$ 
qui ne dŽpend que l'image $\bar{x}$ de $x$ dans $\overline{\scrX}_{\!\!\mbb{A}}$. Pour $i=1,\ldots ,k-1$, fixons une $F$-base $(H_{-i,j})_{1\leq j \leq d_i}$ de $\mathfrak{g}_\lambda(-i;F)$ et une $F$-base 
$(H_{k-i,j})_{1\leq j \leq d_i}$ de $\mathfrak{g}_\lambda(k-i;F)$; rappelons que $d_i= \mathrm{dim}_F(\mathfrak{g}_\lambda(i\mathpvg F))$ et que l'on sait que $d_{-i}\;(=d_i)= d_{k-i}$. Pour 
$\bar{x}\in \overline{\scrX}_{\!\!\mbb{A}}$, notons $A_{-i}(\bar{x})\in M(d_i,\mbb{A})$ la matrice de $\eta_{\lambda,\bar{x}}(-i\mathpvg\mbb{A})$ relativement 
ˆ ces bases et posons $$\varphi_{-i}(\bar{x})= \vert {\rm det}_{\mbb{A}}(A_{-i}(\bar{x}))\vert_{\mbb{A}}^{1/2}\ptf$$ Enfin posons 
$$\varphi(\bar{x})= \varphi_{-1}(\bar{x})\cdots \varphi_{1-k}(\bar{x})\geq 0\ptf$$ 
On relve $\varphi$ en une fonction sur $\scrX_{\mbb{A}}$. On a clairement: 

\begin{lemma}
L'application $\varphi: \scrX_{\mbb{A}}\rightarrow \mbb{R}_+$ ainsi dŽfinie 
ne dŽpend pas du choix des $F$-bases $(H_{-i,j})$ et $(H_{k-i,j})$ pour $i=1,\ldots ,k-1$.
\end{lemma} 

On note $\scrY_{\mbb{A}}$ la \textit{$\mbb{A}$-lame} dans $\UU_{\mbb{A}}= (\prod_v \UU_{F_v})\cap G(\mbb{A})$ dŽfinie par 
$$\scrY_{\mbb{A}} = \left(\prod_v \scrY_v \right) \cap G(\mbb{A})\quad \hbox{avec} \quad \scrY_v = \scrY_{F_v,u}$$ o $v$ parcourt les places de $F$ (voir aussi ˆ \cite[Annexe~C]{L}). Pour chaque $v$, 
puisque l'extension $F_v/F$ est sŽparable de degrŽ de transcendance infini, 
d'aprs \ref{descente sŽparable}, le co-caractre $\lambda$ est $(F_v,u)$-optimal (pour l'action par conjugaison de $G(F_v)$) et on a l'ŽgalitŽ $\scrY= \scrY_v \cap G(F)$. 
En d'autres termes la $F$-lame $\scrY$ se plonge diagonalement dans la $\mbb{A}$-lame $\scrY_{\mbb{A}}$. On a 
$$\varphi(y) >0\quad\hbox{pour tout}\quad y\in \scrY_{\mbb{A}}$$ et 
$$\varphi(y)=1\quad\hbox{pour tout}\quad y\in \scrY\ptf$$ 

Cette fonction $\varphi>0$ sur $\scrY_{\mbb{A}}$ jouera un r™le essentiel dans la dŽcomposition en produit de distributions locales 
de la contribution ˆ la formule des traces pour $G(\mbb{A})$ associŽe en \cite{L} ˆ la $F$-strate $\bsfrY$ de $\UU_F$.


\end{document}